\documentclass{amsart}
\usepackage{amsxtra, microtype}
\usepackage{stmaryrd }
\usepackage{mathrsfs}
\usepackage[margin=1.45in]{geometry}
\usepackage{amssymb,amsmath,amsthm,stmaryrd,latexsym,wasysym}
\usepackage[all]{xy}
\usepackage{hyperref}
\usepackage{mathtools}
\usepackage{tikz}
\usepackage{tikz-cd}
\usetikzlibrary{decorations.pathmorphing}
\usetikzlibrary{arrows}

\usepackage{scalerel}

\usepackage{stackrel}

\usepackage{cancel}
\usepackage{appendix}
\theoremstyle{plain}
\usepackage{xcolor}
\usepackage{xspace}
\newcommand{\dvplus}[2]{
\stackrel[{\scriptscriptstyle #2}]{{\scriptscriptstyle #1}}{+}
}

\newtheorem{theorem}{Theorem}[section]
\newtheorem{lemma}[theorem]{Lemma}
\newtheorem{proposition}[theorem]{Proposition}
\newtheorem{corollary}[theorem]{Corollary}
\newtheorem{definition}[theorem]{Definition} \theoremstyle{definition}
\newtheorem{example}[theorem]{Example}
\newtheorem{remark}[theorem]{Remark}
\newcommand{\lie}[1]{\mathfrak{#1}}

\allowdisplaybreaks

\newcommand{\R}{\mathbb{R}} 
 
\newcommand{\inv}{^{-1}}
\newcommand{\N}{\mathbb{N}}

\newcommand{\an}[1]{\arrowvert_{#1}}

\newcommand{\Man}{{\mathbf{Man}^\infty}}
\newcommand{\nset}{{\underline{n}}}
\newcommand{\lset}{{\underline{l}}}
\newcommand{\F}{\mathbb F}
\newcommand{\nvbzero}[2]{
\mathbf{0}^{#1}_{#2}
}
\newcommand{\E}{\mathbb E}

\newcommand{\A}{\mathcal A}
\renewcommand{\S}{\mathcal S}
\renewcommand{\P}{\mathbb P}

\DeclareMathOperator{\pr}{pr}
\DeclareMathOperator{\Hom}{Hom} 

\DeclareMathOperator{\Obj}{Obj}

\DeclareMathOperator{\sgn}{sgn}
\DeclareMathOperator {\id}    {id}

\allowdisplaybreaks

\begin{document}
%%%%%%%%%%%%%%%%%%%%%%%%%%%%%%%%%%%%%%%%%%%%%%%%%%%%%%%%%%%%%%%%%%%%%%%%%%%
%%%%%%%%%%%%%%%%%%%%%%    Title    %%%%%%%%%%%%%%%%%%%%%%%%%%%%%%%%%%%%%%%%
\title{A geometrisation of $\N$-manifolds}

%%% author one information

\author{M.~Heuer} \address{}  \email{malteheuer@hotmail.de}
\author{M.~Jotz} \address{Institut f\"ur Mathematik,
Julius-Maximilians-Universit\"at W\"urzburg, Germany}
\email{madeleine.jotz-lean@mathematik.uni-wuerzburg.de}

\begin{abstract}
 This paper proposes a \emph{geometrisation} of $\N$-manifolds of degree $n$ as $n$-fold vector bundles equipped with a (signed) $S_n$-symmetry.
 More precisely, it proves an equivalence between the categories of $[n]$-manifolds and the category of symmetric $n$-fold vector bundles, by finding that symmetric $n$-fold vector bundle cocycles and $[n]$-manifold cocycles are identical.
 
 This extends the already known equivalences of $[1]$-manifolds with vector bundles, and of $[2]$-manifolds with involutive double vector bundles, where the involution is understood as an $S_2$-action.
\end{abstract}
\maketitle

\tableofcontents

\section{Introduction}

An $\N$-graded manifold over a smooth manifold $M$ is a sheaf of
$\N$-graded, graded commutative, associative, unital
$C^\infty(M)$-algebras over $M$, that is locally freely generated by
finitely many elements of strictly positive degree.  $\N$-manifolds of
degree $1$ are easily seen to be 
the exterior algebras of sections of smooth vector bundles. 
$\N$-manifolds of degree $2$, called for short $[2]$-manifolds, have recently been 
geometrised to double vector bundles with a linear \emph{indirect} involution \cite{Jotz18b}.
(Such double vector
bundles were called ``symmetric double vector bundles with inverse
symmetry'' by Pradines \cite{Pradines77}.)
Before that, Lie $2$-algebroids i.e.~$[2]$-manifolds with a \emph{homological vector
field}, were linked to VB-Courant algebroids by
Li-Bland in his thesis \cite{Li-Bland12}, see also \cite{Jotz19b}, building up on the
correspondence of Courant algebroids with symplectic Lie 2-algebroids
\cite{Roytenberg02, Severa05}.

Generally, positively graded manifolds are the geometric objects underlying Lie $n$-algebroids, which are widely accepted to infinitesimally describe Lie $n$-groupoids.
A Lie $n$-algebroid is an $[n]$-manifold equipped with a  homological vector field, i.e.~a vector field of degree $1$ that squares to zero.   Lie $n$-algebroids appeared in the early
00's in Voronov's study of Lie bialgebroids \cite{Voronov02} and in
Roytenberg's supergeometric approach to Courant algebroids
\cite{Roytenberg02}, see also \cite{Severa05}. Precisely, Courant algebroids are equivalent to Lie $2$-algebroids equipped with a compatible symplectic structure of degree $2$.
This is at the origin of the interest for Lie $n$-algebroids in the Poisson community, as this fact leads to a path towards the \emph{integration} of Courant algebroids -- to Lie $2$-groupoids with an additional geometric structure. The correspondence of symplectic Lie $2$-algebroids with Courant algebroids fits in fact in the more general equivalence of Lie $2$-algebroids with VB-Courant algebroids \cite{Li-Bland12,Jotz19b}, which, in turn, is based on the underlying equivalence between double vector bundles equipped with a linear involution, and positively graded manifolds generated in the degrees $1$ and $2$ \cite{Jotz18b}.

Note that involutive double vector bundles dualise to metric double vector bundles, and the two classes of objects are therefore equivalent.
Understanding the equivalence between metric or involutive double vector bundles and $[2]$-manifolds provides a precise dictionary, as summarised below, for linear geometric structures on involutive double vector bundles versus compatible geometric structures on $[2]$-manifolds. 
\begin{center}\label{table_question}
\begin{tabular}{ |c|c|} 
\hline
Double vector bundles & Degree $2$ graded geometry \\
 \hline
 metric double vector bundles & $[2]$-manifolds \\ 
 metric VB-algebroids  & Poisson $[2]$-manifolds  \\ 
 VB-Courant algebroids &  Lie $2$-algebroids \\
 LA-Courant algebroids & Poisson Lie $2$-algebroids\\
 tangent doubles of Courant algebroids & symplectic Lie $2$-algebroids \\
 \hline
\end{tabular}
\end{center}
In particular, several new explicit examples of Lie $2$-algebroids, of Poisson $[2]$-manifolds and of Poisson Lie $2$-algebroids arise as the counterparts of known examples of VB-Courant algebroids, of metric VB-algebroids and of LA-Courant algebroids \cite{Jotz18b, Jotz19b, Jotz20}.

\medskip

The involution of an involutive double vector bundle can be understood as an $S_2$-action, while a classical vector bundle, which is commonly known as the geometrisation of a $[1]$-manifold, is trivially acted upon by the trivial group $S_1$.
The goal of this paper, which can be considered a sequel of \cite{Jotz18b}, 
is to establish an equivalence between the category of \emph{$S_n$-symmetric $n$-fold vector bundles} and the category of positively graded manifolds 
of degree $n$.
In other words, this paper geometrises $[n]$-manifolds  via $n$-fold vector bundles with a compatible $S_n$-action.
This is the groundwork needed for the geometrisation of Lie $n$-algebroids, or of Lie $n$-algebroids with additional compatible 
geometric structures such as symplectic or Poisson structures.

The approach in \cite{Jotz18b} is a classical one; an extension of the
construction of vector bundles over a manifold $M$ from free and
locally finitely generated sheaves of $C^\infty(M)$-modules, using the
double vector bundle charts in \cite{Pradines77}. 
Here also, the key to the geometrisation of $[n]$-manifolds is a precise understanding of the atlases of symmetric $n$-fold vector bundles, which, in turn, only exist
once it has been proved that symmetric $n$-fold vector bundles can always be equivariantly decomposed.
The authors prove in \cite{HeJo20} that $n$-fold vector bundles always admit a linear decomposition, and this paper proves that a \emph{symmetric} $n$-fold vector bundle admits a \emph{symmetric} linear decomposition, and so a symmetric $n$-fold vector bundle atlas.
Unlike in the case $n=2$, adding the symmetry to the picture requires a thorough reconsideration of the proof of the decomposition in the general case. This, together with a little gap in \cite{HeJo20}, 
is the reason why some parts of \cite{HeJo20} are considerably revisited here.

However, it turns out that the equivalence of $[n]$-manifolds with symmetric $n$-fold vector bundles is obtained in a more straightforward manner by studying cocycles for both types of geometric structures. This paper concentrates hence on showing that $[n]$-manifold cocycles and symmetric $n$-fold vector bundle cocycles are basically the same objects, in particular with same cocycle conditions.

Along the way, an explicit formula for morphisms of split $[n]$-manifolds and for the composition of two such morphisms is given. The latter is heavily inferred by the correspondence 
of split $[n]$-manifolds with decomposed symmetric $n$-fold vector bundles, and, as far as the authors know, the first explicit formula avoiding Koszul signs, or more precisely computing them explicitly.

This paper does not describe how the graded functions of an $[n]$-manifold correspond to special functions on a symmetric $n$-fold vector bundle. This is part of a further project in progress. Also, the general equivalence of $\mathbb N$-manifolds (or arbitrary, even infinite, degree) with symmetric multiple vector bundles is easily deduced from the main result of this paper, but will be carried out elsewhere.

\medskip

The differentiation of Lie $n$-groupoids to Lie $n$-algebroids is considered folklore in the research community, and goes back to 
Severa \cite{Severa06}.
 However, a precise differentiation process carrying out all the details, in particular the multiple Lie brackets -- or in other words the homological vector field --  has not been published yet. Recently, Kadiyan and Blohmann, as well as independently Du, Fernandes, Ryvkin, and Zhu
have announced a differentiation process, that involves simplicial structures closely related to the $n$-cube categories used here, and to symmetric $n$-fold vector bundles.
This paper hence builds the geometric fundaments for the infinitesimal description of Lie $n$-groupoids as Lie $n$-algebroids, via the intermediate step of symmetric $n$-fold vector bundles with an additional geometric structure
that will correspond to the homological vector field.

\subsection*{Outline, main results and applications}
This paper is organised as follows.  

Section \ref{section_partitions} collects some facts about (ordered) integer partitions, ordered partitions of finite subsets of $\N$ and their sign, cube categories.
Partitions and integer partitions, as well as the interplay between the two notions, are at the core of many constructions in this paper.

Section \ref{background_n_manifolds} 
quickly recalls the definitions of $[n]$-manifolds and their morphisms. Then it discusses in details the split case, in particular morphisms of split $\N$-manifolds and their composition. $\N$-manifold cocycles are then defined as structures that are equivalent to $\N$-manifolds.

Section \ref{multiple_vb} gives necessary background on  $n$-fold vector bundles, as well as on their cores, their linear splittings and their linear decompositions. 
Then it introduces a new indexing for the iterated highest order cores of an $n$-fold vector bundle, and fills a gap in the proof of Corollary 3.6 in \cite{HeJo20}, establishing the existence
 of a linear decomposition for each $n$-fold vector bundle \cite{HeJo20}. This yields that each $n$-fold vector bundle has an $n$-fold vector bundle atlas \cite{HeJo20}. Appendix \ref{dec_splittings} refines the correspondence between linear splittings and decompositions of $n$-fold vector bundles by showing a uniqueness result, which is needed in this paper.

Section \ref{symmetric_nvb} introduces symmetric $n$-fold vector bundles and their morphisms. It discusses decomposed symmetric $n$-fold vector bundles, as well as symmetric linear splittings and decompositions of symmetric $n$-fold vector bundles. Morphisms of decomposed symmetric $n$-fold vector bundles, as well as their composition, are discussed in detail since this is crucial for defining symmetric $n$-fold vector bundle cocycles. The existence of a symmetric linear decomposition of a symmetric $n$-fold vector bundle is proved in Appendix \ref{proof_dec_sym_nvb}.

Finally, Sections \ref{algebraisation} constructs from the equality of the cocycles the functors from symmetric $n$-fold vector bundles to $[n]$-manifolds, and from $[n]$-manifold to symmetric $n$-fold vector bundles, that establish together the equivalence between the two categories.

\subsection*{Relation to other work}
Note that the question solved in this paper is already raised in 
\cite{Vishnyakova19}. More precisely, the author of \cite{Vishnyakova19} asks how to extend the table on Page \pageref{table_question} with the geometrisation of graded manifolds of higher degree. However, $n$-fold vector bundles are \emph{defined} in this reference as special $\mathbb Z^n$-graded manifolds, see also
\cite{Voronov12} and \cite{GrRo09}. The idea is that a double vector bundle $D$ with sides $A$ and $B$ defines a bigraded algebra $D[1]_A[1]_B$, see also \cite{GrMe10a}, and similarly, $n$-fold vector bundles define $n$-graded algebras.
The paper \cite{Vishnyakova19} conversely associates to an $[n]$-manifold $\mathcal M$ its $n$-times iterated tangent bundle and establishes using this idea an equivalence between graded manifolds and multiple vector bundles. 

The approach and result here are different since $[n]$-manifolds are geometrised by objects in ``classical'' differential geometry --  in the sense that the symmetric $n$-fold vector bundles considered here are not considered in the graded setting.

\subsection*{Acknowledgements}
The authors thank Leonid Ryvkin and Chenchang Zhu for encouraging them to complete this long-standing project and for discussions on the differentiation of Lie $n$-groupoids. They thank as well  Christian Blohmann and Lory Kadiyan for interesting discussions, and again Lory Kadiyan for her careful reading of and useful comments on an early version of this paper.

\subsection{Notation and convention}

The cardinality of a finite set $S$ is written $\# S$. As a convention, $0$ is not considered a natural number.

\medskip

Let $E\to M$ and $F\to N$ be smooth vector bundles and let $\omega\colon E\to F$ be 
a morphism of vector bundles over a smooth map $\omega_0\colon M\to
N$. 
The dual of $\omega$ is in general not a morphism of
vector bundles, but a morphism $\omega^\star$ of modules over the unital algebra morphism
$\omega_0^*\colon C^\infty(N)\to C^\infty(M)$:
\begin{equation}\label{dual_of_VB_map}
\omega^\star\colon \Gamma(F^*)\to \Gamma(E^*),
\qquad  \omega^\star(\epsilon)(m)=\omega_m^*(\epsilon_{\omega_0(m)})
\end{equation}
for all $\epsilon\in\Gamma(F^*)$ and $m\in M$. 
The following lemma is immediate (see e.g.~the appendix of \cite{Jotz18b} for a proof).
\begin{lemma}\label{bundlemap_eq_to_morphism}
  The map $\cdot^\star$, that sends a morphism of vector bundles
  $\omega\colon E\to F$ over $\omega_0\colon M\to N$ to the morphism
  $\omega^\star\colon\Gamma(F^*)\to\Gamma(E^*)$ of modules over
  $\omega_0^*\colon C^\infty(N)\to C^\infty(M)$, is a bijection.
\end{lemma}

\medskip

Let $E\to M$ be a smooth vector bundle.
This paper adopts the convention that the wedge product $\omega\wedge\eta\in\Gamma(\wedge^{k+l}E^*)$ of two forms $\omega\in\Gamma(\wedge^kE^*)$ and $\eta\in\Gamma(\wedge^lE^*)$ is given by 
\[\omega\wedge \eta=\frac{(k+l)!}{k!l!}\operatorname{Alt}(\omega\otimes \eta), 
\]
i.e.~
\begin{equation}\label{standard_wedge_convention}
(\omega\wedge\eta)(e_1,\ldots, e_{k+l})=\frac{1}{k!l!}\sum_{\sigma\in S_{k+l}}(-1)^\sigma\cdot \omega(e_{\sigma(1)}, \ldots, e_{\sigma(k)})\cdot \eta(e_{\sigma(k+1)}, \ldots, e_{\sigma(k+l)})
\end{equation}
for $e_1,\ldots, e_{k+l}$ in the same fiber of $E$.

\section{Partitions and integer partitions}\label{section_partitions}
Partitions and integer partitions, together with a notion of sign that partitions define, are at the core of the equivalence between $\N$-manifolds and symmetric multiple vector bundles.
Partitions of finite subsets of $\N$ are also important for indexing iterated higher order cores of multiple vector bundles. This section collects all the notions, notations and constructions with partitions that are needed in the rest of the paper.
\medskip

In this paper the following notation is used.
For $\nset:=\{1,\ldots,n\}$ the 
\textbf{standard $n$-cube category $\square^n$} 
is the category with subsets $I$ of $\nset$ as objects and with arrows
$I\to J\, \Leftrightarrow \, J\subseteq I$.
More generally, an \textbf{$n$-cube category} is
a category that is isomorphic to the standard $n$-cube category
  $\square^n$.

\subsection{Partitions and cube-categories}
Choose a finite subset $I\subseteq \nset$. A \textbf{partition of $I$} is a set $\{I_1,\ldots, I_l\}$ of pairwise disjoint and non-empty subsets $I_1, \ldots, I_l\subseteq I$ for some $l\in\mathbb N$, such that $I=I_1\cup\ldots \cup I_l$. 
Since the order of the subsets \emph{should not} have any significance, i.e.~$\{I_1,I_2,\ldots, I_l\}$ is naturally the same partition of $I$ as $\{I_2,I_1,\ldots, I_l\}$,
for clarity a partition is always  assumed to be \emph{ordered}, and if not mentioned otherwise, in the \textbf{canonical order}.
That is,
$0<\# I_1\leq \# I_2\leq \ldots \leq \# I_l$ and the subsets of same
cardinality are ordered by lexicographic order. Later on, the fact that partitions are always ordered is also crucial for defining \emph{signs} of partitions.
The set of (canonically ordered) partitions of $I$ is written $\mathcal P(I)$ and a partition $\{I_1,\ldots,I_l\}$  is written as a list $(I_1,\ldots, I_l)$ if it needs to be emphasized that it is ordered -- in particular if it is non-canonically ordered, in which case this is always clearly stated.
However, in some situations, it is simpler if the indices of the elements of a partition do not take into account that the partition is ordered (see for instance Lemma \ref{intersectionl-1} below), and then the partition is simply written as a set.

\medskip

Each partition $\rho=(I_1,\ldots, I_l)$ of $\nset$ in $l$ subsets defines an \textbf{$l$-cube category $\lozenge^\rho$} with 
objects the subsets 
\begin{equation}\label{building_bundles_l_cores} 
J\subseteq \nset \text{ such that for each }i=1,\ldots, l, \text{ either } I_i\cap J=\emptyset \text{ or } I_i\subseteq J
\end{equation}
and with arrows $J\to J' \,\, \Leftrightarrow \,\,J'\subseteq J$. 
The $\#\rho$-cube category $\lozenge^\rho$ is a full subcategory of $\square^n$. 
The inclusion functor is written \[i^\rho\colon \lozenge^\rho\to\square^n.\]

\begin{remark}
Given $\rho=(I_1,\ldots, I_l)$, with $l\in\{1,\ldots, n\}$, then the objects of $\lozenge^\rho$ are the unions of elements of $\rho$. The canonical equivalence with the standard $l$-cube category $\square^l$ is given by sending $J\in\Obj(\square^l)$ to $\cup_{j\in J}I_j\in \Obj(\lozenge^\rho)$ and vice-versa. So $\lozenge^\rho$ is really just $\square^l$, but with the ``indices'' $I_1,\ldots, I_l$ replacing the natural numbers $1, \ldots, l$. \emph{Then $\square^n$ is simply understood as the $n$-cube category over the elements $1, \ldots, n$, while $\lozenge^\rho$ is the $l$-cube category over the elements $I_1,\ldots, I_l$.} This point of view is useful for expressing some of the constructions later in the paper, in particular in Section \ref{higher_cores}, and also for keeping track of them.
\end{remark}

Note that if $\rho=(I_1,\ldots,I_l)$ is a canonically ordered partition of
$I\subseteq\nset$, then $(\# I_1,\ldots,\# I_l)$ is a naturally ordered integer partition of
$\# I$. The tuple $(\# I_1,\ldots,\# I_l)$ is called the \textbf{size} of the partition $(I_1,\ldots,I_l)$, and also denoted by $|I_1,\ldots, I_l|=|\rho|$.
In general, integer partitions are assumed to be ordered and written as tuples. If not specified otherwise, they are naturally ordered.

\medskip

For simplicity, a partition of $\nset$ in $l$ subsets is sometimes called an \textbf{$l$-partition} of $\nset$ and $l$ is called the \textbf{length} of $\rho$.
Let $\rho$ be such an $l$-partition of $\nset$. Then a \textbf{coarsement} of $\rho$ is a partition $\rho'$ of $\nset$ that consists of (non-empty) unions of the elements of $\rho$.
The length of $\rho'$ is thus necessarily less than or equal to the length of $\rho$.
The set of coarsements of $\rho$ is written $\operatorname{coars}(\rho)$. Given $(J_1,\ldots, J_{l'})\in\operatorname{coars}(\rho)$ and  $i\in\{1,\ldots, l'\}$,  then $\rho\cap J_i$
is defined to be the (canonically ordered) partition of $J_i$ given by the elements of $\rho$ it is a union of.

A partition $\rho'$ of $\nset$ is a \textbf{refinement} of $\rho$ if $\rho$ is a coarsement of $\rho'$. For the study of cores in Section \ref{higher_cores}, it is important to understand the different coarsements of length $l-2$ of an $l$-partition.

\begin{lemma}\label{intersectionl-1}
Let $\rho=\{I_1,\ldots, I_l\}$ be a partition of $\nset$. Consider two different \emph{coarsements} \[\rho_{ij}=\left\{I_i\cup I_j, I_1, \ldots, \widehat{I_i}, \ldots, \widehat{I_j}, \ldots, I_l\right\}
\text{ and } \rho_{rs}=\left\{I_r\cup I_s, I_1, \ldots, \widehat{I_r}, \ldots, \widehat{I_s}, \ldots, I_l\right\}\] of length $l-1$ of $\rho$, with $i<j$ and $r<s$ in $\{1, \ldots, l\}$.
Then 
\[ \Obj\left(\lozenge^{\rho_{ij}}\right)\cap \Obj\left(\lozenge^{\rho_{rs}}\right)=\Obj\left(\lozenge^{\rho'}\right)
\]
with $\rho'$ the coarsement of length $l-2$ of $\rho$ given by 
\begin{equation}\label{intersection_rho}
 \rho'=\left\{\begin{array}{ll}
\left\{I_i\cup I_j\cup I_r;  I_t \mid t\in\lset, t\neq i,j,r\right\} & \text{ if } i=s \text{ or } j=s,\\
 \left\{I_i\cup I_j\cup I_s;  I_t \mid t\in\lset, t\neq i,j,s\right\} & \text{ if } i=r \text{ or } j=r,\\
 \left\{ I_i\cup I_j, I_r\cup I_s; I_t \mid t\in\lset, t\neq i,j,r,s\right\} & \text{ if } i\neq r,s \text{ and } j\neq r,s.
\end{array}\right.
\end{equation}

\end{lemma} 

\begin{definition}
In the situation of the previous lemma, the $(l-2)$-cube category $\lozenge^{\rho'}$ is denoted $\lozenge^{\rho_{ij}}\cap\lozenge^{\rho_{rs}}$, and called the \textbf{intersection} 
of the two $(l-1)$-cube categories $\lozenge^{\rho_{ij}}$ and $\lozenge^{\rho_{rs}}$. It is again a full subcategory of $\square^n$.

The partition $\rho'$ is itself sometimes written $\rho_{ij}\sqcap\rho_{rs}$. It is the unique common $(l-2)$-coarsement of  $\rho_{ij}$ and $\rho_{rs}$.
\end{definition}

\begin{proof}[Proof of Lemma \ref{intersectionl-1}]
A subset $K\subseteq \nset$ is an object of $\lozenge^{\rho_{ij}}$ and of $\lozenge^{\rho_{rs}}$ if and only if 
$K\cap I_t=\emptyset$ or $I_t\subseteq K$ for all $t\in \{1,\ldots,l\}\setminus\{i,j,r,s\}$ and 
\begin{enumerate}
\item[(i)] $I_i\cup I_j\subseteq K$ and $I_r\cup I_s\subseteq K$ i.e.~$I_i\cup I_j\cup I_r\cup I_s\subseteq K$, or 
\item[(ii)] $(I_i\cup I_j)\cap K=\emptyset$ and $(I_r\cup I_s)\cap K=\emptyset$ i.e.~$(I_i\cup I_j\cup I_r\cup I_s)\cap K=\emptyset$, or 
\item[(iii)] $I_i\cup I_j\subseteq K$ and $(I_r\cup I_s)\cap K=\emptyset$, or
\item[(iv)] $I_r\cup I_s\subseteq K$ and $(I_i\cup I_j)\cap K=\emptyset$.
\end{enumerate}
If $\{i,j\}\cap\{r,s\}\neq \emptyset$, the cases (iii) and (iv) are not possible. The intersection $\{i,j\}\cap \{r,s\}$ has then one element because $\rho_{ij}\neq \rho_{rs}$ and by (i) and (ii) $K$ is an object of $\lozenge^{\rho'}$ with one of the first two descriptions of $\rho'$ in \eqref{intersection_rho}.
If $\{i,j\}\cap\{r,s\}=\emptyset$ then $K$  is an object of $\lozenge^{\rho'}$ with the third description of $\rho'$.
\end{proof}

\subsection{Partitions and signs}\label{partitions_and_orders}
For $\sigma\in S_n$ and $I\subseteq\nset$, denote by
$\epsilon(\sigma,I)$ the \emph{order of inversions} (or simply the \emph{sign}) of
$\sigma\an{I}\colon I\to \sigma(I)$, where $I$ and $\sigma(I)$ are
equipped with the natural order inherited from the natural numbers.
That is, 
\[ \epsilon(\sigma,I):=(-1)^{\#\{(i,j)\in I\times I\mid i<j,
  \sigma(i)>\sigma(j)\}}=\prod_{\substack{i,j\in I\\ i<j}}\frac{\sigma(j)-\sigma(i)}{j-i}
\]
For $\sigma,\nu\in S_n$ and $I\subseteq\nset$, the product formula
\begin{equation}\label{prod1}
\epsilon(\sigma\nu,I)=\epsilon(\sigma,\nu(I))\cdot\epsilon(\nu,I)
\end{equation}
is immediate.

\bigskip

An ordered
partition $\rho=(I_1,\ldots,I_l)$ of $I$ defines as follows a reordering of the elements of $I$. If
$I_j=\{i^j_1,\ldots,i^j_{k_j}\}$ for $j=1,\ldots,l$, with $i^j_1,\ldots,i^j_{k_j}\in\N$ naturally ordered, then
\begin{equation}\label{reordering_I}
i^1_1,\ldots,i_{k_1}^1,i^2_1,\ldots,i^2_{k_2},\ldots,i^l_1,\ldots,i^l_{k_l}
\end{equation}
is a reordering of the elements of $I$. For instance, the (non-canonically) ordered partition $(\{2\},\{1,4\}, \{3\})$ of $\{1,2,3,4\}$ defines the reordering 
$2,1,4,3$ of the numbers $1,2,3,4$.

Given $n\in\N$ and a (not necessarily naturally) ordered integer partition $(i_1,\ldots,i_l)$ of $n$, set the \textbf{canonical partition} $\rho^{i_1,\ldots,i_l}_{\rm can}$ of $\{1,\ldots, n\}=\nset$ to be the unique ordered partition
 $\rho^{i_1,\ldots,i_l}_{\rm can}:=(K_1,\ldots,K_l)$ of $\nset$ with $\# K_s=i_s$ for $s=1,\ldots, l$ and such that the order on $\nset$ induced by $\rho^{i_1,\ldots,i_l}_{\rm can}$ is the natural one.
For instance, for the integer partition $(1,2,1,3,3)$ of the number $10$, the canonical partition 
$\rho^{1,2,1,3,3}_{\rm can}$ of $\underline{10}$ is the partition $(\{1\},\{2,3\},\{4\},\{5,6,7\}, \{8,9,10\})$. 

Given $\sigma \in S_n$ and an ordered partition $\rho=(I_1,\ldots,I_l)$ of a subset $I\subseteq \nset$, write $\sigma(\rho)$ for the ordered partition $(\sigma(I_1),\ldots, \sigma(I_l))$.

\medskip
Given a subset $I\subseteq \nset$ and a partition $\rho=(I_1,\ldots,I_l)$ of $I$, set $\# I_j=:i_j$ and $i =\# I=i_1+\ldots+i_l$.
Consider $\rho_{\rm can}^{i_1,\ldots, i_l}=(K_1,\ldots, K_l)$ and choose a permutation $\sigma\in S_n$ with 
\[ \sigma(K_j)=I_j \quad \text{ for } \quad j=1,\ldots,l,
\]
i.e.~with $\rho=\sigma\left(\rho^{i_1,\ldots,i_l}_{\rm can}\right)$.
Set then 
\[ \sgn(\rho):=
\frac{\prod_{j=1}^l\epsilon(\sigma, K_j)}{\epsilon(\sigma,\{1,\ldots, i\})},
\]
the \textbf{sign of the partition $\rho$}.

\begin{lemma}\label{lemma_sign_partition_welldef}
In the situation above, the sign $\sgn(\rho)$ does not depend on the choice of $\sigma\in S_n$ with 
$\sigma(K_j)=I_j$ for $j=1,\ldots,l$. Hence $\sgn(\rho)$ equals 
\[\epsilon(\sigma,\{1,\ldots, i\})
\]
for a permutation $\sigma\in S_n$
with $\rho=\sigma\left(\rho^{i_1,\ldots,i_l}_{\rm can}\right)$ and preserving in addition the order of the sets $K_1,\ldots, K_l$.

As a consequence, $\sgn(\rho)$ is the sign of the unique permutation $I \to I$ sending $I$ in its canonical order to $I$ in its order induced by $\rho$ as in \eqref{reordering_I}.
\end{lemma}

\begin{proof}
Assume that $\nu\in S_n$ is a second permutation with $\nu(K_j)=I_j$ for $j=1,\ldots,l$. Then $\nu=\lambda\circ \sigma$ with 
 $\lambda\in S_n$ satisfying $\lambda(I_j)=I_j$ for $j=1,\ldots,l$.
 
 First assume that $\lambda\an{I}$ is a transposition, say $\lambda=(s,t)$ with $s\neq t$ in $\{1,\ldots, n\}$. Since $\lambda$ preserves each of the sets $I_j$, $j=1,\ldots,l$, there exists an $r\in \{1,\ldots,l\}$ such that $\lambda\an{I_j}=\id_{I_j}$ for $j\in\{1,\ldots, l\}\setminus\{r\}$ and $\lambda\an{I_r}$ is a transposition.
 Then
\begin{equation}\label{signs_transposition}
 \prod_{j=1}^l\epsilon(\lambda, I_j)=\epsilon(\lambda, I_r)=-1=\epsilon(\lambda,I). 
\end{equation}
The general case follows with \eqref{prod1}: a permutation $\lambda$ preserving each of the sets $I_1,\ldots,I_l$ can be written as a product of transpositions $\lambda_1, \ldots, \lambda_s$ preserving these sets.
Then
\begin{equation}\label{signs_general_preserving}
\begin{split}
\prod_{j=1}^l\epsilon(\lambda, I_j)&=\prod_{j=1}^l\epsilon(\lambda_s\circ\ldots\circ\lambda_1, I_j)\overset{\eqref{prod1}}{=}\prod_{j=1}^l\prod_{k=1}^s\epsilon(\lambda_k, (\lambda_{k-1}\circ\ldots\circ \lambda_1)(I_j))\\
&=\prod_{k=1}^s\left(\prod_{j=1}^l\epsilon(\lambda_k, (\lambda_{k-1}\circ\ldots\circ \lambda_1)(I_j))\right)\\
&\overset{\eqref{signs_transposition}}{=}\prod_{k=1}^s\epsilon(\lambda_k, (\lambda_{k-1}\circ\ldots\circ \lambda_1)(I))\overset{\eqref{prod1}}{=}\epsilon(\lambda, I).
\end{split}
\end{equation}
Now the claim follows with a similar computation:
\[ \frac{\prod_{j=1}^l\epsilon(\nu, K_j)}{\epsilon(\nu,\{1,\ldots,i\})}=\frac{\prod_{j=1}^l\epsilon(\lambda\sigma, K_j)}{\epsilon(\lambda\sigma,\{1,\ldots,i\})}\overset{\eqref{prod1}}{=}\frac{\prod_{j=1}^l\epsilon(\lambda, I_j)}{\epsilon(\lambda,I)}\frac{\prod_{j=1}^l\epsilon(\sigma, K_j)}{\epsilon(\sigma, \{1,\ldots,i\})}\overset{\eqref{signs_general_preserving}}{=}\frac{\prod_{j=1}^l\epsilon(\sigma, K_j)}{\epsilon(\sigma, \{1,\ldots,i\})}.
\]

To see the last claim of the lemma, note that there is a unique bijection $f\colon I\to \{1,\ldots, i\}$ preserving the natural order of both sets.
The sign of this map is hence $1$ and its composition with $\sigma$ gives then the bijection $\sigma\circ f\colon I \to I$, which sends $I$ in its canonical order to $I$ in the order induced by $\rho$ as in \eqref{reordering_I}.
\end{proof}

Finally, note that given $i\in\nset$ and an ordered integer partition $(i_1,\ldots, i_l)$ of $i$, the canonical partition $\rho_{\rm can}^{i_1,\ldots,i_l}$ of $\underline i$ has sign $1$ by construction.

\subsection{Integer partitions of natural numbers}
For $n\in\N$ let $\mathcal P(n)$ be the set of naturally ordered integer partitions of the natural numbers\footnote{The reader should be aware that the notations $\mathcal P(n)$ and $\mathcal P(\nset)$ are not consistent:  $\mathcal P(n)$ is the set of ordered integer partitions of the integers between 1 and $n$, while $\mathcal P(\nset)$ is the set of canonically ordered partitions of  the set $\nset$. These conventions differ for simplicity of some of the formulas.}
$1\leq j\leq n$:
\[\mathcal P(n):=\left\{ (i_1,\ldots, i_l)\in\N^l\left|\,\,  l\in\N, 1\leq i_1\leq\ldots\leq i_l \text{ and } 1\leq i_1+\ldots+i_l\leq n\right.\right\}.
\]
Given a partition $p=(i_1,\ldots, i_l)\in\mathcal P(n)$, the sum $\Sigma p$ of $p$ is simply 
\[ \sum p:=\sum_{j=1}^li_j.
\]
Given $p\in\mathcal P(n)$, the number $\sharp p$ is the length of the list $p$, i.e.~the number of its elements.

\medskip

As explained in the preceding section, there is a bijection between ordered integer partitions of a number $i\in \N$ and ordered partitions of $\underline i$ inducing the canonical order on the elements of $\underline i$.   
A integer partition $(i_1,\ldots,i_l)$ of $k$ corresponds via this bijection to the unique partition $\rho^{i_1,\ldots,i_l}_{\rm can}$.

\section{The category of $\N$-manifolds}\label{background_n_manifolds}
This section begins by recalling the definitions of $\N$-manifolds and their
morphisms, and by discussing in detail the split case.  The reader is referred to
\cite{Mehta06,BoPo13} for more details.
\subsection{Definitions}

An \textbf{$\N$-manifold}  $\mathcal M$ 
  of degree $n\in \N$ and dimension $(m;
  r_1,\ldots, r_n)$ is a smooth manifold $M$ of dimension $m$ together
  with a sheaf $C^\infty(\mathcal M)$ of $\N$-graded, graded
  commutative, associative, unital $C^\infty(M)$-algebras, that is locally freely
  generated by $r_1+\ldots+ r_n$ elements
  $\xi_1^{1},\ldots,\xi_1^{r_1}$, $\xi_2^1,\ldots,\xi_2^{r_2},\ldots$,
  $\xi_n^1,\ldots,\xi_n^{r_n}$ with $\xi_i^j$ of degree $i$ for
  $i\in\{1,\ldots,n\}$ and $j\in\{1,\ldots,r_i\}$.

  A morphism of $\N$-manifolds $\mu\colon\mathcal M\to
  \mathcal N$ over a smooth map $\mu_0\colon M\to N$ of the underlying
  smooth manifolds is a morphism $\mu^\star\colon C^\infty(\mathcal
  N)\to C^\infty(\mathcal M)$ of sheaves of graded algebras over
  $\mu_0^*\colon C^\infty(N)\to C^\infty(M)$.
\medskip

Note that the degree $0$ elements of $C^\infty(\mathcal M)$ are
precisely the smooth functions on $M$.  In the following 
 an $\N$-manifold of degree $n\in \N$ is called an \textbf{$[n]$-manifold}. 
$|\xi|$ is the degree of a homogeneous element $\xi\in
C^\infty(\mathcal M)$, i.e.~an element which can be written as a sum
of functions of the same degree, and $C^\infty(\mathcal M)^i$
is the space of elements of degree $i$ in $C^\infty(\mathcal M)$.  Note that a
\textbf{$[1]$-manifold} over a smooth manifold $M$ is equivalent to a
locally free and finitely generated sheaf of $C^\infty(M)$-modules, hence to a vector bundle over $M$.
The category of $[n]$-manifolds is written $\mathsf{[n]Man}$, and the category of $\N$-manifolds
is written $\N\mathsf{Man}$.

Given an $[n]$-manifold $\mathcal M$ over a smooth manifold $M$, and an open subset $U\subseteq M$, then $\mathcal M\an{U}$ is the $n$-manifold over $U$ 
defined by $C^\infty(\mathcal M\an{U}):=C^\infty(\mathcal M)\an U$. For simplicity, it is called the \textbf{restriction} of $\mathcal M$ to $U$.

\subsection{The category of split $\N$-manifolds}\label{cat_snMan}
  Let $E$ be a smooth vector bundle of rank $r$ over a smooth manifold
  $M$ of dimension $m$. Assign the degree $n$ to the fiber coordinates
  of $E$. This defines $E[-n]$, an $[n]$-manifold of dimension
  $(m;r_1=0,\ldots,r_{n-1}=0,r_n=r)$ with generators
  $C^\infty(E[-n])^n=\Gamma(E^*)$ and $C^\infty(E[-n])^0=C^\infty(M)$.

 Now let $E_{1},E_{2},\ldots,E_{n}$ be smooth vector bundles
  of finite ranks $r_1,\ldots,r_n$ over $M$ and assign the degree $-i$
  to the fiber coordinates of $E_{i}$, for each $i=1,\ldots,n$.  The
  direct sum $E=E_{1}\oplus \ldots\oplus E_{n}$ is a graded vector
  bundle with grading concentrated in degrees $-1,\ldots,-n$.  The
  $[n]$-manifold $E_{1}[-1]\oplus\ldots\oplus E_{n}[-n]$
  has local basis sections of ${E_{i}^*}$ as local generators of
  degree $i$, for $i=1,\ldots,n$, and so dimension $(d;r_1,\ldots,r_n)$. 
  The $[n]$-manifold $E_{1}[-1]\oplus\ldots\oplus
  E_{n}[-n]$ is
  called a \textbf{split $[n]$-manifold}.

  For instance, in the case $n=3$, choose three vector bundles $E_{1}$, $E_2$
  and $E_{3}$ of ranks $r_1$, $r_2$ and $r_3$ over a smooth manifold
  $M$. Then the $[3]$-manifold $\mathcal M=E_{1}[-1]\oplus E_{2}[-2]\oplus E_3[-3]$ is defined by
  $C^\infty(\mathcal M)^0=C^\infty(M)$,
  $C^\infty(\mathcal M)^1=\Gamma(E_{1}^*)$, 
  $C^\infty(\mathcal M)^2=\Gamma(E_{2}^*\oplus \wedge^2E_{1}^*)$,  $C^\infty(\mathcal M)^3=\Gamma(E_3^*\oplus E_1^*\otimes E_2^*\oplus \wedge^3E_1^*)$, and 
  \[C^\infty(\mathcal M)^4=\Gamma(E_4^*\oplus E_1^*\otimes E_3^*\oplus S^2E_2^*\oplus \wedge^2E_1^*\otimes E_2^*\oplus \wedge^4E_1^*),
  \] etc.
 Set $E_{j}^*\odot E_{k}^*$ to mean
  \begin{itemize}
  \item $\wedge$ if $j=k$ is an odd number,
  \item the symmetric product $\cdot$ if $j=k$ is an even number,
  \item $\otimes$ if $j\neq k$.
  \end{itemize}
  Then given an $[n]$-manifold $\mathcal M=E_{1}[-1]\oplus\ldots\oplus E_{n}[-n]$, its functions of degree $k$ are elements of 
  \[ C^\infty(\mathcal M)^k=\bigoplus_{\substack{1\leq i_1\leq\ldots\leq i_l\leq n,\\ i_1+\ldots+i_l=k}}\Gamma\left(E^*_{i_1}\odot\ldots\odot E^*_{i_l}\right)
  \]
  for all $k\in\mathbb N$.

\bigskip

The graded product $\odot$ on
\[ C^\infty(\mathcal M)=\bigoplus_{k\in\mathbb N}\bigoplus_{\substack{1\leq i_1\leq\ldots\leq i_l\leq n,\\ i_1+\ldots+i_l=k}}\Gamma\left(E^*_{i_1}\odot\ldots\odot E^*_{i_l}\right)
  \]
is given by the following lemma.
\begin{lemma}\label{graded_wedge_prod}
Let $\mathcal M=E_{1}[-1]\oplus \ldots\oplus E_{n}[-n]$ be a split $[n]$-manifold over a smooth manifold $M$ and 
consider  $\xi_1,\ldots, \xi_l\in C^\infty(\mathcal M)$ such that
\[ \xi_j\in \Gamma\left(E^*_{i^j_1}\odot\ldots\odot E^*_{i^j_{l_j}}\right)
\]
with $1\leq i_1^j\leq \ldots\leq i^j_{l_j}\leq n$, for $j=1,\ldots, l$. That is, $\xi_j$ has degree $d_j:=i^j_1+\ldots+i^j_{l_j}$.
Set $p_j=(i_1^j,\ldots, i_{l_j}^j)$. It is an ordered integer partition of $d_j$ for $j=1,\ldots, l$.

The list $p$ obtained as the union of $p_1,\ldots, p_l$, and canonically reordered, is then an integer partition with $s:=l_1+\ldots+l_l$ elements of its sum $d=d_1+\ldots+d_l$. Consider the canonical partition 
$\rho^{\rm can}_p=(K_1,\ldots, K_s)$ of $\{1,\ldots, d\}$ associated to it. 
The graded product
\[ \xi_1\odot\ldots\odot\xi_l 
\]
a section of $E^*_{|K_1|}\odot\ldots\odot E^*_{|K_s|}$, 
is defined on $e_{K_1}\in E_{|K_1|}, \ldots, e_{K_s}\in E_{|K_s|}$ by 
\begin{equation}\label{graded_wedge}
\begin{split}
 \sum_{\substack{\rho^p_{\rm can}=\rho_1\cup\ldots\cup \rho_l\\
\text{as sets and}\\
\rho_1, \ldots, \rho_l \text{canonically ordered with}\\
|\rho_j|=p_j \text{ for }j=1,\ldots, l
}}\!\!\!\!\!\!\!\!\!\!\!\!\!\!\sgn(\rho_1,\ldots, \rho_l)\cdot \xi_1(e_{K}\mid K\in \rho_1)\cdot\ldots\cdot \xi_l(e_{K}\mid K\in\rho_l).
\end{split}
\end{equation}
Here, $(\rho_1,\ldots, \rho_l)$ is the partition $\rho_{\rm can} ^p$ of $\underline d$, but reordered in the order of\footnote{For instance, if $p_1=(1,3)$ and $p_2=(1,2)$, then $d=7$, $p=(1,1,2,3)$ and $\rho_{\rm can}^p=(\{1\}, \{2\}, \{3,4\}, \{5,6,7\})$.  Then if $\rho_1=(\{2\}, \{3,4\})$ and $\rho_2=(\{1\}, \{5,6,7\})$, then $(\rho_1,\rho_2)=(\{2\}, \{3,4\},\{1\}, \{5,6,7\})$.
}
 $\rho_1,\ldots, \rho_l$.
\end{lemma}

Note that here, the notation $\xi_j(e_{K}\mid K\in \rho_j)$ is short for the following. If $\rho_j=(K_{i_1}, \ldots, K_{i_{l_j}})$, then 
\[\xi_j(e_{K}\mid K\in \rho_j)=\xi_j(e_{K_{i_1}}, \ldots, e_{K_{i_{l_j}}}).
\]

\begin{example}
Before considering the proof, the reader can convince themself with a few examples that \eqref{graded_wedge} works as it should.
\begin{enumerate}
\item Take $\xi_1\in\Gamma(E_1^*)$, $\eta_1\odot\eta_2\in\Gamma(E_1^*\odot E_2^*)$ and $\tau_2\odot\tau_3\in\Gamma(E_2^*\odot E_3^*)$.
Then $p_1=(1)$, $p_2=(1,2)$ and $p_3=(2,3)$, and so $p=(1,1,2,2,3)$. Here \[\rho_{\rm can}^p=(\{1\}, \{2\}, \{3,4\}, \{5,6\}, \{7,8,9\})\]
and so $(\rho_1,\rho_2, \rho_3)$ range over 
\begin{itemize}
\item $((\{1\}), (\{2\},\{3,4\}), (\{5,6\}, \{7,8,9\}))$ with sign $1$,
\item $((\{2\}), (\{1\},\{3,4\}), (\{5,6\}, \{7,8,9\}))$ with sign $-1$,
\item $((\{1\}), (\{2\},\{5,6\}), (\{3,4\}, \{7,8,9\}))$ with sign $1$,
\item $((\{2\}), (\{1\},\{5,6\}), (\{3,4\}, \{7,8,9\}))$ with sign $-1$.
\end{itemize}
Therefore 
a quick computation shows that $(\xi_1\odot(\eta_1\odot\eta_2)\odot(\tau_2\odot\tau_3))(e_1, e_2, e_{34}, e_{56}, e_{789})$ equals
\[ \left(\xi_1(e_1)\eta_1(e_2)-\xi_1(e_2)\eta_1(e_1)\right)\cdot \left(\eta_2(e_{34})\tau_2(e_{56})+\eta_2(e_{56})\tau_2(e_{34})\right)\cdot \tau_3(e_{789}).
\]
\item Take $\xi\in\Gamma(E_3^*)$, $\eta\in\Gamma(E_1)$ and $\tau\in\Gamma(E_2^*)$. Then $p_1=(3)$, $p_2=(1)$ and $p_3=(2)$ and so $p=(1,2,3)$ and 
$\rho_{\rm can}^p=(\{1\}, \{2,3\}, \{4,5,6\})$. The only possible list $(\rho_1,\rho_2,\rho_3)$ is here 
\[(\{4,5,6\},\{1\}, \{2,3\}),
\]
which has sign $-1$. Hence 
\[ (\xi\odot\eta\odot\tau)(e_1,e_{23}, e_{456})=-\xi(e_{456})\cdot\eta(e_1)\cdot\tau(e_{23}),
\]
which shows 
\[ \xi\odot\eta\odot\tau=-\eta\otimes\tau\otimes\xi.
\]

 \item The reader is invited to check that $\xi\odot \eta$ for $\xi\in\Gamma(\wedge^kE_1^*)$ and $\eta\in\Gamma(\wedge^lE_1^*)$ equals
 \[ \xi\odot\eta=\frac{(k+l)!}{k!l!}\operatorname{Alt}(\xi\otimes \eta)=\xi\wedge \eta.
 \]
 This example illustrates nicely the compatibility of the graded tensor in Lemma \ref{graded_wedge_prod} with the usual wedge product of sections of $\Gamma(\wedge^\bullet E_1^*)$. 
 \end{enumerate}
\end{example}

\begin{proof}[Proof of Lemma \ref{graded_wedge_prod}]
First, the graded symmetry of this `product' in the entries $\xi_1,\ldots, \xi_l$ is checked as follows:
It suffices to show that for $i=1,\ldots, l-1$, exchanging $\xi_i$ with $\xi_{i+1}$ in \eqref{graded_wedge} is the same as a multiplication by the factor $(-1)^{d_id_{i+1}}$. This exchange amounts to exchanging $p_i$ with $p_{i+1}$ in the list of partitions, 
and so exchanging the roles of $\rho_i$ and $\rho_{i+1}$ in each summand of \eqref{graded_wedge}, which only amounts to replacing in each summand the factor
$\sgn(\rho_1,\ldots, \rho_l)$ by $\sgn(\rho_1, \ldots, \rho_{i+1}, \rho_i, \ldots,\rho_l)$. 

But given a permutation $\sigma\in S_d$ such that 
$\sgn(\rho_1,\ldots, \rho_l)=(-1)^\sigma$, the composition of $\sigma$ with $d_i\cdot d_{i+1}$ transpositions on both sides\footnote{Note that the canonical partitions 
$\rho_{\rm can}^{(p_1,\ldots, p_l)}$ and $\rho_{\rm can}^{(p_1,\ldots, p_{i+1},p_i,\ldots, p_l)}$ are not equal in general.}
 gives a permutation $\sigma'\in S_d$ such that 
$\sgn(\rho_1,\ldots,\rho_{i+1}, \rho_i, \ldots, \rho_l)=(-1)^{\sigma'}$. Hence 
\[ \sgn(\rho_1, \ldots, \rho_{i+1}, \rho_i, \ldots,\rho_l)=(-1)^{d_id_{i+1}}\sgn(\rho_1,\ldots, \rho_l).
\]

\bigskip

Then check that \eqref{graded_wedge} is indeed an element of $\Gamma\left(E^*_{|K_1|}\odot\ldots \odot E^*_{|K_s|}\right)$.
The following shows that if  $|K_i|=|K_{j}|=a$ for some $i<j\in\{1,\ldots, s\}$, then exchanging $e_{K_i}$ with $e_{K_{j}}$ in \eqref{graded_wedge}, amounts to multiplying it by $(-1)^a=(-1)^{a\cdot a}$. 

Consider a list of canonically ordered partitions $\rho_1,\dots,\rho_l$ the union of which is $\rho_{\rm can }^p$ and such that $\rho_j$ has size $p_j$ for $j=1,\ldots, l$. Write $(\rho_1,\ldots, \rho_l)=(I_1,\ldots, I_s)$.
Exchanging $K_i$ and $K_j$ in this list defines a new list
$(\rho_1, \ldots,\rho_l)^{ij}=(\rho_1^* , \ldots, \rho_l^*)=(I_1^*,\ldots, I_d^*)$ with size $p$. 
First assume that the partitions $\rho_1^*, \ldots, \rho_l^*$ are all still (canonically) ordered. By the last  claim of Lemma \ref{lemma_sign_partition_welldef}, 
 $\sgn(\rho_1, \ldots, \rho_d)^{ij}=(-1)^a\sgn(\rho_1,\ldots, \rho_l)$,
 since the set $\underline d$ in its order induced by $(\rho_1,\ldots, \rho_l)$ is sent to $\underline d$ in its order induced by $(\rho_1^*, \ldots, \rho_l^*)$ by $a$ permutations exchanging the elements of $K_i$ with the elements of $K_j$.
 As a consequence, exchanging $e_{K_i}$ with $e_{K_j}$  in the summand of \eqref{graded_wedge} defined by $(\rho_1,\ldots, \rho_l)$ yields 
\begin{equation*}
\begin{split}
&\sgn( \rho_1,\ldots, \rho_l)\cdot \xi_1(e_K\mid K\in \rho_1^*)\cdot\ldots\cdot \xi_l(e_{K}\mid K\in\rho_l^*)\\
&=(-1)^a\sgn(\rho_1^*,\ldots, \rho_l^*)\cdot \xi_1(e_K\mid K\in \rho_1^*)\cdot\ldots\cdot \xi_l(e_{K}\mid K\in\rho_l^*),
\end{split}
\end{equation*}
which is $(-1)^a$ times the term of \eqref{graded_wedge} defined by $(\rho_1^*, \ldots, \rho_l^*)$.

Assume now that $K_i\in \rho_r$ and $K_j\in \rho_t$ for some $r,t\in\{1,\ldots, l\}$, and that 
exchanging $K_i$ with $K_j$ leads to one or both of $\rho_r^*$ and $\rho_t^*$ not being canonically ordered anymore. If $r=t$, then exchanging $e_{K_i}$ and $e_{K_j}$ in the summand of \eqref{graded_wedge} defined by $(\rho_1,\ldots, \rho_l)$ multiplies the same term by $(-1)^a$ since $\xi_r$ is graded-symmetric.

Assume next that $r\neq t$ and replacing $K_i$ by $K_j$ in $\rho_r$ leads to a partition which is not ordered anymore. Without loss of generality, $K_i$ needs to be exchanged with one element $K_h$ in this partition for it  becoming canonically ordered again. The obtained ordered partition is then denoted by $\rho_r^*$ as above. 
Assume without loss of generality that after this step, $\rho_1^*, \ldots, \rho_l^*$ are all ordered. 
By Lemma \ref{lemma_sign_partition_welldef}, $\sgn(\rho_1^*, \ldots, \rho_d^*)=(-1)^{2a}\sgn(\rho_1,\ldots, \rho_l)=\sgn(\rho_1,\ldots, \rho_l)$.
Then exchanging $e_{K_i}$ with $e_{K_j}$ in the summand of \eqref{graded_wedge} defined by $(\rho_1,\ldots, \rho_l)$ yields 
\begin{equation*}
\begin{split}
&\sgn(\rho_1,\ldots, \rho_l)\cdot \xi_1(e_K\mid K\in \rho_1^*)\cdot\ldots\cdot(-1)^a\cdot \xi_r(e_K\mid K\in\rho_r^*)\cdot\ldots\cdot \xi_l(e_{K}\mid K\in\rho_l^*)\\
&=(-1)^a\cdot\sgn(\rho_1^*,\ldots, \rho_l^*)\cdot \xi_1(e_K\mid K\in \rho_1^*)\cdot\ldots\cdot\xi_r(e_K\mid K\in\rho_r^*)\cdot\ldots\cdot \xi_l(e_{K}\mid K\in\rho_l^*).
\end{split}
\end{equation*}
The remaining cases are treated similarly.

The above shows that exchanging $e_{K_i}$ with $e_{K_j}$ in \eqref{graded_wedge} multiplies all summands of \eqref{graded_wedge} by $(-1)^a=(-1)^{a\cdot a}$. As a consequence 
\begin{equation*}
\begin{split}
&(\xi_1\odot\ldots\odot\xi_l)(e_{K_1}, \ldots, e_{K_{j}}, \ldots,  e_{K_i},\ldots, e_{K_d})
= (-1)^a\cdot (\xi_1\odot\ldots\odot\xi_l)(e_{K_1}, \ldots, e_{K_d}).
\end{split}
\end{equation*}
That is,  $\xi_1\odot \ldots\odot\xi_l$ is graded symmetric.
\bigskip

Then show that the graded symmetric product $\odot$ coincides with the graded-symmetric product on generators.
Take $\xi_j\in\Gamma(E^*_{i_j})$ for $j=1,\ldots, l$.
Set for $i=1,\ldots, n$
\[ M_i=\{i_j \mid j\in\{1,\ldots, l\} \text{ and }  i_j=i\}.
\]
Set $\mathcal S_i$ to be the set of permutations of $M_i$ for $i=1,\ldots, n$. 

Assume without loss of generality (see the first step of this proof) that $1\leq i_1\leq \ldots\leq i_l\leq n$. Then $p=(i_1,\ldots, i_l)$ 
 and  \eqref{graded_wedge} reads as follows on $e_{K_1}, \ldots, e_{K_l}$, with $\rho_{\rm can}^{p}=(K_1,\ldots, K_l)$.
\begin{equation*}
\begin{split}
(\xi_1\odot\ldots\odot\xi_l)(e_{K_1},\ldots, e_{K_l})&= 
\sum_{\substack{\rho_{\rm can}^{(i_1,\ldots,i_l)}=(I_1,\ldots,I_l)\text{ as sets}\\
\# I_j=i_j \text{ for } j=1,\ldots, l
}}\!\!\!\!\!\!\!\!\!\!\!\!\sgn(I_1,\ldots, I_l)\cdot \xi_1(e_{I_1})\cdot\ldots\cdot \xi_l(e_{I_l}).
\end{split}
\end{equation*}
A choice of a (non-canonically) ordered partition $(I_1,\ldots, I_l)$ of $\underline d$ such that $\rho_{\rm can}^{(i_1,\ldots,i_l)}=(I_1,\ldots,I_l)$ as sets and $\# I_j=i_j$ for $j=1,\ldots, l$
amounts to a choice of order of each of the tuples 
\[ \left( K_{j}\, \mid \, j\in M_i \right),
\]
 hence of a permutation $\sigma_i\in \mathcal S_i$ for $i=1,\ldots, n$.
It is then easy to see that 
\begin{equation*}
\begin{split}
&
\sum_{\substack{\rho_{\rm can}^{(i_1,\ldots,i_l)}=(I_1,\ldots,I_l)\text{ as sets}\\
|I_j|=i_j \text{ for } j=1,\ldots, l
}}\sgn(I_1,\ldots, I_l)\cdot \xi_1(e_{I_1})\cdot\ldots\cdot \xi_l(e_{I_l})\\
&=\prod_{i=1}^n\left(
\sum_{\sigma\in \mathcal S_i}((-1)^{\sigma})^i\cdot\prod_{j=1+\sum_{r=1}^{i-1}N_r}^{\sum_{r=1}^{i}N_r}\xi_j(e_{K_{\sigma(j)}})\right).
\end{split}
\end{equation*}
That is, $\xi_1\odot\ldots\odot\xi_l$ is the block-wise graded skew-symmetrisation
of 
\[ \xi_1\otimes\ldots\otimes \xi_l=\bigotimes_{i=1}^n\underset{\in \Gamma\left((E_i^*)^{\otimes N_i}\right)}{\underbrace{\left(\bigotimes_{j=1+\sum_{r=1}^{i-1}N_r}^{\sum_{r=1}^{i}N_r}\xi_j\right)}}. \]

\bigskip

Finally the associativity \eqref{graded_wedge} is proved. Consider without loss of generality the case $l=3$, and show as follows that 
\[ (\xi_1\odot\xi_2)\odot\xi_3=\xi_1\odot\xi_2\odot\xi_3=\xi_1\odot(\xi_2\odot\xi_3).
\]
By the graded-symmetry of $\odot$, it is enough to prove the first equality on the left-hand side. Let $q$ be the partition obtained as the reordered union of $p_1$ and $p_2$. Then $p$ is the reordered union of $q$ and $p_3$. With the same notation as above, the form 
$(\xi_1\odot\xi_2)\odot\xi_3$ applied to $e_{K_1},\ldots, e_{K_s}$ reads
\begin{equation*}
\begin{split}
&
\!\!\!\!\!\!\!\!\!\!\!\!\!\!\!\!\sum_{\substack{\rho^p_{\rm can}=\rho\cup \rho_3\\
\text{as sets and}\\
\rho, \rho_3 \text{canonically ordered with}\\
|\rho|=q \text{ and } |\rho_3|=p_3
}}\!\!\!\!\!\!\!\!\!\!\!\!\!\!\sgn(\rho, \rho_3)\cdot (\xi_1\odot\xi_2)(e_{K}\mid K\in \rho)\cdot \xi_3(e_{K}\mid K\in\rho_3)\\
=&
\!\!\!\!\!\!\!\!\!\!\!\!\!\!\!\!\sum_{\substack{\rho^p_{\rm can}=\rho\cup \rho_3\\
\text{as sets and}\\
\rho, \rho_3 \text{canonically ordered with}\\
|\rho|=q \text{ and } |\rho_3|=p_3
}}\!\!\!\!\!\!\!\!\!\!\!\!\!\!\sgn(\rho, \rho_3)\cdot\!\!\!\!\!\!\!\!\!\!\!\sum_{\substack{\rho=\rho_1\cup \rho_2\\
\text{as sets and}\\
\rho_1, \rho_2 \text{canonically ordered with}\\
|\rho_1|=p_1 \text{ and } |\rho_2|=p_2
}}\!\!\!\!\!\!\!\!\!\!\!\!\!\!\sgn^\rho(\rho_1,\rho_2)\cdot \xi_1(e_K\mid K\in \rho_1)\cdot\xi_2(e_{K}\mid K\in \rho_2)\cdot \xi_3(e_{K}\mid K\in\rho_3),
\end{split}
\end{equation*}
where $\sgn^\rho(\rho_1,\rho_2)$ is the sign of the unique permutation of $I:=\cup\rho$ sending $I$ in its order induced by $\rho$ to $I$ in its order induced by $(\rho_1,\rho_2)$, see Lemma \ref{lemma_sign_partition_welldef}. 
This is $(\xi_1\odot\xi_2\odot\xi_3)(e_{K_1},\ldots, e_{K_s})$ if 
\begin{equation}\label{sgnsgn}
\sgn(\rho,\rho_3)\cdot\sgn^\rho(\rho_1,\rho_2)=\sgn(\rho_1,\rho_2,\rho_3)
\end{equation}
 for ordered partitions $\rho, \rho_1,\rho_2,\rho_3$ such that 
$\rho^p_{\rm can}=\rho\cup \rho_3$ and $\rho=\rho_1\cup \rho_2$ as sets, and
$|\rho|=q$, $|\rho_3|=p_3$, 
$|\rho_1|=p_1$, $|\rho_2|=p_2$.

 In order to show \eqref{sgnsgn}, use the last claim of Lemma \ref{lemma_sign_partition_welldef}.
 Let $I$ be the union of the elements of $\rho$, i.e.~let $\rho$ be an ordered partition of $I\subseteq \underline d$, and write the elements of  $I$ as $i_1,\ldots, i_{d_1+d_2}$, in the order defined by $\rho$. The elements of $I_j:=\cup\rho_j$ are $i^j_1, \ldots, i^j_{d_j}$ in the order defined by $\rho_j$, for $j=1,2,3$. Note that $I=I_1\cup I_2$.
Then $\sgn(\rho, \rho_3)$ is the sign of the unique permutation $\sigma$ of $\underline d$ sending $\underline d$ in its canonical order to $\underline d$ in the order 
\[ i_1,\ldots, i_{d_1+d_2}, i^3_1, \ldots, i^3_{d_3}.
\]
The number $\sgn^\rho(\rho_1,\rho_2)$ is the sign of the unique permutation $\tau$ of $I_1\cup I_2=I$ sending $I$ in the order $ i_1,\ldots, i_{d_1+d_2}$ to $I$ in the order 
\[ i_1^1,\ldots, i_{d_1}^1, i^2_1, \ldots, i^2_{d_2},
\]
and the sign of $(\rho_1,\rho_2,\rho_3)$ is the sign of the unique permutation $\eta$ of $\underline d$ sending $\underline d$ in its canonical order to $\underline d$ in the order 
\[ i_1^1,\ldots, i^1_{d_1}, i^2_1, \ldots, i^2_{d_2}, i^3_1, \ldots, i^3_{d_3}.
\]
Extend $\tau$ to a permutation of $\underline d$ by setting $\tau\an{I_3}=\id_{I_3}$, which does not change its sign. Then $\eta$ equals $\tau\circ \sigma$, and so $\sgn(\rho_1,\rho_2,\rho_3)=(-1)^\eta=(-1)^\tau(-1)^\sigma=\sgn^\rho(\rho_1,\rho_2)\cdot\sgn(\rho, \rho_3)$.
 \end{proof}

\bigskip
In this paper the category of split $[n]$-manifolds is written $\mathsf{s[n]Man}$.
  A morphism
  \[\mu\colon E_{1}[-1]\oplus \ldots\oplus E_{n}[-n]\to F_{1}[-1]\oplus
  \ldots\oplus F_{m}[-m]\] of split $\N$-manifolds over the bases $M$
  and $N$, respectively, consists of a smooth map
  $\mu_0\colon M\to N$, and for each  $p=(i_1,\ldots,i_l)\in\mathcal P(m)$ 
  a  morphism
  \[ \mu_{p}\colon E_{i_1}\odot E_{i_2}\odot\ldots\odot E_{i_l}\to F_{i_1+\ldots+i_l}
  \]
  of vector bundles over $\mu_0$.
  The map
  $\mu^\star$ sends a degree $k$ generator $\xi\in\Gamma(F_{k}^*)$ to
  \begin{equation*}
    \sum_{\substack{p=(i_1,\ldots,i_l)\in\mathcal P(m)\\
      \sum p=k}} \quad \underset{\in\Gamma(E_{i_1}^*\odot E_{i_2}^*
      \odot\ldots\odot E^*_{i_l})}{\underbrace{{\mu_{(i_1, \ldots,i_l)}}^\star(\xi)}} 
      \,\in C^\infty(E_{1}[-1]\oplus \ldots\oplus E_{n}[-n])^k.
\end{equation*}
The morphism $\mu$ is therefore written here
\[ \mu=\left(\mu_p\right)_{p\in\mathcal P(m)},
\]
the smooth map $\mu_0$ being here implicit as the common base map of all the vector bundle morphisms $\mu_p$, $p\in\mathcal P(m)$.

In the following, given an ordered integer partition $(i_1,\ldots, i_l)$ of a natural number and a list of vectors $e_1\in E_{i_1},\ldots, e_l\in  E_{i_l}$, it is useful as before to index this list by 
\[ K_1, \ldots, K_l
\]
where $(K_1,\ldots, K_l)=\rho_{\rm can}^{(i_1,\ldots, i_l)}$. Hence 
\[ e_{K_1}\in E_{\# K_1}, \ldots, e_{K_l}\in E_{\# K_l}.
\]

\begin{lemma}
Let \[\mu\colon E_{1}[-1]\oplus \ldots\oplus E_{n}[-n]\to F_{1}[-1]\oplus
  \ldots\oplus F_{m}[-m]\] be a morphism of split $\N$-manifolds over base manifolds $M$
  and $N$. Let $\xi_1\in\Gamma(F^*_{d_1}), \ldots, \xi_l\in\Gamma(F^*_{d_l})$, with $d_1\leq \ldots \leq d_l$.

  Then for each list of ordered integer partitions $p_1, \ldots, p_l$ such that $\Sigma p_j=d_j$, $j=1,\ldots, l$
  the image $ \mu^\star(\xi_1\odot\ldots\odot\xi_l)$ has a term in 
  \[ \Gamma\left(E_{i_1}^*\odot \ldots \odot E_{i_s}^*\right),
  \]
  if $p=(i_1,\ldots, i_s)$ is the canonically reordered  integer partition $p_1\cup\ldots\cup p_l$.
  
 More precisely, set $\rho^p_{\rm can}=(K_1,\ldots, K_s)$. Then for all $e_{K_1}\in \Gamma(E^*_{\#K_1}), \ldots, e_{K_s}\in\Gamma(E^*_{\#K_s})$, 
 $\mu^\star(\xi_1\odot\ldots\odot\xi_l)(e_{K_1}, \ldots, e_{K_s})$ equals 
  \begin{equation}\label{mu_on_products_precise}
 \!\!\!\!\!\!\!\!\!\!\!\!\!\!\!\sum_{\substack{\rho_1,\ldots, \rho_l \text{ ordered partitions}\\
 \rho_{\rm can}^p=\rho_1\cup\ldots\cup\rho_l \text{ as sets}\\
 \cup\rho_1<\ldots<\cup \rho_l\\
 \#\cup\rho_1=d_1,\ldots, \#\cup\rho_l=d_l
 }}\!\!\!\!\!\!\!\!\!\!\!\!\!\!\!\!\sgn(\rho_1,\ldots,\rho_l)\cdot(\xi_1\odot\ldots\odot\xi_l)\left(\mu_{|\rho_1|}(e_K\mid K\in\rho_1), \ldots, \mu_{|\rho_l|}(e_K\mid K\in\rho_l)\right).
  \end{equation}
\end{lemma}

Here, for each $j$
the section $\mu_{|\rho_j|}$
 is fed the list of vectors
$(e_{K} \mid K\in \rho_j)$. The latter is a notation for the \emph{ordered} tuple with order given by the (canonical) order of the  partition $\rho_j$.

\begin{proof}
Since $\mu^\star$ is a morphism of graded  algebras, the image of $\xi_1\odot\ldots\odot\xi_l$ under $\mu^\star$ is given by 
  \begin{equation}\label{mu_on_products}
  \begin{split}
  \mu^\star(\xi_1\odot\ldots\odot\xi_l)&=\sum_{\substack{p_1,\ldots, p_l\in\mathcal P(n)\\
\sum p_j=d_j \text{ for } j=1,\ldots, l\\  }
  }\mu^\star_{p_1}(\xi_1)\odot\ldots\odot \mu^\star_{p_l}(\xi_l).
  \end{split}
  \end{equation}
  
  Assume for simplicity that $l=2$ and take first $d_1<d_2$. Choose as in the claim two ordered integer partitions $p_1,p_2$ such that $\sum p_1=d_1$ and $\sum p_2=d_2$.
  Then 
 \begin{equation*}
  \begin{split}
  \mu^\star(\xi_1\odot\xi_2)(e_{K_1}, \ldots, e_{K_s})&=\sum_{\substack{q_1,q_2\in\mathcal P(n)\\
\sum p_j=d_j \text{ for } j=1,2\\ 
q_1\cup q_2=p \text{ as sets}}
  }\left(\mu^\star_{q_1}(\xi_1)\odot \mu^\star_{q_2}(\xi_2)\right)(e_{K_1}, \ldots, e_{K_s}).
  \end{split}
  \end{equation*}
  By Lemma \ref{graded_wedge_prod} this equals 
   \begin{equation}\label{sum_morphism_detail}
  \begin{split}
&  \sum_{\substack{q_1,q_2\in\mathcal P(n)\\
\sum p_j=d_j \text{ for } j=1,2\\ 
q_1\cup q_2=p \text{ as sets}}
  }\sum_{\substack{\rho_1,\rho_2 \text{ ordered }\\
  \rho^p_{\rm can}=\rho_1\cup \rho_2 \text{ as sets }\\
  |\rho_1|=q_1, |\rho_2|=q_2}}\sgn(\rho_1,\rho_2)\cdot \xi_1\left(\mu_{q_1}(e_K\mid K\in \rho_1)\right)\cdot \xi_2\left(\mu_{q_2}(e_{K}\mid K\in \rho_2)\right)\\
  &=\sum_{\substack{\rho_1,\rho_2 \text{ ordered }\\
  \rho^p_{\rm can}=\rho_1\cup \rho_2 \text{ as sets }\\
  \#\cup\rho_1=d_1,\#\cup\rho_2=d_2}}\sgn(\rho_1,\rho_2)\cdot \xi_1\left(\mu_{|\rho_1|}(e_K\mid K\in \rho_1)\right)\cdot \xi_2\left(\mu_{|\rho_2|}(e_{K}\mid K\in \rho_2)\right).
  \end{split}
  \end{equation}
  But since $\#\cup\rho_1=d_1<d_2=\#\cup\rho_2$ for each pair of ordered partitions $\rho_1,\rho_2$ indexing this sum, the sum equals 
  \begin{equation*}
  \begin{split}
\sum_{\substack{\rho_1,\rho_2 \text{ ordered }\\
  \rho^p_{\rm can}=\rho_1\cup \rho_2 \text{ as sets }\\
  \#\cup\rho_1=d_1,\#\cup\rho_2=d_2\\
  \cup\rho_1<\cup \rho_2}}\sgn(\rho_1,\rho_2)\cdot \xi_1\left(\mu_{|\rho_1|}(e_K\mid K\in \rho_1)\right)\cdot \xi_2\left(\mu_{|\rho_2|}(e_{K}\mid K\in \rho_2)\right).
  \end{split}
  \end{equation*}
  \medskip
  
  Next take $l=2$ but $d_1=d_2=:a$. Choose as above two different ordered integer partitions $p_1,p_2$ such that $\sum p_1=\sum p_2=a$ and assume for simplicity\footnote{In general, $\mu^\star(\xi_1\odot\xi_2)$ has more terms, that need to be considered alone or pairwise, as is done here.} that under all candidates for these two partitions, only $p_1$ and $p_2$ give non-vanishing vector bundle morphisms $\mu_{p_1}$ and $\mu_{p_2}$. Then $\mu^\star(\xi_1\odot\xi_2)$ is given by 
  \begin{equation*}
  \begin{split}
  \mu^\star(\xi_1\odot\xi_2)=\mu_{p_1}^\star(\xi_1)\odot\mu_{p_1}^\star(\xi_2)+\mu_{p_1}^\star(\xi_1)\odot\mu_{p_2}^\star(\xi_2)+\mu_{p_2}^\star(\xi_1)\odot\mu_{p_1}^\star(\xi_2)+\mu_{p_2}^\star(\xi_1)\odot\mu_{p_2}^\star(\xi_2).
  \end{split}
  \end{equation*}
  Since $p_1\neq p_2$, the reordered unions $p_1\cup p_1$, $p_2\cup p_2$ and $p_1\cup p_2$ are all different.
  Consider first the term $\mu_{p_1}^\star(\xi_1)\odot\mu_{p_1}^\star(\xi_2)$. (The term $\mu_{p_2}^\star(\xi_1)\odot\mu_{p_2}^\star(\xi_2)$ is treated in the same manner.)
   Let again $p$ be the reordered union $p_1\cup p_1$ and let $\rho^p_{\rm can}=(K_1,\ldots, K_s)$. In this case, by Lemma \ref{graded_wedge_prod}
   the value of $\left(\mu_{p_1}^\star(\xi_1)\odot\mu_{p_1}^\star(\xi_2)\right)(e_{K_1},\ldots, e_{K_s})$
   reads
   \begin{equation*}
  \begin{split}
& \sum_{\substack{\rho_1,\rho_2 \text{ ordered }\\
  \rho^p_{\rm can}=\rho_1\cup \rho_2 \text{ as sets }\\
  |\rho_1|=|\rho_2|=p_1}}\sgn(\rho_1,\rho_2)\cdot \xi_1\left(\mu_{p_1}(e_K\mid K\in \rho_1)\right)\cdot \xi_2\left(\mu_{p_1}(e_{K}\mid K\in \rho_2)\right)\\
  &=\sum_{\substack{\rho_1,\rho_2 \text{ ordered }\\
  \rho^p_{\rm can}=\rho_1\cup \rho_2 \text{ as sets }\\
 |\rho_1|=p_1=|\rho_2|\\
  \cup\rho_1<\cup\rho_2}}\sgn(\rho_1,\rho_2)\cdot \xi_1\left(\mu_{p_1}(e_K\mid K\in \rho_1)\right)\cdot \xi_2\left(\mu_{p_1}(e_{K}\mid K\in \rho_2)\right)\\
  &\hspace*{4cm} +\sgn(\rho_2,\rho_1)\cdot \xi_1\left(\mu_{p_1}(e_K\mid K\in \rho_2)\right)\cdot \xi_2\left(\mu_{p_1}(e_{K}\mid K\in \rho_1)\right)\\
  &=\sum_{\substack{\rho_1,\rho_2 \text{ ordered }\\
  \rho^p_{\rm can}=\rho_1\cup \rho_2 \text{ as sets }\\
 |\rho_1|=p_1=|\rho_2|\\
  \cup\rho_1<\cup\rho_2}}\sgn(\rho_1,\rho_2)\cdot \Bigl(\xi_1\left(\mu_{p_1}(e_K\mid K\in \rho_1)\right)\cdot \xi_2\left(\mu_{p_1}(e_{K}\mid K\in \rho_2)\right)\\
  &\hspace*{5cm} \qquad+(-1)^a\cdot \xi_1\left(\mu_{p_1}(e_K\mid K\in \rho_2)\right)\cdot \xi_2\left(\mu_{p_1}(e_{K}\mid K\in \rho_1)\right)\Bigr)\\
   &=\sum_{\substack{\rho_1,\rho_2 \text{ ordered }\\
  \rho^p_{\rm can}=\rho_1\cup \rho_2 \text{ as sets }\\
 |\rho_1|=p_1=|\rho_2|\\
  \cup\rho_1<\cup\rho_2}}\sgn(\rho_1,\rho_2)\cdot(\xi_1\odot\xi_2)\left(\mu_{|\rho_1|}(e_K\mid K\in \rho_2), \mu_{|\rho_2|}(e_{K}\mid K\in \rho_1)\right).
  \end{split}
  \end{equation*}
  
 Consider the terms $\mu_{p_1}^\star(\xi_1)\odot\mu_{p_2}^\star(\xi_2)+\mu_{p_2}^\star(\xi_1)\odot\mu_{p_1}^\star(\xi_2)$. Let $p$ be the reordered union of $p_1$ and $p_2$ and set $\rho_{\rm can}^p=(K_1,\ldots, K_s)$.
 Here
   \[\left(\mu_{p_1}^\star(\xi_1)\odot\mu_{p_2}^\star(\xi_2)+\mu_{p_2}^\star(\xi_1)\odot\mu_{p_1}^\star(\xi_2)\right)(e_{K_1},\ldots, e_{K_s})\]
   reads 
   \begin{equation*}
   \begin{split}
 &  \sum_{\substack{\rho_1,\rho_2 \text{ ordered }\\
  \rho^p_{\rm can}=\rho_1\cup \rho_2 \text{ as sets }\\
 |\rho_1|=p_1, |\rho_2|=p_2\\
  }}\sgn(\rho_1,\rho_2)\cdot \xi_1\left(\mu_{p_1}(e_K\mid K\in \rho_1)\right)\cdot \xi_2\left(\mu_{p_2}(e_{K}\mid K\in \rho_2)\right)\\
  &+\sum_{\substack{\rho_1,\rho_2 \text{ ordered }\\
  \rho^p_{\rm can}=\rho_1\cup \rho_2 \text{ as sets }\\
 |\rho_1|=p_1, |\rho_2|=p_2\\
  }}\sgn(\rho_2,\rho_1)\cdot \xi_1\left(\mu_{p_2}(e_K\mid K\in \rho_2)\right)\cdot \xi_2\left(\mu_{p_1}(e_{K}\mid K\in \rho_1)\right).
   \end{split}
   \end{equation*}
For a choice of two partitions $\rho_1$ and $\rho_2$ indexing these sums, either\footnote{In fact, for some choices of $p_1$ and $p_2$, 
 $\cup\rho_1<\cup\rho_2$ is automatically true. For instance if $p_1=(1,3)$ and $p_2=(2,2)$ the partition $p$ is $p=(1,2,2,3)$ and the partitions
  must be $\rho_1=(\{1\},\{6,7,8\})$ and $\rho_2=(\{2,3\}, \{4,5\})$.
  But for some choices of $p_1$ and $p_2$ the order of $\cup\rho_1$ and $\cup\rho_2$ cannot be predicted. For instance if $p_1=(1,3)$ and $p_2=(1,1,2)$ since then $p=(1,1,1,2,3)$ and  $\rho_1=(\{1\},\{6,7,8\})$ and $\rho_2=(\{2\},\{3\},\{4,5\})$ satisfy $\cup\rho_1<\cup\rho_2$ while 
$\rho_1=(\{2\},\{6,7,8\})$ and $\rho_2=(\{1\},\{3\},\{4,5\})$ satisfy $\cup\rho_1>\cup\rho_2$.}
 $\cup\rho_1<\cup\rho_2$ or $\cup\rho_1>\cup\rho_2$. In both cases the corresponding terms in the sum give together
   \begin{equation*}
   \begin{split}
   &\sgn(\rho_1,\rho_2)\cdot \xi_1\left(\mu_{p_1}(e_K\mid K\in \rho_1)\right)\cdot \xi_2\left(\mu_{p_2}(e_{K}\mid K\in \rho_2)\right)\\
   &+\sgn(\rho_2,\rho_1)\cdot \xi_1\left(\mu_{p_2}(e_K\mid K\in \rho_2)\right)\cdot \xi_2\left(\mu_{p_1}(e_{K}\mid K\in \rho_1)\right),
   \end{split}
   \end{equation*}
   which is written 
   \[ \sgn(\rho_1,\rho_2)(\xi_1\odot\xi_2)(\mu_{|\rho_1|}(e_K \mid K\in\rho_1), \mu_{|\rho_2|}(e_K\mid K\in\rho_2))
   \]
   if $\cup\rho_1<\cup\rho_2$ and 
    \[ \sgn(\rho_2,\rho_1)(\xi_1\odot\xi_2)(\mu_{|\rho_2|}(e_K \mid K\in\rho_2), \mu_{|\rho_1|}(e_K\mid K\in\rho_1))
   \]
   if $\cup\rho_2<\cup\rho_1$.
   
As a summary   \[\left(\mu_{p_1}^\star(\xi_1)\odot\mu_{p_2}^\star(\xi_2)+\mu_{p_2}^\star(\xi_1)\odot\mu_{p_1}^\star(\xi_2)\right)(e_{K_1},\ldots, e_{K_s})\]
equals 
\[ \sum_{\substack{\rho_1,\rho_2 \text{ ordered }\\
  \rho^p_{\rm can}=\rho_1\cup \rho_2 \text{ as sets }\\
 \#\cup\rho_1=a=\#\cup\rho_2\\
 \cup\rho_1<\cup\rho_2
 }}\sgn(\rho_1,\rho_2)\cdot \sgn(\rho_1,\rho_2)(\xi_1\odot\xi_2)(\mu_{|\rho_1|}(e_K \mid K\in\rho_1), \mu_{|\rho_2|}(e_K\mid K\in\rho_2)).
\]

   The above proves the claim for $l=2$. The general case for $l\geq 3$ works in the same manner, but with more different cases to consider.
  \end{proof}

\medskip
The following result is then an immediate corollary of the preceding lemma. The proof is left to the reader.
\begin{corollary}
Given two morphisms 
\[\mu\colon E_{1}[-1]\oplus \ldots\oplus E_{n}[-n]\to F_{1}[-1]\oplus
  \ldots\oplus F_{m}[-m]\] and \[\nu\colon  F_{1}[-1]\oplus \ldots\oplus F_{m}[-m]\to G_{1}[-1]\oplus
  \ldots\oplus G_{q}[-q]\]
  of split $[n]$-manifolds over smooth maps $\mu_0\colon M\to N$ and $\nu_0\colon N\to Q$, respectively, the composition 
  \[
  \nu\circ \mu=\left( (\nu\circ \mu)_p\right)_{p\in\mathcal P(q)}
\]
is defined by
\begin{equation}\label{composition_graded_morphisms}
 (\nu\circ \mu)_p=
\!\!\!\!\!\!\!\!\! \sum_{\substack{\rho_1,\ldots, \rho_l \text{ ordered}\\
\text{such that } \rho_1\cup\ldots\cup \rho_l=\rho^p_{\rm can} \\
\text{ as sets and}\\
 \cup\rho_1< \ldots< \cup\rho_l}} \sgn(\rho_1,\ldots,\rho_l)\cdot\nu_{\left(\Sigma |\rho_1|,\ldots,\Sigma |\rho_l|\right)}\circ \left(\mu_{|\rho_1|}, \ldots, \mu_{|\rho_l|}\right) 
\end{equation}
for all $p\in\mathcal P(q)$. 
\end{corollary}

 Here, 
$(\nu\circ\mu)_p$ is fed a list of vectors
\[ (e_{K_1},\ldots, e_{K_s})
 \]
 with $\rho^p_{\rm can}=(K_1,\ldots, K_s)$.
The term $\nu_{\left(\Sigma |\rho_1|,\ldots,\Sigma |\rho_l|\right)}\circ \left(\mu_{|\rho_1|}, \ldots, \mu_{|\rho_l|}\right) $ of \eqref{composition_graded_morphisms} indexed by $(\rho_1,\ldots, \rho_l)$ applied to this list is then precisely
\[\nu_{\left(\Sigma |\rho_1|,\ldots,\Sigma |\rho_l|\right)}\left(\mu_{|\rho_1|}(e_{K}\mid K\in \rho_1), \ldots, \mu_{|\rho_l|}(e_{K}\mid K\in \rho_l)\right),
\]
where, as before, 
$(e_{K} \mid K\in \rho_1)$ is a notation for the \emph{ordered} tuple with order given by the  partition $\rho_1$, etc.

\medskip

A morphism $\mu\colon E_{1}[-1]\oplus \ldots\oplus E_{n}[-n]\to F_{1}[-1]\oplus
  \ldots\oplus F_{n}[-n]$ of split $[n]$-manifolds is an \textbf{isomorphism} of split $[n]$-manifolds if it has an inverse.
  More precisely, 
  if $\mu_0$ is a diffeomorphism and there exists a morphism $\nu\colon F_{1}[-1]\oplus \ldots\oplus F_{n}[-n]\to E_{1}[-1]\oplus
  \ldots\oplus E_{n}[-n]$ of split $[n]$-manifolds such that $\nu_0$ is the smooth inverse of $\mu_0$ and 
  \[ (\nu\circ \mu)_p=0 \text{ and } (\mu\circ \nu)_p=0
  \]
  for $p\in\mathcal P(n)$ with $\sharp p\geq 2$, and 
  \[ (\nu\circ \mu)_p=\id_{E_{\Sigma p}}
  \]
  for $\sharp p=1$. (This implies as always $(\mu\circ \nu)_p=\id_{F_{\Sigma p}}$ for $\sharp p=1$.)  For instance, a collection of $n$-isomorphisms $\mu_{(1)}\colon E_1\to F_1$, \ldots, $\mu_{(n)}\colon E_n\to  F_n$ over a smooth diffeomorphism $\mu_0\colon M\to N$  gives an isomorphism of $E_{1}[-1]\oplus \ldots\oplus E_{n}[-n]$ with $F_{1}[-1]\oplus
  \ldots\oplus F_{n}[-n]$ by setting $\mu_p=0$ for $\sharp p\geq 2$.

\medskip Any $\N$-manifold is non-canonically isomorphic to a split
$\N$-manifold of the same degree. More precisely, the embedding of the category of split $[n]$-manifolds in the one of  $[n]$-manifolds
is fully faithful and essentially surjective.

 This is true locally, per definition, and proved globally 
for instance in \cite{BoPo13}, following the proof of the
$\mathbb Z/2\mathbb Z$-graded version of this theorem, which is called
there \emph{Batchelor's theorem} \cite{Batchelor79}, see also
\cite{Berezin87} and \cite{Voronov02}.

\begin{proposition}
Any $[n]$-manifold
  is non-canonically isomorphic to a split $[n]$-manifold.
\end{proposition}

Note that $[1]$-manifolds are automatically split since they are just
vector bundles with a degree shifting in the fibers,
i.e.~a $[1]$-manifold over $M$ is $E[-1]$ for some vector bundle $E\to M$ and
$C^\infty(E[-1])=\Gamma(\wedge^\bullet E^*)$, the exterior
algebra of $E$.

\subsection{$[n]$-manifold cocycles}

Let $\mathcal M$ be an $[n]$-manifold over a smooth manifold $M$ and choose an open cover $(U_\alpha)_{\alpha\in\Lambda}$ of $M$ 
by open sets trivialising $C^\infty(\mathcal M)$. That is, for each $\alpha\in\Lambda$, the $\N$-graded, graded commutative, associative, unital $C^\infty(M)$-algebra
$C^\infty_{U_\alpha}(\mathcal M)$ is freely generated by its elements 
\[  \xi_{\alpha,1}^{1},\ldots,\xi_{\alpha,1}^{r_1}, \xi_{\alpha,2}^1,\ldots,\xi_{\alpha,2}^{r_2}, \ldots,
  \xi_{\alpha,n}^1,\ldots,\xi_{\alpha,n}^{r_n} 
\]
with $\xi_{\alpha,i}^j$ of degree $i$ for
  $i\in\{1,\ldots,n\}$ and $j\in\{1,\ldots,r_i\}$. In other words, for each $\alpha\in\Lambda$ the restriction $\mathcal M\an{U_\alpha}$ is isomorphic via a morphism $\phi_\alpha$ over the identity on $U_\alpha$ to the split $[n]$-manifold
  \[ (U_\alpha\times \R^{r_1})[-1]\oplus \ldots\oplus (U_\alpha\times \R^{r_n})[-n],
  \]
  where the spaces $U_\alpha\times \R^{r_j}$ carry the canonical trivial vector bundle structures over $U_\alpha$, for $j=1,\ldots,n$.
  
 Take $\alpha,\beta\in\Lambda$ and set $U_{\alpha\beta}:=U_\alpha\cap U_\beta$. Then the following diagram of isomorphisms of $[n]$-manifold commutes
  \[
\begin{tikzcd}
\mathcal M\an{U_{\alpha\beta}} \ar[drr,"\phi_\alpha\an{\mathcal M\an{U_{\alpha\beta}}}"]  \ar[d,"\phi_\beta\an{\mathcal M\an{U_{\alpha\beta}}}"] 
&& \\
 (U_{\alpha\beta}\times \R^{r_1})[-1]\oplus \ldots\oplus (U_{\alpha\beta}\times \R^{r_n})[-n]  \ar[rr,"\phi_\alpha\circ\phi_\beta\inv  "] & & (U_{\alpha\beta}\times \R^{r_1})[-1]\oplus \ldots\oplus (U_{\alpha\beta}\times \R^{r_n})[-n]
\end{tikzcd} 
\]
and $\phi^{\alpha\beta}:=\phi_\alpha\circ\phi_\beta\inv$ is an isomorphism of split $[n]$-manifolds. By construction,
\[ \phi^{\alpha\gamma}=\phi^{\alpha\beta}\circ\phi^{\beta\gamma}
\]
over $U_{\alpha\beta\gamma}:=U_\alpha\cap U_\beta\cap U_\gamma$, i.e.
\[ \phi^{\alpha\gamma}_p=(\phi^{\alpha\beta}\circ\phi^{\beta\gamma})_p
\]
for all $p\in\mathcal P(n)$.
The open cover $\{U_\alpha\}_{\alpha\in\Lambda}$ of $M$ together with the collection of isomorphisms
\[( \phi^{\alpha\beta}\mid \alpha,\beta\in\Lambda) \]
satisfying 
\begin{enumerate}
\item $\phi^{\alpha\gamma}=\phi^{\alpha\beta}\circ\phi^{\beta\gamma}$
over $U_{\alpha\beta\gamma}$ for all $\alpha,\beta,\gamma\in\Lambda$ and 
\item $\phi^{\alpha\alpha}=\id_{(U_\alpha\times \R^{r_1})[-1]\oplus \ldots\oplus (U_\alpha\times \R^{r_n})[-n]}$
for all $\alpha\in\Lambda$
\end{enumerate}
is an \textbf{$[n]$-manifold cocycle on $M$}. By their very definition, $[n]$-manifold cocycles on $M$ are equivalent to $[n]$-manifolds over $M$.

\section{Multiple vector bundles and charts}\label{multiple_vb}

This section recalls the definitions, notation
and results of \cite{HeJo20} which are needed in the rest of the paper,
and studies in more detail the \emph{iterated higher order cores} of an $n$-fold vector bundle.

  \medskip
	
In \cite{HeJo20} the authors define as follows an $n$-fold vector bundle. This is just a
different, in their opinion more convenient, formulation for the $n$-fold vector bundles defined by
Mackenzie and Gracia-Saz in \cite{GrMa09}.

\begin{definition}\label{def_n_fold}
  An \textbf{$n$-fold vector bundle}, is a covariant functor
  $\E\colon \square^{n}\to \Man$ --
 to the category of smooth manifolds,
  such that, writing $p^I_{J}:=\E(I\to J)$ for $J\subseteq I\subseteq \nset$,
 \begin{enumerate}
    \item[(a)] for all $I \subseteq \nset$ and all $i\in I$,
  $p^I_{I\setminus\{i\}}\colon \E(I)\to \E(I\setminus \{i\})$ has a 
	smooth vector bundle structure, and 
  \item[(b)] for all
  $I\subseteq \nset$ and $i\neq j\in I$,
\[
\begin{tikzcd}
\E(I)\ar[rr,"p^{I}_{I\setminus\{i\}}"] \ar[d,"p^{I}_{I\setminus\{j\}}"] 
&& \E(I\setminus \{i\})\ar[d,"p^{I\setminus \{i\}}_{I\setminus\{i,j\}} "]\\
\E(I\setminus \{j\}) \ar[rr,"p^{I\setminus \{j\}}_{I\setminus\{i,j\}}  "] 
& & \E(I\setminus \{i,j\}) 
\end{tikzcd} 
\]
is a double vector bundle.
\end{enumerate}

The smooth manifold $\E(\emptyset)$ is denoted $M$ if not mentioned otherwise.

\medskip

Given two $n$-fold vector bundles 
$\E\colon \square^n\to \Man$ and
$\F\colon \square^n\to \Man$, a 
\textbf{morphism of $n$-fold vector bundles} from $\E$ to $\F$ is a
natural transformation $\Phi\colon \E\to \F$ 
such that the commutative diagrams
\[
\begin{tikzcd}
\E(I)\ar[r,"\Phi(I)"] \ar[d,"p^{I}_{I\setminus\{i\}}"] & \F(I)\ar[d,"p^{I}_{I\setminus\{i\}} "]\\
\E(I\setminus \{i\}) \ar[r,"\Phi(I\setminus\{i\})  "] & \F(I\setminus \{i\})
\end{tikzcd} 
\] 
are vector bundle homomorphisms for all
$I\subseteq \nset$ and $i\in I$. The morphism $\tau$ is surjective 
(respectively injective) if each of its components $\tau(I)$, $I\subseteq \nset$ 
is fibrewise surjective (respectively fibrewise injective). It is then called \textbf{an epimorphism} (respectively a \textbf{monomorphism}) of $n$-fold vector bundles.
\end{definition}
Note that such a morphism $\Phi\colon \E\to \F$ is completely determined by its top map $\Phi(\nset)\colon \E(\nset)\to\F(\nset)$. However, the definition above is convenient because it can be extended to morphisms of $\infty$-fold vector bundles \cite{HeJo20}.
\medskip

A subset $S\subseteq \nset$ defines a full subcategory $\square^S$ of the $n$-cube category $\square^n$, with objects the subsets of $S$. Given an $n$-fold vector bundle $\E\colon\square^n\to\Man$, 
the \textbf{$S$-side} of $\E$ is then the restriction $\E_S$ of $\E$ to the subcategory $\square^S$. The functors obtained in this manner are the \textbf{sides} of $\E$. Given two $n$-fold vector bundles $\E$ and $\F$ and a morphism $\Phi\colon \E\to\F$ of $n$-fold vector bundles, the restriction of $\Phi$ to the sides $\E_S$ and $\F_S$ is written $\Phi\an{S}\colon \E_S\to \F_S$.
It is defined by $\Phi\an{S}(J)=\Phi(J)\colon \E(J)\to \F(J)$ for all $J\subseteq S$.

\subsection{Cores of an $n$-fold vector bundle} 
 
   Let $\E\colon \square^n\to \Man$ be an $n$-fold vector bundle, and choose two subsets $J\subseteq S\subseteq \nset$. Then by Proposition 2.18 in \cite{HeJo20}, the space 
   \[ 
   \E^{S}_{J}:=\bigcap_{j\in J}(p^S_{S\setminus\{j\}})\inv\left(
	\nvbzero{S\setminus \{j\}}{S\setminus J}\right)=\left\{ e\in\E(S) \, \left| \, p^S_{S\setminus\{j\}}(e)=\nvbzero{S\setminus \{j\}}{p^S_{S\setminus J}(e)} \text{ for all } j\in J\right.\right\}
   \]
   with $\nvbzero{S\setminus\{j\}}{S\setminus J}\colon \mathbb E(S\setminus J)\to \mathbb E(S\setminus\{j\})$ the composition of zero sections, is an embedded submanifold of $\mathbb E(S)$. 
   (For $J=\emptyset$ the empty intersection is $\E(S)$ by convention, and for $\# J=1$ the ``zero section'' $\nvbzero{S\setminus\{j\}}{S\setminus J}=\nvbzero{S\setminus J}{S\setminus J}$ is the identity on $\E(S\setminus J)$, so $\E^S_J=\E(S)$.)
   It has further a vector bundle structure over $\mathbb E(S\setminus J)$ with projection \[p^S_{S\setminus J}\an{\E^S_J}=\E(S\rightarrow S\setminus J)\an{\E^S_J}\colon \E^S_J\to \mathbb E(S\setminus J)\]
   and with addition defined by
   \[ e_1\dvplus{}{S\setminus J}e_2:=e_1\dvplus{}{S\setminus\{j\}}e_2
   \]
   for any $j\in J$.
   $\E^S_J$ is the total space of an $(\#S-\#J+1)$-fold vector bundle $\E^{(S,J)}$, called here the $(S,J)$-core of $\mathbb E$ and defined as follows.
   Consider the $(\#S-\#J+1)$-cube subcategory $\lozenge^S_J$ of $\square^S$ with objects the subsets $I\subseteq S$ with 
   \[ I\cap J=\emptyset \text{ or } J\subseteq I
   \]
   and with arrows
   \[ I\rightarrow I'\,\, :\Leftrightarrow \,\, I'\subseteq I.
   \]
   That is, $\lozenge^S_J$ is a full subcategory of $\square^S$. Note that if $S=\nset$ the $(n-\# J+1)$-cube category $\lozenge^\nset_J$ equals the $(n-\# J+1)$-cube category $\lozenge^{\rho_J}$
with $\rho_J$ the partition of $\nset$ in the subset $J$ and $\nset-\# J$ subsets with one element each. 
   
The functor $\E^{(S,J)}\colon \lozenge^S_J\to \Man$ sends an object $I$ as above 
\begin{equation}\label{def_objects_core}
 \text{ to }\,\,\E(I)=\E_S(I) \,\,\text{ if } \,\, I\cap J=\emptyset \,\,\text { and to   } \,\, \E^I_J \,\, \text{ if } J\subseteq I,
\end{equation}
and an arrow $I\rightarrow I'$ to 
\begin{equation}\label{def_arrows_core}
 \left\{\begin{array}{lc}
\mathbb E(I\rightarrow I')\an{\E^I_ J}\colon \E^I_J\to \E^{I'} _J & \text{ if } J\subseteq I'\subseteq I,\\
\mathbb E(I\rightarrow I')\colon \E(I)\to \E(I') & \text{ if } I\cap J=\emptyset \text{ and }\\
\mathbb E(I\to I')\an{\E^I_J}\colon \E^I_J\to \E(I')& \text{ if } I'\cap J=\emptyset \text{ but } J\subseteq I.
\end{array}\right.
\end{equation}

Write $i^S_J\colon \lozenge^S_J\to \square^S$ for the inclusion functor.
Then the assignment 
\[ \tau\colon \operatorname{Obj}\left(\lozenge^S_J\right)\to \operatorname{Mor}(\Man)
\]
sending $I\subseteq S$ with 
 $I\cap J=\emptyset$  or $J\subseteq I$ to the smooth embedding 
 \[ \tau(I):=\iota_{\E^{(S,J)}(I)}\colon \E^{(S,J)}(I)\hookrightarrow \E_S(I)=\E(I)
 \]
 defines a natural transformation 
 \[ \E^{(S,J)}\longrightarrow \E_S\circ i^S_J.
 \]
In the following a little abuse of notation is made, and such a natural transformation by embeddings between functors from a subcategory of $\square^n$ is called a \emph{natural transformation by embeddings in $\E$.}

Consider two $n$-fold vector bundles $\E$ and $\F$.
Choose again $J\subseteq S\subseteq \nset$ and build the $(S,J)$-cores $\E^{(S,J)}$ and $\F^{(S,J)}$.
Then a morphism $\Phi\colon\E\to \F$ of $n$-fold vector bundles induces as follows a core morphism
\begin{equation}\label{core_morphism}
\Phi^{(S,J)}\colon \E^{(S,J)}\to \F^{(S,J)}.
\end{equation}
For all $I\in\operatorname{Obj}(\lozenge^S_J)$ the map $\Phi^{(S,J)}(I)\colon \E^{(S,J)}(I)\to \F^{(S,J)}(I)$ is simply the (necessarily smooth) restriction of $\Phi(I)$ to the 
embedded domain and codomain $ \E^{(S,J)}(I)\subseteq \E(I)$ and $\F^{(S,J)}(I)\subseteq \F(I)$, which is well-defined because $\Phi$ preserves the $n$-fold vector bundle structure and so in particular the zeros.

\medskip

Finally note that by definition, the face $\E_{S\setminus J}$ of $\E$ is also a face of the core $\E^{(S,J)}$ since 
\[\operatorname{Obj}(\square^{S\setminus J})\subseteq \operatorname{Obj}(\lozenge^S_J)\]
($\square^{S\setminus J}$ is a full subcategory of $\lozenge^S_J$) and 
\[ \E^{(S,J)}(I)=\E(I)=\E_{S\setminus J}(I)
\]
for all $I\subseteq S\setminus J$.

\subsection{Linear splittings and decompositions of an $n$-fold vector bundle}
Consider a collection $\mathcal A=\{A_I\mid I\subseteq \nset\}$ of vector bundles $A_I$ over a smooth manifold $M$, where as a convention, $A_\emptyset$ is taken to
be the trivial vector bundle $M\to M$.
Set 
\[\E^{\A}\colon \square^n\to\Man,\]
\[\E^{\mathcal A}\colon \square^n\to \Man, \qquad \E^{\mathcal A}(J):=\prod^M_{I\subseteq J}A_I\]
for $J\subseteq \nset$, 
where $\Pi^M$ are the fibered products over $M$.
Set $\E^{\A}(J\rightarrow J')$ to be the canonical projection $\prod^M_{I\subseteq J}A_I\to \prod^M_{I\subseteq J'}A_I$.
Then $\E^{\mathcal A}$ with the obvious vector bundle structures is an $n$-fold vector bundle, the \textbf{decomposed} $n$-fold vector bundle defined by $\mathcal A$ \cite{HeJo20}.
The reader is invited to check that 
\[ (\E^{\mathcal A})^{(S,J)}(I)=\prod^M_{K\in\operatorname{Obj}(\lozenge^I_J)}A_K
\]
for all $J\subseteq I\subseteq S\subseteq \nset$.
Here, the right-hand face is an embedded submanifold of $\E^{\mathcal A}(I)$ by taking its fibered products with the zero sections of the ``missing'' bundles.
 
The \textbf{vacant decomposed} $n$-fold vector bundle defined by $\mathcal A$ is 
\[ \overline{\E^{\mathcal A}}\colon \square^n\to \Man, \quad \overline{\E^{\mathcal A}}(J)=\prod^M_{i\in J}A_ {\{i\}}.
\]
(Note that it uses only the vector bundles in $\mathcal A$ indexed by one-element sets. Hence it can also be defined by a family of $n$ vector bundles over $M$ indexed by the number $1$ to $n$.)
For each $J\subseteq \nset$ the manifold $\overline{\E^{\mathcal A}}(J)$ is clearly also embedded in $\E^{\A}(J)$. Denote the embedding by $\iota(J)\colon \overline{\E^{\mathcal A}}(J)\rightarrow \E^{\A}(J)$.
The collection of these embeddings defines a monomorphism \begin{equation}\label{emb_spl_dec}
\iota\colon \overline{\E^{\mathcal A}}\to \E^{\A}\end{equation}
 of $n$-fold vector bundles.

Write $\overline{\E^{\mathcal A}}=\overline{\E}$ for simplicity.
For $\# J\geq 2$, $J\subseteq S\subseteq \nset$ and all $I\in\operatorname{Obj}(\lozenge^S_J)$ the space 
$\overline{\E}^{(S,J)}(I)$ is here further 
\[ \overline{\E}^{(S,J)}(I)=\prod^M_{i\in I\setminus J}A_ {\{i\}}
\]
since
\[\overline{\E}^{(S,J)}(I)=\left\{\begin{array}{ll}\overline{\E}^I_J  &\text{ if } J\subseteq I\\
\overline{\E}(I) &\text{ if } I\subseteq S\setminus J,
\end{array}\right.
\]
with 
\[ \overline{\E}^I_J=\left\{ e\in\prod^M_{i\in I}A_{\{i\}}\,\left|\ p^I_{I\setminus\{j\}}(e)=\nvbzero{I\setminus\{j\}}{p^I_{I\setminus J}(e)} \text{ for all } j\in J\right.\right\}
=\prod^M_{\{i\}\in \operatorname{Obj}(\lozenge^I_J)}A_{\{i\}}
=\prod^M_{i\in I\setminus J}A_{\{i\}},
\]
The name \emph{vacant} \cite{Mackenzie92} comes from the fact that the $(I,I)$-cores of $\overline{\E}$ are consequently all trivial, for $I\subseteq \nset$ with $\# I\geq 2$.

The induced monomorphism $\tau^{(S,J)}\colon \overline{\E}^{(S,J)}\rightarrow (\E^{\mathcal A})^{(S,J)}$ is clearly the one defined by the natural embeddings 
\[\overline{\E}^{(S,J)}(I)=\prod^M_{i\in I\setminus J}A_{\{i\}} \hookrightarrow \prod^M_{
K\in\operatorname{Obj}(\lozenge^I_J)}A_K=(\E^{\mathcal A})^{(S,J)}(I).
\]
\medskip

Let $\E$ be an $n$-fold vector bundle and consider
the collection of vector bundles $\E^J_J\to M=\E(\emptyset)$ for $J\subseteq \nset$. In particular, $\E^{\{i\}}_{\{i\}}=\E(\{i\})=:E_i$ are the \textbf{lower sides} of $\E$ and $\E^\emptyset_\emptyset=\E(\emptyset)=M$. 
The collection $\mathcal A_{\E}=\left\{\left. \E_I^I \right| \emptyset \neq I\subseteq \nset\right\}$ is the family of \emph{building bundles} of $\E$ in the sense that $\E$ is non-canonically isomorphic to the $n$-fold vector bundle $\E^{\mathcal A_{\E}}$, or in other words, $\E$ is \emph{decomposed by $\A$}.
The proof of the existence of such an isomorphism, called a \emph{decomposition}, is the subject of the following two sections.
\begin{definition}\label{def_n-dec}
Let $\E\colon\square^n\to \Man$ be an $n$-fold vector bundle.
\begin{enumerate}
\item The decomposed $n$-fold vector bundle $\E^{\mathcal A_{\E}}$ defined as above by $\E$ is denoted by $\E^{\rm dec}$.
The vacant decomposed $n$-fold vector bundle $\overline{\E^{\mathcal A_{\E}}}$ is written $\overline{\E}$.
\item A \textbf{linear splitting} of the $n$-fold vector bundle
$\E$ is a monomorphism
$\Sigma\colon\overline{ \E}\to \E$ of $n$-fold vector bundles,
such that for $i=1,\ldots,n$, $\Sigma(\{i\})\colon E_i\to E_i$ is the
identity.

\item A \textbf{decomposition} of the $n$-fold vector bundle $\E$ is a
natural isomorphism $\mathcal S\colon \E^{\rm dec}\to \E$
of $n$-fold vector bundles over the identity maps
$\mathcal S(\{i\})=\id_{E_{i}}\colon E_i\to E_i$ such that
additionally the induced core morphisms $\S^{(I,I)}(\{I\})$ are the
identities $\id_{\E^I_I}$ for all $I\subseteq\nset$.
\end{enumerate}
\end{definition}

By Corollary 3.6 in \cite{HeJo20}, any $n$-fold vector bundle admits a decomposition. Since there appears to be a little gap in the proof of this result in \cite{HeJo20}, it is revisited in Section \ref{fix_of_3.6}.

\bigskip

Let $\mathcal A=(A_I\mid I\subseteq \nset)$, $\mathcal B=(B_I\mid I\subseteq \nset)$ and $\mathcal C=(C_I\mid I\subseteq \nset)$    be three families of vector bundles over smooth manifolds $M$, $N$ and $P$ respectively.        A morphism $\tau\colon \E^{\mathcal A}\to\E^{\mathcal B}$ is easily seen to amount to a collection of vector bundle morphisms 
\[ \tau_\rho\colon A_{I_1}\otimes \ldots\otimes A_{I_l}\to B_I
\]
over $\tau(\emptyset)\colon M\to N$
for all $\emptyset \neq I\subseteq \nset$ and all $\rho=(I_1,\ldots,I_l)\in\mathcal P(I)$, see \cite{HeJo20}.
Given this collection, the map $\tau(\nset)\colon \E^{\mathcal A}(\nset)\to\E^{\mathcal B}(\nset)$ is given by 
\[ \left(a_I\right)_{\emptyset\neq I\subseteq \nset}\quad \mapsto \quad \left(\sum_{\rho=(I_1,\ldots, I_l)\in\mathcal P(I)}\tau_\rho(a_{I_1}, \ldots, a_{I_l})\right)_{\emptyset\neq I\subseteq \nset}.
\]
Given a second morphism $\mu\colon \E^{\mathcal B}\to \E^{\mathcal C}$ over $n$-fold vector bundles, the composition $\mu\circ \tau\colon \E^{\mathcal A}\to \E^{\mathcal C}$ is given by 
\[ (\mu\circ \tau)_\rho\left((a_J)_{J\in\rho}\right)=\sum_{(J_1,\ldots, J_l)\in\operatorname{coars}(\rho)}
			\mu_{(J_1,\ldots,J_l)}
			\Bigl(\tau_{\rho\cap J_1}
			\bigl((a_{J})_{J\in\rho\cap J_1}\bigr),\ldots,
			\tau_{\rho\cap J_l}\bigl((a_{J})_{J\in\rho\cap J_l}\bigr)\Bigr)
\]
for all $\emptyset\neq I\subseteq \nset$ and all $\rho\in \mathcal P(I)$. Here, $\rho\cap J_k$ is the canonically ordered partition $\{I_s\in \rho\mid I_s\subseteq J_k\}$ for $k=1,\ldots, l$.

 \subsection{Iterated highest order cores of an $n$-fold vector bundle}\label{higher_cores}

This section studies in detail the \emph{iterated highest order cores} of an $n$-fold vector bundle. 
\begin{definition}
Let $\E\colon \square^n\to \Man$ be an $n$-fold vector bundle. Then the \textbf{highest order cores} 
$\E^{(\nset,J)}$, written simply $\E^J$, for $J\subseteq \nset$ with $\#J=2$ are also called \textbf{$(n-1)$-cores of $\E$}, since they are modelled on the $(n-1)$-cube categories $\lozenge^\nset_J$.
The \textbf{$(n-2)$-cores $\E$} are then defined to be the highest order  cores of the $(n-1)$-cores of $\E$, hence the $(n-2)$-cores of the $(n-1)$-cores of $\E$.

This construction can then be iterated: for $l\in \{1, \ldots, n-2\}$, the \textbf{$l$-cores of $\E$} are defined to be the $l$-cores of the $(l+1)$-cores of $\E$. 
Further, the (unique) \textbf{$n$-core of $\E$} is set by convention to be $\E$ itself.

The  $l$-cores for $l=1,\ldots, n-1$ are generally called the \textbf{iterated highest order cores of $\E$}.
\end{definition}

\bigskip
By construction an iterated higher order core of $\E$ is a functor from a full subcategory of $\square^n$ to $\Man$. As above, an $l$-core has a natural transformation by embeddings in the former $(l+1)$-core in its recursive construction, and all morphisms are restrictions of the morphisms of this $(l+1)$-core. Then the $l$-core has a natural transformation by embeddings in $\E$. A priori the recursively defined subcategory of $\square^n$ indexing an $l$-core is complicated to write down. In addition, two different chains of construction of $l$-cores can lead to the same $l$-core. 
The goal of the following theorem is to remedy to these problems, by understanding that  the collection of $l$-cores of $\E$, for $l\in \{1, \ldots, n\}$, is simply parametrised by the partitions of $\nset$ in $l$ subsets.
\begin{proposition}\label{l_cores_as_partitions}
Let $\E\colon \square^n\to \Man$ be an $n$-fold vector bundle and choose $l\in \{1, \ldots, n\}$.
\begin{enumerate}
\item For each $l$-core $\F$ of $\E$
 there is a partition $\rho:=\{I_1,\ldots, I_l\}$ of $\nset$ in $l$ (non-empty) subsets  such that $\F$ is a functor
from the $l$-cube subcategory $\lozenge^\rho$ of $\square^n$ to $\Man$. Conversely, for each partition $\rho$ of $\nset$ in $l$ elements there is an $l$-core 
$\F\colon\lozenge^\rho\to\Man$ of $\E$.
If two $l$-cores are defined on the same subcategory $\lozenge^\rho$, they are equal. 

\item
Choose an  $l$-partition $\rho$ of $\nset$, and consider the corresponding $l$-core $\E^\rho$ of $\E$. Then the $i$-cores of $\E^\rho$, for $i=1,\ldots, l$, are the $i$-cores of $\E$ indexed by coarsements of $\rho$.

\item Assume that $\rho=\{I_1,\ldots, I_l\}$ is a partition of $\nset$ in $l$ subsets and that $\rho_{1}$ and $\rho_{2}$ are two different $(l-1)$-coarsements of $\rho$ as in Lemma \ref{intersectionl-1}. Then the $(l-2)$-core $\E^{\rho_{1}\sqcap \rho_{2}}$ is a common $(l-2)$-core of $\E^{\rho_{1}}$ and $\E^{\rho_{2}}$.

\end{enumerate}
\end{proposition}

\begin{definition}
Let $\E\colon \square^n\to\Man$ be an $n$-fold vector bundle and let $\rho$ be a partition of $\nset$ in $l$ elements.
The unique $l$-core of $\E$ corresponding to $\rho$ as in the previous theorem is denoted by $\E^\rho$.
\end{definition}

Note that an $n$-fold vector bundle has only one $1$-core, since there is only one partition of $n$ with $1$ element. It is the ultracore. 

Note also that the $2$-cores
of $\E$ are indexed by pairs of nonempty subsets $I,J\subseteq \nset$ with $I\cap J=\emptyset$, and $I\cup J=\nset$. 
The corresponding $2$-cube category $\lozenge^{\{I,J\}}$ has then the objects $\emptyset, I, J, I\cup J=\nset$.

The characterisation above shows as well that an $n$-fold vector bundle can be naturally understood as its own $n$-core, since the partition of $\nset$ in $n$ subsets gives with \eqref{building_bundles_l_cores}  the $n$-cube category $\square^n$.

\begin{proof}[Proof of Proposition \ref{l_cores_as_partitions}]
 For $l=n$, the first statement is clearly true since there is only one partition of $\nset$ in $n$ subsets and $\E\colon\square^n\to\Man$ is its own $n$-core.

Choose a subset $J\subseteq \nset$ of cardinality $2$. 
Then by Theorem 2.20 in \cite{HeJo20}, the core $\E^{(\nset,J)}$ is an $(n-1)$-fold vector bundle
\begin{equation}\label{simplified_notation_higher_core}
 \E^J:=\E^{(\nset,J)}\colon \lozenge^{\rho_J}\to \Man,
\end{equation}
where as  before $\rho_J$ is the partition of $\nset$ with $J$ and one-element sets, hence a partition of $\nset$ with $n-1$ subsets. By definition there is a one-to-one
correspondence between  subsets of $\nset$ of cardinality $2$ and $(n-1)$-cores, hence between partitions of $\nset$ with $n-1$ elements and $(n-1)$-cores. So (1) is also true for $l=n-1$, while (2) is true for $l=n$ and $i=n-1$.
\medskip

A partition of $\nset$ in $n-2$ subsets has either one subset with $3$ elements and $n-3$ subsets with one element each, or two subsets with $2$ elements each and $n-4$ subsets with one elements each. Of course, here $n\geq 3$ and if $n=3$ only the first case is possible, while for $n\geq 4$ both cases are possible.

Consider first the first case, and assume without loss of generality that the partition is $\rho:=\{\{1,2,3\}, \{4\}, \ldots, \{n\}\}$. Consider $J_1:=\{1,2\}$ and $J_2:=\{1,3\}$
and the
two $(n-1)$-cores $\E^{J_1}$ and $\E^{J_2}$ as in \eqref{simplified_notation_higher_core}. Then $\lozenge^{\rho_{J_1}}$ is the cube category over the elements $J_1, \{3\}, \ldots, \{n\}$, and $\lozenge^{\rho_{J_2}}$ is the cube category over the elements $J_2, \{2\},\{4\}, \{5\}, \ldots, \{n\}$.

Build the $\{J_1, \{3\}\}$-core of $\E^{J_1}$, which is an $(n-2)$-core of $\E^{J_1}$ and of $\E$. By definition, it is modeled on the full subcategory $\lozenge$ 
of $\lozenge^{\rho_{J_1}}$
with objects $J\subseteq n$ consisting in unions of the sets $J_1\cup\{3\}, \{4\}, \ldots, \{n\}$. Hence $\lozenge=\lozenge^{\rho}$.
This shows the existence of an $(n-2)$-core modeled on $\rho$.

The $\{J_2, \{2\}\}$-core of $\E^{J_2}$ is then similarly modeled on the same category $\lozenge^\rho$.
These two $(n-2)$-cores are hence modeled on the same partition of $\nset$. 
Choose $I\in\Obj(\lozenge^\rho)$. If $I\cap\{1,2,3\}=\emptyset$, then also $I\cap J_1=I\cap J_2=\emptyset$ and 
\[ (\E^{J_1})^{\{J_1,\{3\}\}}(I)=\E^{J_1}(I)=\E(I)=\E^{J_2}(I)=(\E^{J_2})^{\{J_2,\{2\}\}}(I).
\]
If $\{1,2,3\}\subseteq I$ then also $J_1,J_2\subseteq I$ and so\footnote{There is a little abuse of notation here; in the formula below $I$ should formally be replaced by $\{J_1;\{t\}\mid t\in I\setminus J_1\}$ and $\{J_2;\{t\}\mid t\in I\setminus J_2\}$, depending on which chain of cores is considered. For the convenience of the reader, it is just written $I$ and understood from the context.}
\begin{equation*}
\begin{split}
 (\E^{J_1})^{\{J_1,\{3\}\}}(I)&=(\E^{J_1})^{I}_{\{J_1, \{3\}\}}
 =(c^{I}_{J_1})\inv \left(\nvbzero{I\setminus J_1}{I\setminus \{1,2,3\}}\right)\cap (p^I_{3})\inv\left(\nvbzero{I\setminus \{3\}}{I\setminus \{1,2,3\}}\right)\\
&=\E^I_{\{1,2,3\}}\\
&=(c^{I}_{J_2})\inv \left(\nvbzero{I\setminus J_2}{I\setminus \{1,2,3\}}\right)\cap (p^I_{2})\inv\left(\nvbzero{I\setminus \{2\}}{I\setminus \{1,2,3\}}\right)\\
 &=(\E^{J_2})^{I}_{\{J_2, \{2\}\}}=(\E^{J_2})^{\{J_2,\{2\}\}}(I),
\end{split}
\end{equation*}
where the third and fourth equalities are applications of \cite[Lemma 2.19]{HeJo20}.
Since the projections in each $l$-core are restrictions of the projections of the $n$-fold vector bundle $\E$, it suffices to show that the images of the two functors  $(\E^{J_1})^{\{J_1,\{3\}\}}$ and $(\E^{J_2})^{\{J_2,\{2\}\}}$ on objects of $\lozenge^\rho$ are the same to get that the two functors are equal.

\medskip

Next assume that $n\geq 4$ and consider the partition $\rho:=\{\{1,2\}, \{3,4\}, \{5\}, \ldots, \{n\}\}$ of $\nset$ in $(n-2)$ elements.
Set $J_1:=\{1,2\}$ and $J_2:=\{3,4\}$
and consider as before the
two $(n-1)$-cores $\E^{J_1}$ and $\E^{J_2}$ of $\E$. Then $\lozenge^{\rho_{J_1}}$ is the square category over the elements $J_1, \{3\}, \ldots, \{n\}$, and $\lozenge^{\rho_{J_2}}$ is the square category over the elements $\{1\}, \{2\}, J_2, \{5\},\{6\}, \ldots, \{n\}$.

Build the $\{\{3\},\{4\}\}$-core of $\E^{J_1}$, which is an $(n-2)$-core of $\E^{J_1}$ and of $\E$. By definition, it is modeled on the full subcategory $\lozenge$ 
of $\lozenge^{\rho_{J_1}}$
with objects $I\subseteq n$ consisting in unions of the sets $J_1, J_2, \{5\}, \{6\}, \ldots, \{n\}$. Hence $\lozenge=\lozenge^{\rho}$.
This shows the existence of an $(n-2)$-core modeled on $\rho$.

The $\{\{1\}, \{2\}\}$-core of $\E^{J_2}$ is then similarly modeled on the same category $\lozenge^\rho$.
These two $(n-2)$-cores are hence modeled on the same partition of $\nset$. In order to show that 
\[ (\E^{J_1})^{\{\{3\},\{4\}\}}=(\E^{J_2})^{\{\{1\},\{2\}\}}\colon\lozenge^\rho\to\Man,
\]
it is again enough to show that the functors are equal on objects of $\lozenge^\rho$.
First take $I\in\Obj(\lozenge^\rho)$ with $I\cap J_1=\emptyset$ and $I\cap J_2=\emptyset$. Then
\[(\E^{J_1})^{\{\{3\},\{4\}\}}(I)=\E^{J_1}(I)=\E(I)=\E^{J_2}(I)=(\E^{J_2})^{\{\{1\},\{2\}\}}(I).
\]
If $I\in\Obj(\lozenge^\rho)$ does $J_1\subseteq I$ and $I\cap J_2=\emptyset$, then
\[(\E^{J_1})^{\{\{3\},\{4\}\}}(I)=\E^{J_1}(I)=E^I_{J_1}\]
while 
\[(\E^{J_2})^{\{\{1\},\{2\}\}}(I)=(\E^{J_2})^I_{\{\{1\},\{2\}\}}=E^I_{J_1}
\]
since the $I$-face of $\E^{J_2}$ is the $I$-face of $\E$.
Finally if $J_1\cup J_2\subseteq I$, then 
\begin{equation*}
\begin{split}
(\E^{J_1})^{\{\{3\},\{4\}\}}(I)&=(\E^{J_1})^I_{\{\{3\},\{4\}\}}\\
&=\left\{ e\in \E^{J_1}(I) \left| p^I_{I\setminus\{3\}}(e)=\nvbzero{I\setminus\{3\}}{p^I_{I\setminus J_2}(e)} \text{ and } p^I_{I\setminus\{4\}}(e)=\nvbzero{I\setminus\{4\}}{p^I_{I\setminus J_2}(e)}\right.\right\}\\
&=\left\{ e\in \E(I) \left| 
\begin{array}{c}p^I_{I\setminus\{1\}}(e)=\nvbzero{I\setminus\{1\}}{p^I_{I\setminus J_1}(e)},\quad  p^I_{I\setminus\{2\}}(e)=\nvbzero{I\setminus\{2\}}{p^I_{I\setminus J_1}(e)}\\
p^I_{I\setminus\{3\}}(e)=\nvbzero{I\setminus\{3\}}{p^I_{I\setminus J_2}(e)} \text{ and } p^I_{I\setminus\{4\}}(e)=\nvbzero{I\setminus\{4\}}{p^I_{I\setminus J_2}(e)}\end{array}\right.\right\}\\
&=(\E^{J_2})^{\{\{1\},\{2\}\}}(I),
\end{split}
\end{equation*}
since the description in the third line is symmetric in $J_1$ and $J_2$.

\medskip

Given an arbitrary $(n-1)$-core $\E^J$ for $J\subset\nset$ with $\#J=2$, its $(n-2)$-cores are either 
$(\E^J)^{\{\{i\},\{j\}\}}$ for $i,j\in\nset\setminus J$ or $(\E^J)^{\{J,\{i\}\}}$ for $i\in\nset\setminus J$. As explained above, in the first case the $(n-2)$-core is indexed by the partition $\{J,\{i,j\}; \{t\}\mid t\in\nset\setminus(J\cup\{i,j\})\}$, and in the second case on the partition $\{J\cup\{i\};  \{t\}\mid t\in\nset\setminus(J\cup\{i\})\}$. 

Hence (1) is proved for $l=n-2$. (2) is true for $l=n$ and $i=n-1,n-2$, and $l=n-1$ and $i=n-2$. (3) for $l=n$ was proved above at the same time as the equality of the two $(n-2)$-cores indexed by the same partition, since in the first case \[
\{\{1,2,3\}, \{4\}, \ldots, \{n\}\}=\{J_1, \{3\}, \ldots, \{n\}\}\sqcap\{J_2, \{2\},\{4\}, \{5\}, \ldots, \{n\}\}.
\]
and in the second case
\[\{\{1,2\}, \{3, 4\}, \{5\}, \ldots, \{n\}\}=\{J_1, \{3\}, \ldots, \{n\}\}\sqcap\{\{1\}, \{2\}, J_2, \{5\},\{6\}, \ldots, \{n\}\}.
\]
\bigskip

The proof now works recursively since an $l$-core of $\E$ is always an $l$-core of an $(l+1)$-core of $\E$:
Assume that (1) is true for some fixed $l\in \{2, \ldots, n\}$. 
Then an $(l-1)$-core $\F$ of $\E$ is an $(l-1)$-core of an $l$-core of $\E$, hence of an $l$-core $\E^\rho$ for some partition $\rho=\{I_1, \ldots, I_l\}$ of $\nset$.  The $(l-1)$-core $\F$ is then defined by a choice of $i< j$ in $\{1, \ldots, l\}$ such that  
\[ \F=(\E^\rho)^{\{I_i, I_j\}}.
\]
$\F$ is then modeled on 
$\rho'=\{I_i\cup I_j, I_1, \ldots, \widehat i, \ldots, \widehat j, \ldots, I_l\}$, which is an $(l-1)$-partition of $\nset$.
Conversely, choose a partition $\rho=\{I_1,\ldots, I_{l-1}\}$ of $\nset$ in $(l-1)$ subsets. Then since $l-1\in\{1, \ldots, n-2\}$
one of the subsets $I_1, \ldots, I_{l-1}$ of $\nset$ must contain more than one element. Without generality, $I_{l-1}$ does. Write $I_{l-1}=J_{l-1}\cup J_l$
with $J_{l-1}, J_l\subseteq \nset$ disjoint and non-empty. Set $\rho'=\{I_1,\ldots, I_{l-2}, J_{l-1}, J_l\}$. Then 
\[ (\E^{\rho'})^{\{J_{l-1},J_l\}}
\]
is an $(l-1)$-core of $\E$ indexed by $\rho$. Computations as above show that this $(l-1)$-core of $\E$ does not depend on the choice of $\rho'$ above  $\rho$.
Precisely, the considerations above show that for $I=I_{i_1}\cup\ldots\cup I_{i_k}\in\Obj(\lozenge^\rho)$,
\begin{equation}\label{explicit_iterated_core}
\begin{split}
\E^\rho(I)
&=\left\{ e\in \E(I) \left| \begin{array}{c}
\text{ For } s=1,\ldots,k \text{ and all } j\in I_{i_s}:\\
p^{I}_{I\setminus\{j\}}(e)=\nvbzero{I\setminus\{j\}}{p^{I}_{I\setminus I_{i_s}}(e)}.
\end{array}\right.\right\}\\
&=\bigcap_{s=1}^k\bigcap_{j\in I_{i_s}}(p^I_{I\setminus\{j\}})\inv\left(
	\nvbzero{I\setminus \{j\}}{I\setminus I_{i_s}}\right)\
\end{split}
\end{equation}
which does not depend on the choice of $\rho'$. Therefore (1) holds for all $l=1,\ldots, n$. By construction (2) holds as well for all $l=1, \ldots, n$ and all $i=1, \ldots, l$, and (3) follows again from the one-to-one correspondence between $l$-partitions and $l$-cores, for all $l=1,\ldots, n$.
\end{proof}

\begin{corollary}
Let $\E\colon \square^n\to\Man$ be an $n$-fold vector bundle and let $\rho=\{I_1,\ldots, I_l\}$ be a partition of $\nset$ in $l$ elements.
The lower sides of $\E^\rho$ are the vector bundles $\E^{I_i}_{I_i}\to M$ for $i=1, \ldots, l$, and the other building bundles of $\E^\rho$ are the bundles 
$\E_J^J$ for  $J\in\operatorname{Obj}(\lozenge^\rho)$.
\end{corollary}

\begin{proof}
The first statement is clear since $\lozenge^\rho$ is the $l$-cube category over the elements $I_1,\ldots, I_l$.
It is also a special case of the second statement.
For the second statement, choose $J\in\Obj(\lozenge^\rho)$, without loss of generality $J=I_1\cup \ldots \cup I_k$ for $1\leq k\leq l$ and compute with \eqref{explicit_iterated_core}
\begin{equation*}
\begin{split}
(\E^\rho)^J_J&=\left\{ e\in\E^\rho(J)\left|
\text{ for all } s=1,\ldots, k, \,\, p^J_{J\setminus I_ s}(e)=\nvbzero{J\setminus I_s}{p^J_\emptyset(e)}
\right.
\right\}\\
&=\left\{ e\in \E(J) \left| \begin{array}{c}
\text{ For } s=1,\ldots,k \text{ and all } j\in I_{s}:\\
p^{J}_{J\setminus\{j\}}(e)=\nvbzero{J\setminus\{j\}}{p^{J}_{J\setminus J_{s}}(e)}\\
\text{ for all } s=1,\ldots, k, \,\, p^J_{J\setminus I_ s}(e)=\nvbzero{J\setminus I_s}{p^J_\emptyset(e)}
\end{array}\right.\right\}\\
&=\left\{ e\in \E(J) \left| \begin{array}{c}
\text{ For all } j\in J:\,\,
p^{J}_{J\setminus\{j\}}(e)=\nvbzero{J\setminus\{j\}}{p^{J}_{\emptyset}(e)}\\
\end{array}\right.\right\}=\E^J_J.
\end{split}
\end{equation*}\qedhere
\end{proof}

\begin{lemma}\label{second_result_iterated_cores}
Let $\E\colon \square^n\to\Man$ be an $n$-fold vector bundle and let $\rho$ be a partition of $\nset$ in $l$ elements.
 \begin{enumerate}

\item For each $J\in \operatorname{Obj}(\lozenge^\rho)$ the manifold $\mathbb E^\rho(J)$ is an embedded submanifold of $\E(J)$ and the collection of these embeddings 
defines a natural transformation
\[ \tau^\rho\colon \E^\rho\longrightarrow \E\circ i^\rho
\]
since 
\item for each arrow $(J\rightarrow J') \in  \operatorname{Mor}(\lozenge^\rho)$ the arrow $\E^{\rho}(J\rightarrow J')$ is the restriction\footnote{Therefore, for simplicity, the 
restriction $\E^\rho(J\to J') \colon \E^\rho(J)\to \E^\rho(J')$ of $\E(J\to J')$ is written as well $p^J_{J'}$ in the following.}
 of $\E(J\rightarrow J')=p^J_{J'}\colon \E(J)\to \E(J')$
to the domain and codomain $\E^\rho(J)\rightarrow \E^\rho(J')$.
\item Let $\F\colon  \square^n\to\Man$ be a second $n$-fold vector bundle and $\Phi\colon \E\to \F$ a morphism of $n$-fold vector bundles.
Then $\Phi$ restricts to a morphism $\Phi^\rho\colon \E^\rho\to \F^\rho$ of $l$-fold vector bundles. Precisely, for each $J\in \operatorname{Obj}(\lozenge^\rho)$
\[ \Phi^\rho(J)=\Phi(J)\an{\E^\rho(J)}\colon \E^\rho(J)\to \F^\rho(J)\subseteq \F(J)
\]
is the restriction on the embedded submanifold $\E^\rho(J)$ of the domain $\E(J)$ of $\Phi(J)$ and with restricted codomain the embedded submanifold $\F^\rho(J)\subseteq \F(J)$.
\end{enumerate}
\end{lemma}

\begin{proof}
This is immediate from the definition of the highest order cores in \eqref{def_objects_core} and \eqref{def_arrows_core} and the recursive construction of the iterated highest order cores, as well as the recursive restriction of morphisms as in \eqref{core_morphism}.
\end{proof}

\begin{example}

Consider as before an $n$-fold vector bundle $\E$ and the induced collection of vector bundles 
\[ \mathcal A=\left\{\E_K^K\mid K\subseteq \nset\right\}.
\]
A partition $\rho$ of $\nset$ defines then a $\#\rho$-fold vector bundle 
\[ \E^{\mathcal A, \rho}\colon \lozenge^\rho\to \Man, \qquad \E^{\mathcal A, \rho}(J):=\prod^M_{\substack{I\subseteq J\\ I\in \operatorname{Obj}(\lozenge^\rho)}}\E^I_I
\]
for $J\in \operatorname{Obj}(\lozenge^\rho)$,
and a vacant $\#\rho$-fold vector bundle 
\[ \overline{\E^{\rho}}\colon \lozenge^\rho\to \Man, \qquad  \overline{\E^{\rho}}(J):=\prod^M_{\substack{I\subseteq J\\ I\in \rho}}\E^I_I
\]
for $J\in \operatorname{Obj}(\lozenge^\rho)$.
Here $\overline{\E^{\rho}}(J)$  is clearly embedded in $\E^{\mathcal A, \rho}(J)$, which is embedded in $\E^{\mathcal A}(J)$ for all $J\in \operatorname{Obj}(\lozenge^\rho)$ and the collection of these embeddings define natural transformations
\begin{equation*}
\begin{split}
\overline{\E^{\rho}}\overset{\overline{\tau}^\rho}{\longrightarrow}\E^{\mathcal A, \rho}\overset{\tau^{\mathcal A, \rho}}{\longrightarrow} \E^{\mathcal A}\circ i^\rho.
\end{split}
\end{equation*}
(Recall that $i^\rho\colon \lozenge^\rho\to \square^n$ is the inclusion functor.)

\medskip
The $l$-cores of $\E^{\mathcal A}$ coincide with the functors $\E^{\mathcal A, \rho}\colon \lozenge^\rho\to \Man$ given by partitions $\rho$ of $\nset$ with $l$ elements. 
A decomposition $\mathcal S\colon \E^{\mathcal A}\to \E$ of $\E$ restricts as in Lemma \ref{second_result_iterated_cores} to a decomposition $\mathcal S^\rho\colon \E^{\mathcal A,\rho}\to \E^\rho$ of the $l$-core $\E^\rho$ such that the following diagram of natural transformations commutes.

% https://q.uiver.app/?q=WzAsNCxbMCwwLCJcXG1hdGhiYiBFXntcXG1hdGhjYWwgQX1cXGNpcmMgaV5cXHJobyJdLFswLDIsIlxcbWF0aGJiIEVee1xcbWF0aGNhbCBBLCBcXHJob30iXSxbMiwyLCJcXG1hdGhiYiBFXlxccmhvIl0sWzIsMCwiXFxtYXRoYmIgRVxcY2lyYyBpXlxccmhvIl0sWzEsMiwiXFxtYXRoY2FsIFNee1xccmhvfSIsMix7ImxldmVsIjoyfV0sWzAsMywiXFxtYXRoY2FsIFNcXGFycm93dmVydF97XFxsb3plbmdlXlxccmhvfSIsMCx7ImxldmVsIjoyfV0sWzEsMCwiXFx0YXVee1xcbWF0aGNhbCBBLCBcXHJob30iLDAseyJsZXZlbCI6Mn1dLFsyLDMsIlxcdGF1XlxccmhvIiwyLHsibGV2ZWwiOjJ9XV0=
\[\begin{tikzcd}
	{\mathbb E^{\mathcal A}\circ i^\rho} && {\mathbb E\circ i^\rho} \\
	\\
	{\mathbb E^{\mathcal A, \rho}} && {\mathbb E^\rho}
	\arrow["{\mathcal S^{\rho}}"', Rightarrow, from=3-1, to=3-3]
	\arrow["{\mathcal S\arrowvert_{\lozenge^\rho}}", Rightarrow, from=1-1, to=1-3]
	\arrow["{\tau^{\mathcal A, \rho}}", Rightarrow, from=3-1, to=1-1]
	\arrow["{\tau^\rho}"', Rightarrow, from=3-3, to=1-3]
\end{tikzcd}\]
\end{example}

Consider an $n$-fold vector bundle $\E$. Choose a partition $\rho=\{I_1,\ldots, I_l\}$ of $\nset$ in $l$ subsets and consider two distinct coarsements $\rho_1$ and $\rho_2$ of $\rho$ in $(l-1)$ subsets.
Let $\mathcal S^i\colon \E^{\mathcal A, \rho_i}\to \E^{\rho_i}$ be a decomposition of the $(l-1)$-core $\E^{\rho_i}$ for $i=1,2$.
The two decompositions $\mathcal S^{\rho_1}$ and $\mathcal S^{\rho_2}$ are \textbf{compatible} if
\begin{equation}\label{compatible_dec}
(\mathcal S^{\rho_1})^{\rho_1\sqcap\rho_2}=(\mathcal S^{\rho_2})^{\rho_1\sqcap\rho_2}\colon \E^{\mathcal A, \rho_1\sqcap \rho_2}\rightarrow \E^{\rho_1\sqcap \rho_2}.
\end{equation}
Take now more precisely the coarsement $\underline{\rho}:=\{I_1\cup I_2, I_3, \ldots, I_4\}$ of $\rho$. Then $(\E^\rho)_{\{I_3,\ldots, I_4\}}$ is a face of $\E^\rho$ and of $\E^{\underline{\rho}}$.
Consider a linear splitting $\Sigma\colon \overline{\E^\rho}\to \E^\rho$ of $\E^\rho$. Then $\Sigma$ \textbf{is compatible with} a decomposition $\mathcal S^{\underline{\rho}}$ of $\E^{\underline{\rho}}$ if the following diagram commutes
% https://q.uiver.app/?q=WzAsNCxbMCwwLCIoXFxvdmVybGluZXtcXEV9XntcXHJob19KfSlfe1xcbnNldFxcc2V0bWludXMgSn0iXSxbMCwyLCJcXG92ZXJsaW5le1xcRX1fe1xcbnNldFxcc2V0bWludXMgSn0iXSxbMiwyLCJcXEVfe1xcbnNldFxcc2V0bWludXMgSn0iXSxbMiwwLCIoXFxFXntcXHJob19KfSlfe1xcbnNldFxcc2V0bWludXMgSn0iXSxbMywyLCIiLDAseyJsZXZlbCI6Miwic3R5bGUiOnsiaGVhZCI6eyJuYW1lIjoibm9uZSJ9fX1dLFswLDEsIiIsMCx7ImxldmVsIjoyLCJzdHlsZSI6eyJoZWFkIjp7Im5hbWUiOiJub25lIn19fV0sWzAsMywiKFxcbWF0aGNhbCBTXkpcXGNpcmMgXFx0YXVee1xccmhvX0p9KVxcYW57XFxuc2V0XFxzZXRtaW51cyBKfSJdLFsxLDIsIlxcU2lnbWFcXGFue1xcbnNldFxcc2V0bWludXMgSn0iLDJdXQ==
\begin{equation}\label{comp_sigma_S}
\begin{tikzcd}
	{(\overline{\E^{\underline{\rho}}})_{\{I_3, \ldots, I_4\}}} &&& {(\E^{\underline{\rho}})_{\{I_3, \ldots, I_4\}}} \\
	\\
	{(\overline{\E^\rho})_{\{I_3, \ldots, I_4\}}} &&& {(\E^\rho)_{\{I_3, \ldots, I_4\}}}
	\arrow[Rightarrow, no head, from=1-4, to=3-4]
	\arrow[Rightarrow, no head, from=1-1, to=3-1]
	\arrow["{(\mathcal S^{\underline{\rho}}\circ \iota)\an{\{I_3, \ldots, I_4\}}}", from=1-1, to=1-4]
	\arrow["{\Sigma\an{\{I_3, \ldots, I_4\}}}"', from=3-1, to=3-4]
\end{tikzcd}
\end{equation}
where $\iota$ is the canonical natural transformation
\[ \iota\colon \overline{\E^{\underline\rho}}\to \E^{\mathcal A, \underline\rho}.
\]
The compatibility of $\Sigma$ with a decomposition of $\E^{\underline{\rho}}$ for any other coarsement $\underline{\rho}$ of $\rho$ with $(l-1)$ elements is defined similarly.

\subsection{Proof of Corollary 3.6 in \cite{HeJo20}}\label{fix_of_3.6}
The following theorem is proved in \cite[Theorem 3.3]{HeJo20}.

\begin{theorem}\label{thm_split_dec_n_non_sym}
\begin{enumerate}
\item[(a)] Let $\mathcal S$ be a decomposition of an $n$-fold vector 
bundle $\E\colon\square^n\to \Man$. Then the composition 
$\Sigma=\mathcal S\circ\tau\colon\overline{\E}\to\E$, with 
$\tau$ defined as in \eqref{emb_spl_dec}, is a splitting of $\E$. 
Furthermore, the core morphisms $\S^{\rho_J}\colon\E^{\mathcal A, \rho_{J}}\to\E^{\rho_J}$
are decompositions of $\E^{\nset}_J$ for all $J\subseteq\nset$ with $\# J=2$ and 
these decompositions and the linear splitting are compatible.

\item[(b)] Conversely, given a linear splitting $\Sigma$ of $\E$ and 
	compatible\footnote{Pairwise compatible and all compatible with $\Sigma$.} decompositions of the $(n-1)$-cores $\E^{\rho_J}$ 
        for $J\subseteq \nset$ with $\# J=2$, there 
	exists a unique decomposition $\mathcal S$ of $\E$ such that 
	$\Sigma=\mathcal S\circ\tau$ and such that the core morphisms 
	of $\S$ are given by $\S^{\rho_J}=\S^J$ for all $J$. 
\end{enumerate}
\end{theorem}
A symmetric version of this result is proved later on (see Proposition \ref{thm_split_dec_n}) and the details of the proof of this theorem are discussed in Appendix \ref{dec_splittings}.

Theorem 3.5 in \cite{HeJo20} states then that
  for each $n$-fold vector bundle $\E$, there is a linear
  splitting
  \[\Sigma\colon \overline{\E} \to \E\,,
    \]
    that is a monomorphism of $n$-fold vector bundles from the vacant,
    decomposed $n$-fold vector bundle $\overline{\E}$ associated to $\E$.		
	\cite{HeJo20} shows this result by proving inductively (over $n$) the following two claims:
\begin{enumerate}
\item[(a)] Given an $n$-fold vector bundle $\E$, there exist $n$
  linear splittings $\Sigma_{\nset\setminus\{k\}}$ of
  $\E_{\nset\setminus\{k\}}$ for $k\in\nset$, such that
  $\Sigma_{\nset\setminus\{i\}}(I)=\Sigma_{\nset\setminus\{j\}}(I)$
  for any $I\subseteq\nset\setminus\{i,j\}$, i.e.~such that
  \[\Sigma_{\nset\setminus\{i\}}\an{\nset\setminus\{i,j\}}=\Sigma_{\nset\setminus\{j\}}\an{\nset\setminus\{i,j\}}
  \]
  for all $i,j\in\nset$.
\item[(b)] Given a family of splittings as in (a), there exists a linear
  splitting $\Sigma$ of $\E$ with \[\Sigma\an{\nset\setminus\{k\}}=\Sigma_{\nset\setminus\{k\}}\]
 for each $k\in\nset$.
  \end{enumerate}

The following proposition arises as a corollary of the second claim, which is the missing piece to the proof of Corollary 3.6 in \cite{HeJo20}.
\begin{proposition}\label{missing_step}
Let $\E\colon\square^n\to \Man$ be an $n$-fold vector bundle. Assume that each $(n-1)$-core $\E^{\rho_J}\colon \lozenge^{\rho_J}\to\Man$ of $\E$, for $J\subseteq \nset$ with $\#J=2$, has 
a decomposition $\mathcal S^J\colon \E^{\mathcal A, \rho_J}\to \E^{\rho_J}$, such that for all $J_1,J_2\subseteq \nset$ with $\#J_1=\#J_2=2$ the decompositions
$\mathcal S^{J_1}$ and $\mathcal S^{J_2}$ are compatible as in \eqref{compatible_dec}.
Then there exists a splitting
\[ \Sigma\colon \overline{\E}\to \E
\]
of $\E$ that is compatible with all decompositions $\mathcal S^J$ as in \eqref{comp_sigma_S}: for all $J\subseteq \nset$ with $\#J=2$,
\begin{equation}\label{compatible_splitting_construction}
 \Sigma\an{\nset\setminus J}=\mathcal S^J\an{\nset\setminus J}\circ\tau\an{\nset\setminus J}\colon \overline{\E}_{\nset\setminus J}\to \E_{\nset\setminus J}.
\end{equation}
\end{proposition}

\begin{proof}For each $J\subseteq \nset$ with $\# J=2$, the decomposition $\mathcal S^J$ of $\E^{\rho_J}$ defines a decomposition $\S^J\an{\nset\setminus J}$ of $\E_{\nset\setminus J}$ since $\E_{\nset\setminus J}$ is a face of $\E^{\rho_J}$.
Denote by $\Sigma^{J}\colon \overline{\E}_{\nset\setminus J}\to \E_{\nset\setminus J}$ the induced linear splitting of $\E_{\nset\setminus J}$, i.e.
\[ \Sigma^J:=\mathcal S^J\an{\nset\setminus J}\circ\tau\an{\nset\setminus J}\colon \overline{\E}_{\nset\setminus J}\rightarrow \E_{\nset\setminus J}.
\]

Consider $i\in \nset$. Then the $(n-2)$-fold vector bundles $\E_{\nset\setminus J}$, for $i\in J\subseteq \nset$  with $\#J=2$,  are sides of the $(n-1)$-fold 
vector bundle $\E_{\nset\setminus \{i\}}$. 
Choose two such subsets $J_1$ and $J_2$. Then $\rho_{J_1}\sqcap\rho_{J_2}=\{J_1\cup J_2; \{t\}\mid t\in\nset\setminus(J_1\cup J_2)\}$.
Since $\mathcal S^{J_1}$ and $\mathcal S^{J_2}$ coincide on the common core $\E^{\rho_{J_1}\sqcap \rho_{J_2}}$ of $\E^{J_1}$ and $\E^{J_2}$, 
they coincide in particular on its face $\E_{\nset\setminus (J_1\cup J_2)}$.
Hence $\mathcal S^{J_1}\an{\nset\setminus(J_1\cup J_2)}=\mathcal S^{J_2}\an{\nset\setminus(J_1\cup J_2)}$  and the induced splittings $\Sigma^{J_1}$ and $\Sigma^{J_2}$ satisfy consequently
\[ \Sigma^{J_1}(I)=\Sigma^{J_2}(I)
\]
for all $I\subseteq (\nset\setminus J_1)\cap(\nset\setminus J_2)=\nset\setminus(J_1\cup J_2)$. 

That is, the splittings $\Sigma^J$ for $\#J=2$ and $i\in J$ satisfy (a) above. By (b) there is consequently a linear splitting 
\[\Sigma_{\nset\setminus\{i\}}\colon \overline{\E}_{\nset\setminus\{i\}}\to \E_{\nset\setminus\{i\}}
\]
of $\E_{\nset\setminus\{i\}}$, such that
\[ \Sigma_{\nset\setminus\{i\}}\an{\nset\setminus J}=\Sigma^J.
\]
for all $J\subseteq \nset$ with $\#J=2$ and $i\in J$. Now choose $K\subseteq (\nset\setminus \{i\})\cap (\nset\setminus\{j\})=\nset\setminus\{i,j\}$ for $i\neq j\in \nset$. Set $I:=\{i,j\}$. Then 
\[ \Sigma_{\nset\setminus\{i\}}(K)=\Sigma^I(K)=\Sigma_{\nset\setminus \{j\}}(K)
\]
shows that (a) then still holds for the $(n-1)$-sides $\E_{\nset\setminus\{i\}}$ of $\E$, for all $i\in \nset$. As a consequence, there exists with (b) a linear splitting 
\[ \Sigma\colon \overline{\E}\to\E
\]
of $\E$ with 
\[ \Sigma\an{\nset\setminus\{i\}}=\Sigma_{\nset\setminus\{i\}}
\]
for $i=1,\ldots, n$. The compatibility in \eqref{compatible_splitting_construction} with the decompositions of the highest order cores is immediate by construction.
\end{proof}

Now the complete proof of \cite[Corollary 3.6]{HeJo20}
can be given.
\begin{corollary}\label{existence_dec_non_sym}
	Every $n$-fold vector bundle $\E$ is non-canonically isomorphic to the
	associated decomposed $n$-fold vector bundle $\E^{\mathcal A}$.	
	\end{corollary}
	\begin{proof}

	First consider all $2$-cores of $\E$. As explained in Proposition \ref{l_cores_as_partitions}, these are indexed by all possible partitions of $\nset$ in two (non-empty) subsets.
	Consider such a partition $\rho=\{I, \nset\setminus I\}$ for $\emptyset\neq I\subsetneq \nset$
and choose a linear decomposition
 % https://q.uiver.app/?q=WzAsOCxbMCwwLCJFX2xcXHRpbWVzIEVfe24tbH1cXHRpbWVzIEVfe1xcbnNldH0iXSxbMCwyLCJFX2wiXSxbMSwxLCJFX3tuLWx9Il0sWzEsMywiTSJdLFsyLDAsIlxcRV57XFxyaG9fbH0iXSxbMiwyLCJFX2wiXSxbMywxLCJFX3tuLWx9Il0sWzMsMywiTSJdLFswLDJdLFswLDFdLFsxLDNdLFsyLDNdLFswLDQsIlxcbWF0aGNhbCBTXntcXHJob19sfSJdLFs0LDZdLFs0LDVdLFs1LDddLFs2LDddLFsyLDZdLFszLDddLFsxLDVdXQ==
\[\begin{tikzcd}
	{\E_I^I\times_M \E^{\nset\setminus I}_{\nset\setminus I}\times_M \E_{\nset}^{\nset}} && {\E^{\rho}} \\
	& {\E_{\nset\setminus I}^{\nset\setminus I}} && {\E_{\nset\setminus I}^{\nset\setminus I}} \\
	{\E_I^I} && {\E_I^I} \\
	& M && M
	\arrow[from=1-1, to=2-2]
	\arrow[from=1-1, to=3-1]
	\arrow[from=3-1, to=4-2]
	\arrow[ from=2-2, to=4-2]
	\arrow["{\mathcal S^{\rho}}", from=1-1, to=1-3]
	\arrow[from=1-3, to=2-4]
	\arrow[ from=1-3, to=3-3]
	\arrow[from=3-3, to=4-4]
	\arrow[from=2-4, to=4-4]
	\arrow["{\id\qquad}", from=2-2, to=2-4]
	\arrow["{\id}", from=4-2, to=4-4]
	\arrow["{\id }", from=3-1, to=3-3]
\end{tikzcd}\]
of the $2$-core $\E^{\rho}$, which is a double vector bundle with sides $\E_I^I$ and $\E^{\nset\setminus I}_{\nset\setminus I}$ and with core $\E_{\nset}^{\nset}$.
(Double vector bundles always admit decompositions, see  \cite{GrRo09,delCarpio-Marek15}. See also the introduction of \cite{HeJo20} for more historical details.)

Choose now for each partition $\rho=\{I_1,I_2,I_3\}$ of $\nset$ in three subsets 
a linear splitting $\Sigma^\rho\colon \overline{\E^{\rho}}\to \E^\rho$ of the 3-core $\E^\rho$ of $\E$. 
Then $\Sigma^\rho$ is \emph{automatically} compatible with the decompositions 
\[ \mathcal S^{\{I_1\cup I_2, I_3\}}, \quad \mathcal S^{\{I_1\cup I_3, I_2\}}, \quad \text{ and }\quad \mathcal S^{\{I_2\cup I_3, I_1\}} 
\]
as in \eqref{comp_sigma_S}
 since \eqref{comp_sigma_S} 
reads here
\[ \Sigma^\rho(I_i)=\mathcal S^{\{I_j\cup I_k, I_i\}}(I_i)\colon \E_{I_i}^{I_i}\to \E_{I_i}^{I_i},
\]
for $\{i,j,k\}=\{1,2,3\}$, 
which is immediate since both $\Sigma^\rho(I_i)$ and $\mathcal S^{\{I_j\cup I_k, I_i\}}(I_i )$ are the identity on $\E_{I_i}^{I_i}$.
Secondly,  the decompositions $\mathcal S^{\{I_1\cup I_2, I_3\}}$, $\mathcal S^{\{I_1\cup I_3, I_2\}}$ and $\mathcal S^{\{I_2\cup I_3, I_1\}}$ of the $2$-cores of $\E^\rho$ are compatible since e.g.~\[
\{I_1\cup I_2, I_3\}\sqcap \{I_1\cup I_3, I_2\}=\{\nset\},
\]
so \eqref{compatible_dec}  reduces here to 
\[\left(\mathcal S^{\{I_1\cup I_2, I_3\}}\right)^{\{\nset\}}=
\left( \mathcal S^{\{I_1\cup I_3, I_2\}}\right)^{\{\nset\}}\colon \E_{\nset}^{\nset}\to \E_{\nset}^{\nset},
\]
which again is immediate since both decompositions restrict to the identity on the common core $\E_{\nset}^{\nset}$ of $\E^{\{I_1\cup I_2, I_3\}}$ and  $\E^{\{I_1\cup I_3, I_2\}}$. 

By Theorem \ref{thm_split_dec_n_non_sym},
there exists a decomposition $\mathcal S^\rho$ of $\E^\rho$ that restricts to $\Sigma^\rho$ and to $\mathcal S^{\{I_1\cup I_2, I_3\}}$, $\mathcal S^{\{I_1\cup I_3, I_2\}}$ and $\mathcal S^{\{I_2\cup I_3, I_1\}}$ on the $2$-cores of $\E^\rho$. Since $\rho$ was an arbitrary partition of $\nset$ in three subsets, this yields decompositions of all $3$-cores, such that the restrictions to two refinements of a same $2$-core are automatically compatible:
If $\rho_1$ and $\rho_2$ are two 3-partitions of $\nset$ such that $\rho_1\sqcap\rho_2$ is a 2-partition of $\nset$, then $\rho_1\sqcap \rho_2$ indexes the common $2$-core of $\E^{\rho_1}$ and $\E^{\rho_2}$, and by the construction above of $\S^{\rho_1}$ and $\S^{\rho_2}$
\[ \left(\mathcal S^{\rho_1}\right)^{\rho_1\sqcap\rho_2}=\mathcal S^{\rho_1\sqcap\rho_2}=\left(\mathcal S^{\rho_2}\right)^{\rho_1\sqcap\rho_2}.
\]
Use Proposition \ref{missing_step} and choose linear splittings of the $4$-cores which are compatible with all decompositions of the corresponding $3$-cores. By Theorem \ref{thm_split_dec_n_non_sym} there are decompositions of the $4$-cores with are compatible with the chosen splittings and with all decompositions of the $3$-cores constructed above. As above, the decompositions of the $4$-cores are therefore automatically compatible, and there exist by Proposition \ref{missing_step} linear splittings of the $5$-cores that are compatible with the decompositions of the $4$-cores. Repeat these steps until a decomposition of the $n$-core $\E$ is constructed.
        	\end{proof}

\subsection{Atlases of multiple vector bundles}

\cite{HeJo20} shows that an $n$-fold vector bundle can be equivalently 
defined as a smooth manifold endowed with an $n$-fold vector bundle atlas.

The following definition \cite{HeJo20} is a straightforward generalisation of Pradines' definition of a double vector bundle atlas \cite{Pradines77}.
\begin{definition}\label{n_fold_vb_atlas_def}
  Let $M$ be a smooth manifold and $E$ a topological space
  together with a continuous map $\pi\colon E\to M$.  An
  \textbf{n-fold vector bundle chart on $E$} is a tuple
  \[c=(U,\Theta,(V_I)_{\emptyset\neq I\subseteq \nset}),\] where $U$ is an open
  set in $M$, for each $\emptyset\neq I\subseteq \nset$ the space $V_I$ is a fixed (finite
  dimensional) real vector space,  and
  $\Theta\colon \pi\inv(U)\to U \times \prod_{\emptyset\neq I\subseteq \bar{n}}V_I$
  is a homeomorphism such that $\pi=\pr_1\circ\Theta$.

Two $n$-fold vector bundle charts $c$ and $c'$ are
\textbf{smoothly compatible} if the ``change of chart''
$\Theta'\circ\Theta\inv$ over $U\cap U'$ has the following form\footnote{Hence if it is a morphism of decomposed $n$-fold vector bundles, see \cite{HeJo20}.}:
\begin{equation*}\label{change_of_charts}
  \bigl(p,(v_I)_{\emptyset\neq I\subseteq\nset}\bigr)\mapsto
  \left(p,\left(\sum_{\rho=(I_1,\ldots,I_k)\in\mathcal P(I)}\omega_{\rho}(p)(v_{I_1},\ldots,v_{I_k})\right)_{\emptyset\neq I\subseteq\nset}\right)
\end{equation*}
with $p\in U\cap U'$, $v_I\in V_I$ and
$\omega_\rho\in C^\infty(U\cap U',\Hom(V_{I_1}\otimes\ldots\otimes
V_{I_k},V_{I}))$ for $\rho=(I_1,\ldots,I_k)\in\mathcal P(I)$.

A \textbf{smooth n-fold vector bundle atlas} $\lie A$ on $E$
is a set of n-fold vector bundle charts of $E$ that are
pairwise smoothly compatible and such that the set of underlying open
sets in $M$ is a covering of $M$.  As usual, $E$ is then a
smooth manifold and two smooth $n$-fold vector bundle atlases $\lie A_1$
and $\lie A_2$ are \textbf{equivalent} if their union is a smooth
n-fold vector bundle atlas.  A smooth \textbf{$n$-fold vector bundle structure}
on $E$ is an equivalence class of smooth n-fold vector bundle
atlases on $E$. The pair of $E$ and a smooth $n$-fold vector bundle structure on $E$ is written $\E$.
\end{definition}

Consider a smooth $n$-fold vector bundle atlas, and write
$\{U_\alpha\}_{\alpha\in\Lambda}$ for the underlying open covering of
$M$.  For $\alpha,\beta,\gamma\in\Lambda$ the identity
$\Theta_\gamma\circ\Theta_\alpha\inv=\Theta_\gamma\circ
\Theta_\beta\inv\circ\Theta_\beta\circ\Theta_\alpha\inv$ on
$\pi^{-1}(U_\alpha\cap U_\beta\cap U_\gamma)$ yields the following cocycle
conditions. For $\emptyset\neq I\subseteq \nset$ and $\rho=(I_1,\ldots,I_k)\in\mathcal P(I)$:
\begin{equation}\label{n-cocycles}
  \begin{split}
    \omega^{\gamma\alpha}_\rho&(p)(v_{I_1},\ldots,v_{I_k})=\\
    &\sum_{(J_1,\ldots, J_l)\in\operatorname{coars}(\rho)}
			\omega^{\gamma\beta}_{(J_1,\ldots,J_l)}(p)
			\Bigl(\omega^{\beta\alpha}_{\rho\cap J_1}(p)
			\bigl((v_{I})_{I\in\rho\cap J_1}\bigr),\ldots,
			\omega^{\beta\alpha}_{\rho\cap J_l}(p)
			\bigl((v_I)_{I\in\rho\cap J_l}\bigr)\Bigr)\,.
\end{split}
\end{equation}

\bigskip

Since the involved constructions need to be refined, this sections recalls 
in detail the proof of Corollary 3.10 in \cite{HeJo20}:
\begin{corollary}\label{eq_defs_nvb}
Definition \ref{def_n_fold} of an $n$-fold vector bundle as a functor
from the $n$-cube category is equivalent
to Definition \ref{n_fold_vb_atlas_def} of an $n$-fold vector bundle as a
space with a maximal $n$-fold vector bundle atlas.
\end{corollary}

\begin{proof}
Let $\mathbb E$ be an $n$-fold vector bundle. By Corollary \ref{existence_dec_non_sym} (see also Theorem 3.2 and
Theorem 3.3 in \cite{HeJo20}), there is a decomposition
$\mathcal S\colon \E^{\rm dec}\to \E$ of $\E$, with $\E^{\rm dec}$ the decomposed $n$-fold vector bundle defined by 
the family $(\E^I_I)_{I\subseteq \nset}$ of vector bundles over $M$.  Set $E:=\E(\nset)$, as usual $M:=\E(\emptyset)$, and 
$\pi=\E(\nset\to\emptyset)\colon E\to M$. For each $\emptyset\neq I\subseteq \nset$,
set $V_I:=\mathbb R^{\dim E^I_I}$, the vector space on which $E^I_I$ is
modelled. Take a covering $\{U_\alpha\}_{\alpha\in\Lambda}$ of $M$ by
open sets trivialising all the vector bundles $E^I_I$;
\[ \phi^\alpha_I\colon q_{I}^{-1}(U_\alpha)\overset{\sim}{\longrightarrow} U_\alpha\times V_I
\]
for all  $\emptyset\neq I\subseteq \nset$ and all $\alpha\in\Lambda$, where $q_I\colon E^I_I\to M$ is the projection (the restriction of $p^I_\emptyset$).
Then define $n$-fold vector bundle charts $\Theta_\alpha\colon\pi^{-1}(U_\alpha)\to U_\alpha\times \prod_{I\subseteq \nset}V_I$ by 
\[\Theta_\alpha=\left(\pi\times \left(\phi^\alpha_I\right)_{I\subseteq \nset}\right)\circ \mathcal S(\nset)^{-1}\an{\pi^{-1}(U_\alpha)}\colon \pi^{-1}(U_\alpha)\to U_\alpha\times  \prod_{\emptyset\neq I\subseteq \nset}V_I.
\]
Given $\alpha,\beta\in\Lambda$ with $U_\alpha\cap U_\beta\neq\emptyset$, the change of chart
\[\Theta_\alpha\circ\Theta_\beta\inv\colon (U_\alpha\cap U_\beta)\times \prod_{\emptyset\neq I\subseteq \nset}V_I\to  (U_\alpha\cap U_\beta)\times \prod_{\emptyset\neq I\subseteq \nset}V_I
\]
is given by 
\begin{equation}\label{simple_change_of_charts}
\left(p, (v_I)_{\emptyset \neq I\subseteq \nset}\right)\mapsto \left(p, (\rho^{\alpha\beta}_I(p)v_I)_{\emptyset \neq I\subseteq \nset}\right),
\end{equation}
with
$\rho_I^{\alpha\beta}\in C^\infty(U_\alpha\cap U_\beta,
\operatorname{Gl}(V_I))$
the cocycle defined by $\phi^\alpha_I\circ (\phi_I^\beta)\inv$. The
two charts are hence smoothly compatible. Hence his defines an $n$-fold vector
bundle atlas
$\lie A=\{(U_\alpha,\Theta_\alpha, (V_I)_{I\subseteq\nset})\mid \alpha\in
\Lambda\}$ on $E$.

\medskip

Conversely, given a space $E$ with an $n$-fold vector bundle structure
over a smooth manifold $M$ as in Definition \ref{n_fold_vb_atlas_def},
define $\mathbb E\colon \square^{\N}\to \Man$ as follows.  Take a
maximal atlas
$\lie A=\{(U_\alpha,\Theta_\alpha, (V_I)_{I\subseteq\nset})\mid \alpha\in
\Lambda\}$
of $E$; in particular $\{U_\alpha\}_{\alpha\in\Lambda}$ is an open cover
of $M$.  For $\alpha,\beta,\gamma\in\Lambda$  the
identity
$\Theta_\gamma\circ\Theta_\alpha\inv=\Theta_\gamma\circ
\Theta_\beta\inv\circ\Theta_\beta\circ\Theta_\alpha\inv$
on $\pi^{-1}(U_\alpha\cap U_\beta\cap U_\gamma)$ yields the cocycle
conditions \eqref{n-cocycles}.

Set $\mathbb E(\nset):=E$, $\mathbb E(\emptyset):=M$, 
and more generally for $\emptyset\neq I\subseteq \nset$, 
\[\mathbb
E(I)=\left.\left(\bigsqcup_{\alpha\in\Lambda}\left(U_\alpha\times\prod_{\emptyset\neq J\subseteq
  I}V_J\right)\right)\right/\sim \]
with $\sim$ the equivalence relation defined on
$\bigsqcup_{\alpha\in\Lambda}(U_\alpha\times\prod_{\emptyset\neq J\subseteq I}V_J)$
by 
\[ U_\alpha\times\prod_{\emptyset\neq J\subseteq I}V_J\quad \ni\quad
\left(p,(v_J)_{\emptyset\neq J\subseteq I}\right)\quad \sim \quad
\left(q,(w_J)_{\emptyset\neq J\subseteq I}\right)\quad \in \quad U_\beta\times\prod_{\emptyset\neq J\subseteq I}V_J
\]
if and only if $p=q$ and
\[ (v_J)_{\emptyset\neq J\subseteq I}=\left(\sum_{\rho=(J_1,\ldots,J_k)\in\mathcal
    P(J)}\omega^{\alpha\beta}_{\rho}(p)(w_{J_1},\ldots,w_{J_k})\right)_{\emptyset\neq J\subseteq I}.
\]
The relations \eqref{n-cocycles} show the symmetry and transitivity of
this relation. As in the construction of a vector bundle from vector
bundle cocycles,  $\mathbb E(I)$ has a unique smooth
manifold structure such that
$\pi_I\colon \mathbb E(I)\to M$, $\pi_I[p,(v_I)_{\emptyset\neq I\subseteq J}]=p$
is a surjective submersion and such that
 the maps 
\[\Theta^I_\alpha\colon \pr_I\left(U_\alpha\times\prod_{\emptyset\neq J\subseteq
  I}V_J\right)\to U_\alpha\times\prod_{\emptyset\neq J\subseteq I}V_J, \qquad
\left[p,(v_I)_{\emptyset\neq I\subseteq J}\right]\mapsto \left(p,(v_I)_{\emptyset\neq I\subseteq J}\right)
\]
are diffeomorphisms, where $\pr_I\colon \bigsqcup_{\alpha\in\Lambda}(U_\alpha\times\prod_{\emptyset\neq J\subseteq
  I}V_J)\to \mathbb E(I)$ is the projection to the equivalence classes.

$\E(I)$ comes equipped with $\# I$ surjective submersions
\[ p^I_{I\setminus \{i\}}\colon \E(I)\to \E(I\setminus\{i\})\]
for $i\in I$, defined in charts by 
\[ U_\alpha\times \prod_{\emptyset\neq J\subseteq I}V_J \,\ni\, \left(p, (v_J)_{\emptyset\neq J\subseteq
  I}\right)\,\mapsto \, \left(p, (v_J)_{ \emptyset\neq J\subseteq I\setminus\{i\}}\right)\,\in \, U_\alpha\times \prod_{\emptyset\neq J\subseteq I\setminus\{i\}} V_J
\]
and it is easy to see that $\E(I)$ is a vector bundle over
$\E(I\setminus\{i\})$, and that for $i,j\in I$, 
\[
\begin{tikzcd}
\E(I)\ar[rr,"p^{I}_{I\setminus\{i\}}"] \ar[d,"p^{I}_{I\setminus\{j\}}"] && \E(I\setminus \{i\})\ar[d,"p^{I\setminus \{i\}}_{I\setminus\{i,j\}} "]\\
\E(I\setminus \{j\}) \ar[rr,"p^{I\setminus \{j\}}_{I\setminus\{i,j\}}  "] & &\E(I\setminus \{i,j\})
\end{tikzcd} 
\]
is a double vector bundle,  with obvious local trivialisations given
by the local charts.
\end{proof}

\begin{remark}
  Note that the construction above of an $n$-fold vector bundle atlas
  on $\E(\nset)$ from an $n$-fold vector bundle yields an atlas with
  simpler changes of charts \eqref{simple_change_of_charts} than the
  most general allowed change of charts \eqref{change_of_charts}. This
  is due to the choice of a \emph{global} decomposition of the
  $n$-fold vector bundle. Choosing different local 
  decompositions yields an atlas with changes of charts as in
  \eqref{change_of_charts}.  
\end{remark}

\section{Symmetric $n$-fold vector bundles}\label{symmetric_nvb}
This section introduces \emph{symmetric $n$-fold vector bundles}
and their morphisms, their symmetric atlases and the decomposed
symmetric $n$-fold vector bundles.

\subsection{Global definition of symmetric $n$-fold vector bundles.}
This section defines symmetric $n$-fold vector bundles. Recall
that an $n$-fold vector bundle is a functor
$\E\colon\square^n\to\Man$. The $n$-cube category $\square^n$ is
equipped as follows with a canonical left action $\Phi$ of the symmetric group
$S_n$ by isomorphisms of categories. The
permutation $\sigma\in S_n$ acts by $\Phi_\sigma\colon\square^n\to\square^n$,
which maps an object $I\subseteq\nset$ to the object
$\sigma(I)\subseteq\nset$ and morphisms in the obvious way. This
action induces as follows a right action of $S_n$ on the category of $n$-fold
vector bundles.
\begin{definition}
	For $\sigma\in S_n$ define the \textbf{$\sigma$-flip} of an 
	$n$-fold vector bundle $\E$ to be the $n$-fold vector bundle 
	$\E^\sigma:=\E\circ\Phi_{\sigma}\colon\square^n\to\Man$. Since 
	$\Phi$ is a left action of $S_n$ on $\square^n$, then
	$(\E^\sigma)^\tau=\E^{\sigma\tau}$. Given a morphism $\tau\colon \E\to\F$ of $n$-fold 
	vector bundles, there is an obvious morphism 
	$\tau^\sigma\colon \E^\sigma\to\F^\sigma$ of $n$-fold vector bundles defined by $\tau^\sigma(I)=\tau(\sigma(I))$ for all $I\subseteq \nset$.
\end{definition}
Note that $\E^\sigma$ has the same underlying spaces as $\E$ but in 
different positions, which is crucial when considering morphisms of 
multiple vector bundles. For example the double vector bundle 
$(D;A,B;M)$ is different from the double vector bundle $(D;B,A;M)$.

\begin{definition}\label{def_sym_nvb}
	An $n$-fold vector bundle 
	$\E\colon \square^n\to \Man$ has a \textbf{symmetric structure}
if \begin{enumerate}
 \item  its building bundles satisfy\footnote{The building bundles $\E^I_I$ and $\E^J_J$ are of course equal up to isomorphism for $\# I=\# J$. In condition (c), the identity map is then this isomorphism.} $\E^I_I=\E^J_J$ 
   for $\#I=\#J$, and
 \item  $\mathbb E$ is endowed with a left
   $S_n$-action $\Psi\colon S_n\to \operatorname{Mor}(\E)$ in the sense that
   \begin{enumerate}
     \item[(a)] for any $\sigma\in S_n$,  $\Psi_\sigma\colon\E\to\E^\sigma$ 
     is a morphism of $n$-fold vector bundles, 
     \item[(b)] $\Psi_{\id}=\id_{\E}\colon\E\to\E$ and for all $\sigma,\tau\in S_n$
     \[ \Psi_{\sigma\tau}=(\Psi_\sigma)^\tau\circ \Psi_\tau\colon \E\to \E^\tau\to \E^{\sigma\tau}=(\E^\sigma)^\tau
     \]
     	\item[(c)]
		$\Psi_\sigma^{I,I}=\varepsilon(\sigma,I)\cdot\id_{\E^I_I}
		\colon \E^I_I\to \E^{\sigma(I)}_{\sigma(I)}=\E^I_I$.
              \end{enumerate}
\end{enumerate}
            
An $n$-fold vector bundle together with a symmetric structure
is called a \textbf{symmetric $n$-fold vector bundle}.

\medskip

Let $\E\colon \square^n\to \Man$ and $\F\colon \square^n\to \Man$ be two 
symmetric $n$-fold vector bundles with actions $\Phi$ and $\Psi$, respectively.
A \textbf{morphism $\tau\colon \mathbb E\to \mathbb F$ of symmetric $n$-fold vector bundles}
is a morphism of $n$-fold vector bundles which is $S_n$-equivariant, 
that is, which is such that the diagrams
\[
\begin{tikzcd}
\E(I)\ar[rr,"\tau(I)"] \ar[d,"\Phi_\sigma(I)"] && \mathbb F(I)\ar[d,"\Psi_\sigma(I) "]\\
\E(\sigma(I)) \ar[rr,"\tau(\sigma(I))  "] && \mathbb F(\sigma(I)) 
\end{tikzcd} 
\] 
commute for all
$I\subseteq \nset$ and all $\sigma\in S_n$. 
That is, the diagram 
\[
\begin{tikzcd}
\E\ar[rr,"\tau"] \ar[d,"\Phi_\sigma"] && \mathbb F\ar[d,"\Psi_\sigma "]\\
\E^\sigma \ar[rr,"\tau^\sigma  "] && \mathbb F^\sigma
\end{tikzcd} 
\] 
of natural transformations commutes for all $\sigma\in S_n$.

\medskip

  $\mathsf{SnVB}$ is the category of symmetric $n$-fold vector
  bundles.
\end{definition}

\subsection{Example: the pullback of a symmetric $n$-fold vector bundle}

Let $\E$ be an $n$-fold vector bundle.
Then the $n$\textbf{-pullback of }$\E$ is the set
\[P=\left\{(e_1,\ldots,e_n) \left| e_i\in \E(\nset\setminus\{i\}) %
		\text{ and } p_{\underline{n}\setminus\{i,j\}}^{\underline{n}\setminus\{i\}}(e_i)= p_{\underline{n}\setminus\{i,j\}}^{\nset\setminus\{j\}}(e_j)\in\E(\nset\setminus\{i,j\})
			\text{ for } i,j\in \nset\right.\right\}\,.
 \]
By Theorem 2.10 in \cite{HeJo20}, $P$ is a smooth embedded submanifold 
of the product $\E(\nset\setminus\{1\})\times \ldots \times \E(\nset\setminus\{n\})$, and 
the functor $\P$
    defined by 
    \begin{itemize}
    \item $\P(\nset)=P$, 
    \item $\P(S)=\E(S)$ for all $S\subsetneq \nset$
    \end{itemize}
    and the vector bundle projections
    \begin{itemize}
 \item   $p^S_{S\setminus\{i\}}\colon \E(S)\to \E(S\setminus\{i\})$ for all
    $S\subsetneq \nset$ and $i\in S$
    \item
    $p'_{\nset\setminus\{i\}}\colon P\to \E(\nset\setminus\{i\})$,
    $(e_1,\ldots,e_n)\mapsto e_i$ \end{itemize}
    is an $n$-fold vector bundle.
Furthermore, define the map $\pi(\nset)\colon \E(\nset)\to P$ by $\pi(\nset)\colon e\mapsto (p_{\nset\setminus\{1\}}(e),\ldots, p_{\nset\setminus\{n\}}(e))$.   The map $\pi(\nset)$ defines together with $\pi(J)=\id_{\E(J)}$ for $J\subsetneq \nset$,
    a surjective morphism $\pi\colon \E \to \P$ of  $n$-fold vector bundles.
Note that for each $i\in\nset$, the top map $\pi(\nset)\colon \E(\nset)\to P$
of $\pi$ is necessarily  a vector
bundle morphism over the identity on $\E(\nset\setminus\{i\})$.
Note also that the building bundle $\P^{\nset}_{\nset}\to M$ is 
\begin{equation}\label{ultracore_P}
\left\{\left.\left(0^{\nset\setminus\{1\}}_m,\ldots,0^{\nset\setminus\{n\}}_m\right)\in \E(\nset\setminus\{1\})\times\ldots\times \E(\nset\setminus\{n\})
\right| m\in M \right\}\,,
\end{equation}
hence the trivial vector bundle $M\to M$, while by construction 
$\mathbb P^I_I=\E^I_I$ for all $I\subsetneq \nset$.

\medskip

Let $\E$ be a symmetric $n$-fold vector bundle. This section shows that its pullback $n$-fold 
vector bundle $\mathbb P\colon \square^n\to \Man$ (see \cite{HeJo20}) is 
also a symmetric $n$-fold vector bundle, and the projection 
$\pi\colon \mathbb E\to \mathbb P$ is a morphism of symmetric $n$-fold vector bundles.

\medskip

The building bundles of $\mathbb P$ must satisfy $\mathbb P_I^I=\mathbb P_J^J$ for 
$I,J\subseteq \nset $ with $\# I=\# J$. This is immediate since 
for $I\subsetneq  \nset$ the vector bundle $\mathbb P_J^J$ equals $\E_J^J$, and 
the building blocks of $E$ satisfy $\E^I_I=\E^J_J$ for $I,J\subseteq \nset$ 
with $\# I=\# J$. There is further only one subset of $\nset$ with 
cardinality $n$, so the condition (1) in Definition \ref{def_sym_nvb} is 
trivially satisfied for that set.

Consider the action $\Psi$ of $S_n$ on $\E$ and define the following 
action $\Phi$ of $S_n$ on $\mathbb P$: for $\sigma\in S_n$ and $J\subsetneq \nset$, 
\[\Phi_\sigma(J):=\Psi_\sigma(J)\colon \mathbb E(J)\to \mathbb E^\sigma(J)
=\mathbb E(\sigma(J)).
\]
The map $\Phi_\sigma(\nset)\colon P\to P$ is further defined by
\begin{equation}\label{top_sym} (e_1, \ldots, e_n)\mapsto (\Psi_\sigma(\nset\setminus\{\sigma^{-1}(1)\})(e_{\sigma^{-1}(1)}), 
\ldots, \Psi_\sigma(\nset\setminus\{\sigma^{-1}(n)\})(e_{\sigma^{-1}(n)}))\in P\,,
\end{equation}
and is automatically smooth.

For all $I\subsetneq \nset$ and all $\sigma\in S_n$, the map 
$(\Phi_\sigma)^I_I\colon \mathbb P^I_I=\E^I_I\to \E^I_I$ equals 
$(\Psi_\sigma)^I_I=\epsilon(\sigma,I)\cdot \id_{\E^I_I}$ since 
$\Phi_\sigma(I)=\Psi_\sigma(I)$. Since $\mathbb P_{\nset}^{\nset}$ is the trivial 
vector bundle over $M$
and $(\Phi_\sigma)_{\nset}^{\nset}$ is obviously the identity on this trivial bundle 
by \eqref{ultracore_P} and \eqref{top_sym}, the third condition in Definition \ref{def_sym_nvb} 
is satisfied for $I=\nset$.

It remains to show that $\Psi_\sigma$ is a natural transformation for each 
$\sigma\in S_n$ and that for all $I\subseteq \nset$ and $i\in I$ the square
\[
\begin{tikzcd}
\E({I})\ar[rr,"\Psi_\sigma(I)"] \ar[d,"p^{I}_{I\setminus\{i\}}"] && \E(\sigma(I))\ar[d,"p^{\sigma(I)}_{\sigma(I)\setminus\{\sigma(i)\}} "]\\
\E(I\setminus \{i\}) \ar[rr,"\Psi_\sigma(I\setminus\{i\})  "] && \E(\sigma(I)\setminus \{\sigma(i)\}) 
\end{tikzcd} 
\] 
is a vector bundle homomorphism. But again, since $\Phi_\sigma(I)=\Psi_\sigma(I)$ for all $I\subsetneq \nset$, it suffices to check that 
\[
\begin{tikzcd}
P\ar[rr,"\Phi_\sigma(\nset)"] \ar[d,"p^{\nset}_{\nset\setminus\{i\}}"] && P\ar[d,"p^{\nset}_{\nset\setminus\{\sigma(i)\}} "]\\
\E(\nset\setminus \{i\}) \ar[rr,"\Psi_\sigma(\nset\setminus\{i\})  "] && \E(\nset\setminus \{\sigma(i)\})
\end{tikzcd} 
\] 
is a vector bundle homomorphism for each $i\in\nset$. This, in turn, follows immediately from the definition of $\Phi_\sigma(\nset)$.

Finally the diagrams 
\[
\begin{tikzcd}
\mathbb E(J)\ar[rr,"\pi(J)"] \ar[d,"\Psi_\sigma(J)"] && \mathbb P(J)\ar[d,"\Phi_\sigma(J)"]\\
\mathbb E^\sigma(J)\ar[rr,"\pi_\sigma(J)  "] && \mathbb P^\sigma(J)
\end{tikzcd} 
\] 
all commute, since for $\sigma\in S_n$ and $I\subsetneq \nset$, this is 
\[
\begin{tikzcd}
\E(J)\ar[rr,"\id"] \ar[d,"\Psi_\sigma(J)"] && \E(J)\ar[d,"\Psi_\sigma(J)"]\\
\E(\sigma(J))\ar[rr,"\id  "] && \E(\sigma(J))
\end{tikzcd}, 
\] 
and for $I=\nset$, this is 
\[
\begin{tikzcd}
\E(\nset)\ar[rr,"\pi(\nset)"] \ar[d,"\Psi_\sigma(\nset)"] && P\ar[d,"\Phi_\sigma(\nset)"]\\
\E(\nset)\ar[rr,"\pi(\nset)  "] && P
\end{tikzcd} ,
\] 
which commutes since for all $e\in E$, 
\begin{equation*}
\begin{split}\pi(\nset)(\Psi_\sigma(\nset)(e))&=\left(p_{1}(\Psi_\sigma(\nset)(e)),\ldots, p_{n}(\Psi_\sigma(\nset)(e))
\right)\\
&=\left(\Psi_\sigma(\nset\setminus\{\sigma\inv(1)\})(p_{\sigma\inv(1)}(e)), \ldots, \Psi_\sigma(\nset\setminus\{\sigma\inv(n)\})(p_{\sigma\inv(n)}(e))\right)\\
&=\Phi_\sigma(\nset)(p_1(e), \ldots, p_n(e))=\Phi_\sigma(\nset)(\pi(\nset)(e)).\qedhere 
\end{split}
\end{equation*}

\subsection{Iterated highest order cores of symmetric $n$-fold vector bundles}

This section shows that for $l\in \{1, \ldots, n\}$, the family of $l$-cores of a symmetric $n$-fold vector bundle
inherit an $S_n$-symmetry. More precisely, the $S_n$-action on $\E$ restricts to morphisms between the different $l$-cores.
\begin{lemma}
Let $\E$ be a symmetric $n$-fold vector bundle with $S_n$-action $\Phi$. Choose a partition $\rho=\{I_1,\ldots, I_l\}$ of $\nset$ in $\#\rho=l$ subsets. Then for each $\sigma\in S_n$, the partition $\sigma(\rho)=\{\sigma(I_1), \ldots, \sigma(I_l)\}$ again has $l$ elements, and $\Phi_\sigma\colon \E\to \E^\sigma$ restricts to a morphism
\[ \Phi^\rho_\sigma\colon \E^\rho\to (\E^{\sigma(\rho)})^\sigma
\]
of $l$-fold vector bundles.
For each $J\in \operatorname{Obj}(\lozenge^\rho)$ the smooth map
\[ \Phi^\rho_\sigma(J)\colon \E^\rho(J)\to \E^{\sigma(\rho)}(\sigma(J))
\]
is the restriction to the embedded domain $\E^\rho(J)\subseteq \E(J)$ and the embedded codomain $\E^{\sigma(\rho)}(\sigma(J))\subseteq \E(\sigma(J))$ of the smooth map
\[ \Phi(\sigma)(J)\colon \E(J)\rightarrow \E(\sigma(J)).
\]
\end{lemma}

\begin{proof}
First choose a subset $J\subseteq \nset$ with two elements and consider the $(n-1)$-core $\E^{J}$ of $\E$. The morphism $\Psi_\sigma\colon \E\to \E^\sigma$ of $n$-fold vector bundles restricts to a 
morphism
\[ (\Psi_\sigma)^J\colon \E^J\to (\E^\sigma)^J
\]
for $(n-1)$-fold vector bundles. An easy computation shows that \begin{equation}\label{highest_sigma_comp}
(\E^\sigma)^J=(\E^{\sigma(J)})^\sigma
\end{equation}
 with, as usual,
$(\E^{\sigma(J)})^\sigma(I):=\E^{\sigma(J)}(\sigma(I))$
for all $I\subseteq\nset$. Assume that for an $l$-partition $\rho=\{I_1,\ldots,I_l\}$ of $\nset$, 
\[(\E^\sigma)^\rho=(\E^{\sigma(\rho)})^\sigma
\]
and choose an $(l-1)$-coarsement $\rho'$ of $\rho$. Then without loss of generality $\rho'=\{I_1\cup I_2, I_3, \ldots, I_l\}$ and 
\[ \left(\E^\sigma\right)^{\rho'}=\left(\left(\E^\sigma\right)^\rho\right)^{\rho'}=\left(\left(\E^{\sigma(\rho)}\right)^\sigma\right)^{\rho'}= \left(\left(\E^{\sigma(\rho)}\right)^{\sigma(\rho')}\right)^{\sigma}=\left(\E^{\sigma(\rho')}\right)^\sigma,
\]
since $\sigma(\rho')$ is a coarsement of $\sigma(\rho)$. In the third equality, \eqref{highest_sigma_comp} is used since the partition $\rho'$ gives a highest order core of $\E^\rho$.

The rest of the statement is then immediate.
\end{proof}

\subsection{Decomposed symmetric $n$-fold vector bundles}\label{dec_sym_nvb}
  Consider a smooth manifold $M$ and a collection of vector bundles
$\mathcal A=\{q_i\colon A_i\to M \mid 1\leq i\leq n\}$.  Define a
functor $\E^{\mathcal A}\colon \square^{n}\to \Man$ as follows.  Each
object $I\subseteq \nset$ is sent to \[\E^{\mathcal A}(I):=\prod^M_{\emptyset\neq J\subseteq I} A_{\# J}\, \simeq \prod^M_{i\in\{1,\ldots, n\}}A_i^{\#\{J\subseteq I\mid \# J=i\}},\]
as a fibered product of vector bundles over $M$. That is, 
\[ \E^{\mathcal A}(\{1,2\})=A_1\times A_1\times A_2,
\]
\[\E^{\mathcal A}(\{2,4,5\})=A_1\times A_1\times A_1\times A_2\times A_2\times A_2\times A_3, 
\]
etc.
	
  For $I\in\nset$ with $\# I\geq 1$ and for $k\in I$, the
  arrow $I\to I\setminus\{k\}$ is sent to the canonical vector bundle
  projection\footnote{By convention $A_0$ is set to be $M$.}
  \[p^I_{I\setminus \{k\}}\colon \prod^M_{\emptyset\neq J\subseteq I} A_{\# J}\to \prod^M_{J\subseteq
      I\setminus\{k\}} A_{\# J}.\] In particular, the arrow
  $\{i\}\to\emptyset$ for $i\in\nset$ is sent to the vector bundle
  projection
  $p^{\{i\}}_\emptyset=q_{\{i\}}\colon \E^{\mathcal A}(\{i\})=A_1\to
  \E^{\mathcal A}(\emptyset)=M$.    

  This decomposed  $n$-fold vector bundle $\E^{\mathcal A}$ satisfies the following two conditions:
  \begin{enumerate}
  \item By construction the building bundles satisfy $(\E^{\mathcal A})^{I}_{I}=A_{\# I}=A_{\# J}=(\E^{\mathcal A})^J_J$ for $\# I=\# J$. 
  \item The maps 
  \begin{equation*}
  \begin{split}
   \Psi^{\mathcal A}_\sigma(I)\colon \mathbb E^{\mathcal A}(I)&\to\mathbb E^{\mathcal A, \sigma}(I)=\mathbb E^{\mathcal A}(\sigma(I))\\
  (a_J)_{\emptyset\neq J\subseteq I}&\mapsto (\epsilon(\sigma\inv, J)a_{\sigma\inv(J)})_{\emptyset\neq J\subseteq \sigma(I)}
  \end{split}
  \end{equation*}
  for $\sigma\in S_n$ and $I\subseteq \nset$ define an $S_n$-symmetry on $\mathbb E^{\mathcal A}$.

For any $\sigma\in S_n$ the n-fold vector bundle $\mathbb E^{\mathcal A,\sigma}$ is 
the decomposed $n$-fold vector bundle with $\mathbb E^{\mathcal A,\sigma}(I)=\mathbb E^{\mathcal A}(\sigma(I))=\prod^M_{J\subseteq \sigma(I)} A_{\# J}$. Therefore, $\Psi^{\mathcal A}_\sigma\colon \mathbb E^{\mathcal A}\to \mathbb E^{\mathcal A, \sigma}$ is obviously a morphism of $n$-fold vector bundles.
Furthermore, by definition, $(\Psi^{\mathcal A}_\sigma)^I_I$ is clearly $\epsilon(\sigma\inv, \sigma(I))\cdot \id_{A_{\# I}}=\epsilon (\sigma, I)\cdot \id_{A_{\# I}}$ for all $I\in\underline n$.

Obviously, $\Psi^{\mathcal A}_{\id}(I)=\id_{\E^{\mathcal A}(I)}$ for all $I\subseteq \nset$. Choose $\emptyset\neq J\subseteq \nset$ and take
$\sigma, \tau\in S_n$ and
$\left(p, (v_I)_{\emptyset\neq I\subseteq J}\right)\in  \prod_{\emptyset\neq I\subseteq J}^M A_{\# I}=\E^{\mathcal A}(J)$. Then 
\begin{equation}\label{decomposed_Sn_action}
\begin{split}
\left(\Psi_\sigma^{\mathcal A}\circ \Psi_\tau^{\mathcal A}\right)\left(
  (v_I)_{\emptyset\neq I\subseteq J}\right)
  &\,=\Psi_\sigma^{\mathcal A}\left(\left( \underset{=:w_I\in A_{\# I}}{\underbrace{\epsilon(\tau\inv,I)v_{\tau\inv(I)}}}
  \right)_{\emptyset\neq I\subseteq \tau(J)}
  \right)\\
  &\,=\left( \left(\epsilon(\sigma\inv,I)w_{\sigma\inv(I)}\right)_{\emptyset\neq I\subseteq \sigma(\tau(J))}
  \right)\\
  &\,=\left( \left(\epsilon(\sigma\inv,I)\epsilon(\tau\inv, \sigma\inv(I))v_{\tau\inv(\sigma\inv(I))}\right)_{\emptyset\neq I\subseteq  \sigma(\tau(J))}
  \right)\\
  &\overset{\eqref{prod1}}{=}\left( \left(\epsilon((\sigma\tau)\inv, I)v_{(\sigma\tau)\inv(I))}\right)_{\emptyset\neq I\subseteq \sigma\tau(J)}
  \right)\\
  &\,=\Psi_{\sigma\tau}^{\mathcal A}\left(
  (v_I)_{\emptyset\neq I\subseteq J}\right).
  \end{split}
  \end{equation}
 This shows
$(\Psi_\sigma^{\mathcal A})^\tau\circ \Psi_\tau^{\mathcal A}= \Psi_{\sigma\tau}^{\mathcal A}$.

  \end{enumerate}
    
  \begin{definition}\label{decomposed_nfvb}
  Let $n\in\mathbb N$ and let $A_1, \ldots, A_n$ be vector bundles over a smooth manifold $M$.
  \begin{enumerate}
  \item 
  The symmetric $n$-fold vector bundle $\E^{\mathcal A}\colon \square^{n}\to \Man$ as defined above is the \textbf{decomposed symmetric $n$-fold vector
    bundle} defined by $\mathcal A:=\{A_1,\ldots, A_n\}$.
    \item If $A_2, \ldots, A_n$ are all trivial vector bundles over $M$, then
    the symmetric  decomposed $n$-fold vector bundle $\E^{\mathcal A}=:\E^{\{A_1\}}\colon \square^{n}\to \Man$ defined by $\mathcal A$ is \emph{vacant}.
    The functor $\E^{\{A_1\}}$ is here precisely given by $\E^{\{A_1\}}(I)=\prod_{i\in I}^MA_1=A_1^{\# I}$ for all $I\subseteq \nset$. It is the \textbf{vacant symmetric $n$-fold vector bundle} defined by the vector bundle $A_1$.
      \end{enumerate}
    \end{definition}
    There exists a canonical monomorphism $\iota\colon \E^{\{A_1\}}\to \E^{\mathcal A}$ of symmetric $n$-fold vector bundles, given by the canonical embeddings
    \[ \tau(I)\colon \prod_{i\in I}A_1 \hookrightarrow \prod_{J\subseteq I}^MA_{\# J}, \quad (a_i)_{i\in I}\mapsto (a_J)_{J\subseteq I}; \quad a_J:=\left\{\begin{array}{cl}
    a_i & J=\{i\}\\
    0_m^{A_l} & \# J=l\geq 2\end{array}\right.
    \]
    if $m\in M$ is the foot point of the tuple $(a_i)_{i\in I}$.
    
     Assume that $\E\colon \square^{n}\to \Man$ is a symmetric decomposed $n$-fold vector bundle and that for $i\in \{1,\ldots, n\}$, the vector bundle $A_i$ is $\E^I_I$ for any $I\subseteq \nset$ with $\# I=i$. Then the vacant symmetric decomposed $n$-fold vector bundle defined by $A_1$ is written $\overline{\E}\colon \square^{n}\to \Man$, and the symmetric decomposed $n$-fold vector bundle defined by $\{A_1,\ldots, A_n\}$ is written $\E^{\rm dec}\colon \square^{n}\to \Man$. If $\Psi$ is the $S_n$-action on $\E$, then the $S_n$-action on $\overline{\E}$ is written $\overline{\Psi}$, and the $S_n$-action $\Psi^{\A}$ on $\E^{\rm dec}$ is written $\Psi^{\rm dec}$. 
    \medskip

\subsection{Morphisms of decomposed symmetric $n$-fold vector bundles}  
  Let $\E^{\mathcal A}\colon \square^{n}\to \Man$ and
  $\E^{\mathcal B}\colon \square^{n}\to \Man$ be two decomposed
  symmetric $n$-fold vector bundles, defined by
  $\mathcal A=(A_i\to M)_{1\leq i\leq n}$ and
  $\mathcal B=(B_i\to N)_{1\leq i\leq n}$, respectively, and with
  $S_n$-actions $\Psi^{\mathcal A}$ and $\Psi^{\mathcal B}$.  Recall that a
  morphism of (decomposed) symmetric $n$-fold vector bundles
  from $\E^{\mathcal A}$ to $\E^{\mathcal B}$ is a natural
  transformation
  $\tau\colon \E^{\mathcal A}\rightarrow \E^{\mathcal B}$ such that
  for all objects $I$ of $\square^{n}$ and for all $i\in I$, the
  commutative diagram
\[
\begin{tikzcd}
\E^{\mathcal A}(I)\ar[r,"\tau(I)"] \ar[d,"p^{I}_{I\setminus\{i\}}"] & \E^{\mathcal B}(I)\ar[d,"p^{I}_{I\setminus\{i\}} "]\\
\E^{\mathcal A}(I\setminus \{i\})\ar[r,"\tau(I\setminus\{i\})  "] & \E^{\mathcal B}(I\setminus \{i\})
\end{tikzcd} 
\]
is a morphism of vector bundles, and such that for all $\sigma\in S_n$,
\[ \Psi_\sigma^{\mathcal B}(\nset)\circ \tau(\nset)=\tau(\nset)\circ \Psi_\sigma^{\mathcal A}(\nset).
  \]

In the following theorem, vector bundle morphisms $\tau_{(i_1,\ldots,i_k)}\colon A_{i_1}\otimes\ldots\otimes
  A_{i_k}\to B_i$ over $\tau(\emptyset)\colon M\to N$ are considered, which are indexed by pairs $(i, (i_1,\ldots,i_k))$, with $i\in\{1,\ldots,n\}$ and  $(i_1,\ldots,i_k)$ an ordered integer partition of $i$.
  Rewrite $A_{i_1}\otimes \ldots\otimes A_{i_k}$ as 
\[ A_1^{\otimes \#\{s\in\{1,\ldots,k\}\mid i_s=1\}}\otimes \ldots\otimes A_i^{\otimes \#\{s\in\{1,\ldots,k\}\mid i_s=i\}}.
\]
Then 
\[
\tau_{(i_1,\ldots,i_k)}\colon A_1^{\otimes \#\{s\in\{1,\ldots,k\}\mid i_s=1\}}\otimes \ldots\otimes A_i^{\otimes \#\{s\in\{1,\ldots,k\}\mid i_s=i\}}\to B_i
\]
is \textbf{symmetric} (respectively \textbf{skew-symmetric}) \textbf{in the $j$-entries}, for $j\in\{1,\ldots,i\}$, if it is symmetric (respectively alternating) on  its 
$j$-th factor $A_j^{\otimes \#\{s\in\{1,\ldots,k\}\mid i_s=j\}}$.

\begin{proposition}\label{morphisms_of_dSnVB}
  A morphism $\tau\colon \E^{\mathcal A}\to \E^{\mathcal B}$ of
  decomposed symmetric $n$-fold vector bundles is equivalent to a
  family of vector bundle morphisms
  $\tau_{(i_1,\ldots,i_k)}\colon A_{i_1}\otimes\ldots\otimes
  A_{i_k}\to B_i$ over $\tau(\emptyset)\colon M\to N$, indexed by pairs $(i, (i_1,\ldots,i_k))$, with $i\in\{1,\ldots,n\}$ and  $(i_1,\ldots,i_k)$ an ordered integer partition of $i$, such that for $j\in \{1,\dots, i\}$:
  \begin{enumerate}
  \item $\tau_{(i_1,\ldots,i_k)}$ is symmetric in the $j$-entries for $j$ even, and 
  \item $\tau_{(i_1,\ldots,i_k)}$ is skew-symmetric in the $j$-entries for $j$ odd.
  \end{enumerate}
  
  Given such a family, 
  the corresponding morphism of decomposed
  symmetric $n$-fold vector bundles
  $\tau\colon \E^{\mathcal A}\to \E^{\mathcal B}$ is given by
  $\tau(J)\colon \E^{\mathcal A}(J)\to \E^{\mathcal B}(J)$,
\[\tau(J)\left((a_I)_{\emptyset\neq I\subseteq J}\right)=\left(\underset{\in B_{\# I}}{\underbrace{\sum_{\rho=(I_1,\ldots,I_k)\in\mathcal P(I)}\sgn(\rho)\cdot \tau_{(\# I_1,\ldots,\# I_k)}(a_{I_1},\ldots,a_{I_k})}}\right)_{\emptyset\neq I\subseteq
    J},
\]
for all $\emptyset\neq J\subseteq \nset$, for $i\in\{1,\ldots, n\}$ and $(i_1,\ldots, i_k)$ a (canonically) ordered integer partition of $i$.

Conversely, given $\tau\colon \E^{\mathcal A}\to \E^{\mathcal B}$, the
morphism
$\tau_{(i_1,\ldots,i_k)}\colon A_{i_1}\otimes\ldots\otimes A_{i_k}\to
B_i$ is given as follows.
Consider the canonical partition $\rho^{i_1,\ldots,i_k}_{\rm can}:=(K_1,\ldots,K_k)$ of $\{1,\ldots, i\}$ defined by $(i_1,\ldots, i_k)$.
Then for $a_{i_1}\in A_{i_1}, \ldots, a_{i_k}\in A_{i_k}$, the image 
\[\tau_{(i_1,\ldots,i_k)}(a_{i_1},\ldots, a_{i_k})\in B_i=B_{\# \underline{i}}
\]
is the $\underline{i}$-entry of the image under $\tau(\nset)$ of the tuple $(a_J)_{\emptyset\neq J\subseteq \nset}$ with 
\[
a_J=\left\{\begin{array}{cl}0 &\text{ for } J\neq K_1,\ldots, K_k\\
 a_{i_l}&\text{ for } J=K_l, \,\, l\in\{1,\ldots, k\}.
\end{array}\right.
\]
\end{proposition}

The proof of this statement relies on the following lemma.

\begin{lemma}\label{lemma_intermediate_morphisms}
In the situation of the previous theorem, consider a family of vector bundle morphisms
  $\tau_{(i_1,\ldots,i_k)}\colon A_{i_1}\otimes\ldots\otimes
  A_{i_k}\to B_i$ over $\tau(\emptyset)\colon M\to N$, indexed by pairs $(i, (i_1,\ldots,i_k))$, with $i\in\{1,\ldots,n\}$ and  $(i_1,\ldots,i_k)$ an ordered integer partition of $i$, such that for $j\in \{1,\dots, i\}$:
  \begin{enumerate}
  \item $\tau_{(i_1,\ldots,i_k)}$ is symmetric in the $j$-entries for $j$ even, and 
  \item $\tau_{(i_1,\ldots,i_k)}$ is skew-symmetric in the $j$-entries for $j$ odd.
  \end{enumerate}
 Take a subset $\emptyset\neq I\subseteq \nset$ and an ordered partition $\rho=(I_1,\ldots, I_k)$ of $I$, as well as a permutation $\sigma\in S_n$. Denote by $(J_1^{\rho,\sigma}, \ldots, J_k^{\rho,\sigma})$ the ordered partition $\sigma(\rho)$. That is, the subsets $J_1^{\rho,\sigma}, \ldots, J_k^{\rho,\sigma}\subseteq \sigma(I)$ are the subsets $\sigma(I_1), \ldots, \sigma(I_k)$, but back in the lexicographic order.
 Then for  $a_{\sigma(I_1)}\in A_{\# I_1}, \ldots, a_{\sigma(I_k)}\in A_{\# I_k}$
 \begin{equation}\label{sym_dec_morphism_cond_simpl}
 \begin{split}
& \tau_{(\# I_1,\ldots,\# I_k)}(a_{\sigma(I_1)}, \ldots, a_{\sigma(I_k)}) =\frac{\sgn(\sigma(\rho))\cdot \epsilon(\sigma, I)}{\sgn(\rho)\cdot \prod_{l=1}^k\epsilon(\sigma, I_l)}\cdot \tau_{(\# I_1,\ldots, \# I_k)}\left(a_{J_1^{\rho,\sigma}}, \ldots, a_{J_k^{\rho,\sigma}}\right).
 \end{split}
 \end{equation}

\end{lemma}

\begin{proof}
Choose a subset $\emptyset\neq I\subseteq \nset$ and an ordered partition $(I_1,\ldots, I_k)\in\mathcal P(I)$. For simplicity, write in this proof $\tau$ for $\tau_{(\# I_1,\ldots,\# I_k)}$.
\begin{enumerate}
\item If $\sigma\in S_n$ does $\sigma(I_j)=I_j$ for $j=1,\ldots,k$, then $\sigma(\rho)=\rho$ and $J_j^{\rho,\sigma}=\sigma(I_j)=I_j$ for $j=1,\ldots,k$. Hence both sides of \eqref{sym_dec_morphism_cond_simpl} are equal if and only if 
\[\prod_{j=1}^k\epsilon(\sigma, I_j)=\epsilon(\sigma, I).
\]
But this equation holds true  by \eqref{signs_general_preserving}.
\item Consider the case where $\sigma$ is a permutation preserving $\rho$, but not each of its elements. That is, here $\sigma(I_j)\in\rho$ for $j=1,\ldots, k$. Assume first the more special case where for some $j<l\in\{1,\ldots,k\}$, necessarily with $\#I_j=\#I_l$:
\begin{equation}\label{special_permutation_3}
\begin{split}
 &\sigma\an{I_j}\colon I_j\to I_l \text{ and } \sigma\an{I_l}\colon I_l\to I_j \text{ are order-preserving bijections and }\\
 &\sigma\an{I\setminus(I_j\cup I_l)}=\id_{I\setminus(I_j\cup I_l)}.
 \end{split}
 \end{equation}
 In this case $\sigma\an{I}$ is a product of $\#I_j$ transpositions, and so $\epsilon(\sigma, I)=(-1)^{\# I_j}$.
 On the other hand, $\epsilon(\sigma, I_j)=1=\epsilon(\sigma,I_l)$ because $\sigma$ is order-preserving on both sets, and $\epsilon(\sigma,I_s)=1$ for all $s\in\{1,\ldots,k\}\setminus\{j,l\}$ since 
 $\sigma$ is the identity on these sets.
Therefore 
\begin{equation}\label{sign_general_order_preserving_1}
\prod_{s=1}^l\epsilon(\sigma, I_s)=1=(-1)^{\# I_j}\epsilon(\sigma, I)
\end{equation}
and since $\rho=\sigma(\rho)$, 
\eqref{sym_dec_morphism_cond_simpl} is here 
 \begin{equation*}
\tau\left(a_{I_1}, \ldots, a_{I_l}, \ldots, a_{I_j}, \ldots, a_{I_k}\right)
=(-1)^{\#I_j}\tau\left(a_{I_1}, \ldots, a_{I_j}, \ldots, a_{I_l}, \ldots, a_{I_k}\right),
\end{equation*}
which holds since $\tau$ is symmetric in the $j$-entries for $j$ even, and 
 skew-symmetric in the $j$-entries for $j$ odd.

If $\sigma$ is now a general permutation preserving $\rho$, then its restriction to $I$ is a 
composition
\[\sigma\an{I}=(\sigma_r\circ\ldots\sigma_1\circ \lambda)\an{I}\]
with $\sigma_i\an{I}$ exchanging two of the elements of $\rho$ as in \eqref{special_permutation_3} for $i=1, \ldots, r$ and 
$\lambda\an{I}$ a permutation like in the first case and  in the proof of Lemma \ref{lemma_sign_partition_welldef}, i.e.~preserving each element of $\rho$.

Then if $\sigma_i$ exchanges two elements of cardinality $c_i$ for $i=1,\ldots,r$, then
\begin{equation*}
\begin{split}
\prod_{l=1}^k\epsilon(\sigma,I_l)&\overset{\eqref{prod1}}{=}\prod_{l=1}^k\left(\left(
\prod_{j=1}^r\epsilon(\sigma_j, \sigma_{j-1}\circ\ldots\circ\sigma_1(\lambda(I_l)))
\right)\cdot\epsilon(\lambda, I_l)\right)\\
&=\left(\prod_{j=1}^r\prod_{l=1}^k
\epsilon(\sigma_j, \sigma_{j-1}\circ\ldots\circ\sigma_1(\lambda(I_l)))
\right)\cdot \prod_{l=1}^k\epsilon(\lambda, I_l)\\
&\overset{\eqref{signs_general_preserving}, \eqref{sign_general_order_preserving_1}}{=}\left(\prod_{j=1}^r(-1)^{c_j}\epsilon(\sigma_j, \sigma_{j-1}\circ\ldots\circ\sigma_1(\lambda(I)))
\right)\cdot \epsilon(\lambda,I)\\
&\overset{\eqref{prod1}}{=}(-1)^{c_1+\ldots+c_r}\cdot\epsilon(\sigma,I),
\end{split}
\end{equation*}
which shows that \eqref{sym_dec_morphism_cond} follows in this case from \eqref{sym_dec_morphism_cond} for the simpler case above and the symmetry/skew-symmetry of $\tau$ on its different components.
\item Now take a general permutation $\sigma\in S_n$, but assume first that $\sigma\an{I_j}\colon I_j\to \sigma(I_j)$ is order-preserving for $j=1,\ldots, k$. The ordered partition $\sigma(\rho)$ consists in the subsets $\sigma(I_1), \ldots, \sigma(I_k)$, but reordered. Hence there is a permutation $\nu\in S_n$ that does 
\[ \nu(\sigma(I_l))=J_l^{\rho, \sigma} \text{ for } l=1, \ldots, k,
\]
and the restrictions of which to $\sigma(I_1), \ldots, \sigma(I_k)$ are all order-preserving. 
As before, consider the canonical partition $\rho^{\# I_1,\ldots,\# I_k}_{\rm can}:=(K_1,\ldots,K_k)$ defined by $(\# I_1,\ldots, \# I_k)$ on $\{1,\ldots, i\}$ for  $\# I=:i$.
Choose  a permutation $\lambda\in S_n$ 
with $\lambda(K_l)=I_l$ for $l=1,\ldots, k$ and  such that $\lambda\an{K_l}\colon K_l\to I_l$ is order-preserving for $l=1, \ldots, k$.
Then $\nu\sigma\lambda\in S_n$ does 
\[ (\nu\sigma\lambda)(K_l)=\nu(\sigma(I_l))=J_l^{\rho,\sigma} \quad \text{ for }  l=1,\ldots, k,
\]
and \[(\nu\sigma\lambda)\an{K_l}\colon K_l\to J_l^{\rho,\sigma} \quad \text{ is order preserving for }  l=1, \ldots, k.\]
Therefore 
\[ \sgn(\sigma(\rho))=\frac{\prod_{l=1}^k \epsilon(\nu\sigma\lambda, K_l)}{\epsilon(\nu\sigma\lambda, \underline i)}\overset{\eqref{prod1}}{=}\frac{1}{\epsilon(\nu, \sigma(I))\epsilon(\sigma, I)\epsilon(\lambda, \underline i)}
\]
and so 
\[ \frac{\sgn(\sigma(\rho))\cdot \epsilon(\sigma, I)}{\sgn(\rho)\cdot\prod_{l=1}^k\epsilon(\sigma, I_l)}=\frac{\epsilon(\nu, \sigma(I))\epsilon(\lambda, \underline i)}{\epsilon(\lambda,\underline i)}=\epsilon(\nu, \sigma(I))=\frac{\prod_{l=1}^k\epsilon(\nu, \sigma(I_l))}{\epsilon(\nu, \sigma(I))}.
\]
As in  the second step above, since $\nu$ is a composition of permutations $\sigma_r, \ldots, \sigma_1$ as in \eqref{special_permutation_3}. Conclude using the second case above.

Now a general $\tilde \sigma\in S_n$ is the composition $\tilde \sigma=\sigma\mu$ of a permutation $\mu$ that fixes $\rho$ with a permutation $\sigma$ as above, i.e.~that is order-preserving on each of the subsets $I_l$, for $l=1,\ldots, k$. Then $\tilde \sigma(\rho)=\sigma(\rho)$, $\tilde\sigma(I_l)=\sigma(I_l)$  and so also
$J_l^{\rho,\tilde\sigma}=J_l^{\rho,\sigma}$ for $l=1,\ldots, k$. Hence \eqref{sym_dec_morphism_cond_simpl}  follows from the first part of this case, using  \eqref{signs_general_preserving} and \eqref{prod1}.
\qedhere
\end{enumerate}
\end{proof}

\begin{proof}[Proof of Proposition \ref{morphisms_of_dSnVB}]
Recall that a morphism  $\tau$ of decomposed $n$-fold vector bundles is given as follows \cite{HeJo20} by its component $\tau(\nset)$ on the total space of $\mathbb E^{\mathcal A}$:
\[ \tau(\nset)((a_I)_{\emptyset\neq I\subseteq \nset})=\left(\sum_{\rho=(I_1,\ldots, I_k)\in\mathcal P(I)}\tau_\rho(a_{I_1}, \ldots, a_{I_k})\right)_{\emptyset\neq I\subseteq \nset}
\]
for all $(a_I)_{\emptyset\neq I\subseteq \nset}\in \mathbb E^{\mathcal A}(\nset)$, with vector bundle morphisms
\[ \tau_\rho\colon A_{\# I_1}\otimes \ldots\otimes A_{\# I_k}\to B_{\# I}
\]
over $\tau(\emptyset)\colon M\to N$, 
for $\emptyset\neq I\subseteq \nset$ and $\rho=(I_1,\ldots, I_k)\in\mathcal P(I)$.

For $\sigma\in S_n$ the morphism  $\tau(\nset)\circ \Psi_{\sigma\inv}^{\mathcal A}(\nset)$ sends $(a_I)_{\emptyset\neq I\subseteq \nset}$ to
\begin{equation*}
\begin{split}
\left(\tau\circ\Psi_{\sigma\inv}^{\mathcal A}\right)\left((a_I)_{\emptyset\neq I\subseteq \nset}\right)&=\tau\left(\left(\epsilon(\sigma, I)a_{\sigma(I)}\right)_{\emptyset\neq I\subseteq \nset}\right)\\
=\,& \left(\sum_{\rho=(I_1,\ldots, I_k)\in\mathcal P(I)}\tau_\rho\left(\epsilon(\sigma, I_1)a_{\sigma(I_1)}, \ldots, \epsilon(\sigma, I_k)a_{\sigma(I_k)}\right)\right)_{\emptyset\neq I\subseteq \nset}\\
&\left(\sum_{\rho=(I_1,\ldots, I_k)\in\mathcal P(I)}\prod_{l=1}^k\epsilon(\sigma, I_l)\cdot \tau_\rho\left(a_{\sigma(I_1)}, \ldots, a_{\sigma(I_k)}\right)\right)_{\emptyset\neq I\subseteq \nset},
\end{split}
\end{equation*}
while the morphism $\Psi_{\sigma\inv}^{\mathcal B}(\nset)\circ\tau(\nset)$ sends it to
\begin{equation*}
\begin{split}
\left(\Psi_{\sigma\inv}^{\mathcal B}\circ \tau\right)\left((a_I)_{\emptyset\neq I\subseteq \nset}\right)
=\,&\Psi_\sigma^{\mathcal B}\left(\left(\sum_{\rho=(I_1,\ldots, I_k)\in\mathcal P(I)}\tau_\rho(a_{I_1}, \ldots, a_{I_k})\right)_{\emptyset\neq I\subseteq \nset}
\right)\\
=\,&\left(\epsilon(\sigma, I)\cdot \sum_{\rho=(I_1,\ldots, I_k)\in\mathcal P(\sigma(I))}\tau_\rho(a_{I_1}, \ldots, a_{I_k})\right)_{\emptyset\neq I\subseteq \nset}\\
=\,&\left(\sum_{\rho=(I_1,\ldots, I_k)\in\mathcal P(I)}\epsilon(\sigma, I)\tau_{\sigma(\rho)}\left(a_{J^{\rho,
\sigma}_1}, \ldots, a_{J^{\rho,
\sigma}_k}\right)\right)_{\emptyset\neq I\subseteq \nset}.
\end{split}
\end{equation*}
where for $\rho\in\mathcal P(I)$ of length $k$, the tuple $\left(J^{\rho,\sigma}_1,\ldots,J^{\rho,\sigma}_k\right)$ is the ordered partition $\sigma(\rho)$ as in Lemma \ref{lemma_intermediate_morphisms}.

If $\tau_\rho=\sgn(\rho)\cdot \tau_{(\# I_1, \ldots, \# I_k)}$ for $\rho=(I_1,\ldots, I_k)\in\mathcal P(I)$, with $\tau_{(\# I_1, \ldots, \# I_k)}\colon A_{\# I_1}\otimes\ldots\otimes A_{\# I_k}\to B_{\# I}$ as in the statement of the theorem, then \eqref{sym_dec_morphism_cond_simpl} gives immediately $\tau(\nset)\circ \Psi_{\sigma\inv}^{\mathcal A}(\nset)=\Psi_{\sigma\inv}^{\mathcal B}(\nset)\circ\tau(\nset)$.
Hence a family of vector bundle morphisms $\tau_{(i_1,\ldots, i_k)}$ as in the statement defines a morphism of symmetric $n$-fold vector bundles.

\bigskip

It remains to show that any morphism of symmetric decomposed $n$-fold vector bundles must be of this form. If $\tau$ is a morphism of symmetric $n$-fold vector bundles, then 
$\left(\tau(\nset)\circ \Psi_{\sigma\inv}^{\mathcal A}(\nset)\right)\left( (a_J)_{\emptyset\neq J\subseteq \nset}\right)=\left(\Psi_{\sigma\inv}^{\mathcal B}(\nset)\circ\tau(\nset)\right)\left((a_J)_{\emptyset\neq J\subseteq \nset}\right)$ for all $(a_J)_{\emptyset\neq J\subseteq \nset}\in \mathbb E^{\mathcal A}(\nset)$ and all $\sigma\in S_n$.

Choose a subset $I\subseteq \nset$, an ordered partition $(I_1,\ldots, I_k)\in\mathcal P(I)$ and a permutation $\sigma\in S_n$. Evaluating the two sides of this formula on tuples $(a_J)_{\emptyset\neq J\subseteq \nset}$ with 
\[
a_J=0\text{ for } J\neq \sigma(I_1),\ldots, \sigma(I_k)
\]
yields 
\begin{equation}\label{sym_dec_morphism_cond}
 \prod_{l=1}^k\epsilon(\sigma, I_l)\cdot \tau_\rho\left(a_{\sigma(I_1)}, \ldots, a_{\sigma(I_k)}\right)
=\epsilon(\sigma, I)\tau_{\sigma(\rho)}\left(a_{J^{\rho,\sigma}_1}, \ldots, a_{J^{\rho,\sigma}_k}\right)
\end{equation}
for all $a_{\sigma(I_1)}\in A_{\# I_1}, \ldots, a_{\sigma(I_k)}\in A_{\# I_k}$ over a same point of $M$. 

A study of \eqref{sym_dec_morphism_cond} in two special cases yields the claim as follows.
\begin{enumerate}
\item Consider the case where $\sigma$ is a permutation preserving $\rho$, but not each of its elements. 
 That is, here $\sigma(I_j)\in\rho$ for $j=1,\ldots, k$. Assume first the more special case of
\eqref{special_permutation_3}.
 In this case 
\eqref{sign_general_order_preserving_1}
yields that  \eqref{sym_dec_morphism_cond} is 
 \begin{equation*}
\tau_\rho\left(a_{I_1}, \ldots, a_{I_j}, \ldots, a_{I_l}, \ldots, a_{I_k}\right)
=(-1)^{\#I_j}\tau_{\rho}\left(a_{I_1}, \ldots, a_{I_l}, \ldots, a_{I_j}, \ldots, a_{I_k}\right),
\end{equation*}
which shows that 
\[\tau_\rho\colon A_1^{\otimes \#\{s\in\{1,\ldots,k\}\mid \# I_s=1\}}\otimes \ldots\otimes A_n^{\otimes \#\{s\in\{1,\ldots,k\}\mid \# I_s=n\}}\to B_{\# I}\]
must be symmetric on $A_t^{\otimes \#\{s\in\{1,\ldots,k\}\mid \# I_s=t\}}$ for $t$ even, and skew-symmetric for $t$ odd.
\item Now choose a subset $I$ of $\nset$ with cardinality $i$ and choose a partition $\rho=(I_1,\ldots,I_k)\in\mathcal P(I)$. Set $i_j:=\#I_j$ for $j=1,\ldots,k$, and consider the canonical partition $\rho^{i_1,\ldots,i_k}_{\rm can}=(K_1,\ldots, K_k)$ of $\underline i$. Choose $\sigma\in S_n$ with
\[ \sigma(K_j)=I_j \quad \text{ for } \quad j=1,\ldots,k,
\]
i.e.~with $\rho=\sigma\left(\rho^{i_1,\ldots,i_k}_{\rm can}\right)$ and preserving the order of the partitions.  Set \[\tau_{(i_1,\ldots,i_k)}:=\tau_{\rho^{i_1,\ldots,i_k}_{\rm can}}\colon A_{i_1}\otimes\ldots\otimes A_{i_k}\to B_i.\]

Then \eqref{sym_dec_morphism_cond} applied to $\sigma$ and $\rho^{i_1,\ldots,i_k}_{\rm can}$ reads
\[\tau_\rho\left(a_{I_1}, \ldots, a_{I_k}\right)=\frac{\prod_{j=1}^k\epsilon(\sigma, K_j)}{\epsilon(\sigma, \underline i)}\cdot \tau_{i_1,\ldots,i_k}\left(a_{I_1}, \ldots, a_{I_k}\right)
=\sgn(\rho)\cdot \tau_{i_1,\ldots,i_k}\left(a_{I_1}, \ldots, a_{I_k}\right)
\]
for $a_{I_j}\in A_{i_j}$, $j=1, \ldots, k$.
 \qedhere
\end{enumerate}
 \end{proof}

\medskip

Next the composition of morphisms needs to be understood. Choose three decomposed symmetric $n$-fold vector bundles $\E^{\mathcal A}$, $\E^{\mathcal B}$ and $\E^{\mathcal C}$ defined by three families 
$\mathcal A=(A_i\to M \mid i=1,\ldots, n)$, $\mathcal B=(B_i\to M \mid i=1,\ldots, n)$ and $\mathcal C=(C_i\to M \mid i=1,\ldots, n)$ of vector bundles over smooth manifolds $M$, $N$ and $P$, respectively. Consider two morphisms 
\[ \eta=(\eta_p)_{p\in\mathcal P(n)}\colon \E^{\mathcal A}\to\E^{\mathcal B} \quad \text{ over }\quad \eta_0\colon M\to N
\]
and 
\[ \tau=(\tau_p)_{p\in\mathcal P(n)}\colon \E^{\mathcal B}\to\E^{\mathcal C} \quad \text{ over }\quad \tau_0\colon N\to P.
\]
The composition of $\tau$ with $\eta$ is given as follows. As stated in Proposition \ref{morphisms_of_dSnVB}, for $p=(i_1,\ldots, i_k)\in\mathcal P(n)$ with $\sum p=i$, and $a_l\in A_{i_l}$ for $l=1,\ldots,k$,
\[(\tau\circ\eta)_{(i_1,\ldots,i_k)}(a_{1},\ldots, a_{k})\in C_i=C_{\# \underline{i}}
\]
is the $\underline{i}$-entry of the image under $(\tau\circ\eta)(\nset)$ of the tuple $(a_J)_{\emptyset\neq J\subseteq \nset}$ with 
\[
a_J=\left\{\begin{array}{cl}0 &\text{ for } J\neq K_1,\ldots, K_k\\
 a_{l}&\text{ for } J=K_l, \,\, l\in\{1,\ldots, k\},
\end{array}\right.
\]
where $\rho^{i_1,\ldots,i_k}_{\rm can}:=(K_1,\ldots,K_k)$ is the canonical partition of $\underline i$ defined by $(i_1,\ldots, i_k)$.
Write $\rho^{i_1,\ldots,i_k}_{\rm can}=\rho$ for simplicity and compute
\begin{equation*}
\begin{split}
&(\tau\circ\eta)_{(i_1,\ldots,i_k)}(a_{1},\ldots, a_{k})= \left((\tau\circ\eta)(\nset)\left( (a_J)_{\emptyset\neq J\subseteq \nset}\right)\right)_{\underline i}\\
=    &\sum_{(J_1,\ldots, J_l)\in\operatorname{coars}(\rho)}
			\tau_{(J_1,\ldots,J_l)}
			\Bigl(\eta_{\rho\cap J_1}
			\bigl((a_{I})_{I\in\rho\cap J_1}\bigr),\ldots,
			\eta_{\rho\cap J_l}
			\bigl((a_I)_{I\in\rho\cap J_l}\bigr)\Bigr)\\
			=    &\sum_{\substack{\rho_1,\ldots, \rho_l \text{ ordered} \\
			\rho=\rho_1\cup\ldots\cup\rho_l\\
			\text{ as sets}\\
			\cup \rho_1\leq\ldots\leq\cup \rho_l}}
			\tau_{(\cup \rho_1,\ldots,\cup \rho_l)}
			\Bigl(\eta_{\rho_1}
			\bigl((a_{I})_{I\in\rho_1}\bigr),\ldots,
			\eta_{\rho_l}
			\bigl((a_I)_{I\in\rho_l}\bigr)\Bigr).
\end{split}
\end{equation*}
For $\rho_1,\ldots, \rho_l$ ordered partitions with $\rho=\rho_1\cup\ldots\cup\rho_l$ as sets and $\cup\rho_1< \ldots< \cup\rho_l$ 
the term 
\[
\tau_{(\cup \rho_1,\ldots,\cup \rho_l)}
			\Bigl(\eta_{\rho_1}
			\bigl((a_{I})_{I\in\rho_1}\bigr),\ldots,
			\eta_{\rho_l}
			\bigl((a_I)_{I\in\rho_l}\bigr)\Bigr)\]
			equals 
			\begin{equation*}
			\begin{split}
			\sgn(\cup\rho_1,\ldots,\cup\rho_l)\cdot\prod_{j=1}^l\sgn(\rho_j)\cdot\tau_{(\sum|\rho_1|,\ldots,\sum|\rho_l|)}
			\Bigl(\eta_{|\rho_1|}
			\bigl((a_{I})_{I\in\rho_1}\bigr),\ldots,
			\eta_{|\rho_l|}
			\bigl((a_I)_{I\in\rho_l}\bigr)\Bigr).
			\end{split}
			\end{equation*}
			This leads to the following result.
			\begin{proposition}\label{composition_sym_morphisms_nfold}
			Consider three decomposed symmetric $n$-fold vector bundles $\E^{\mathcal A}$, $\E^{\mathcal B}$ and $\E^{\mathcal C}$
over smooth 
manifolds $M$, $N$ and $P$, respectively, 
and two morphisms 
\[ \eta=(\eta_p)_{p\in\mathcal P(n)}\colon \E^{\mathcal A}\to\E^{\mathcal B} \quad \text{ over }\quad \eta_0\colon M\to N
\]
and 
\[ \tau=(\tau_p)_{p\in\mathcal P(n)}\colon \E^{\mathcal B}\to\E^{\mathcal C} \quad \text{ over }\quad \tau_0\colon N\to P
\]
of symmetric $n$-fold vector bundles.
The morphism
\[ \tau\circ\eta\colon \E^{\mathcal A}\to\E^{\mathcal C} \quad \text{ over }\quad \tau_0\circ\eta_0\colon M\to P
\]
of symmetric $n$-fold vector bundles 
is defined by 
\[ (\tau\circ\eta)_p=\!\!\!\!\!\!\!\!\! \sum_{\substack{\rho_1,\ldots, \rho_l \text{ ordered}\\
\text{such that } \rho_1\cup\ldots\cup \rho_l=\rho^p_{\rm can} \\
\text{ as sets}\\
 \cup\rho_1< \ldots<\cup\rho_l}}\sgn(\rho_1,\ldots,\rho_l)\cdot\tau_{\left(\Sigma |\rho_1|,\ldots,\Sigma |\rho_l|\right)}\circ \left(\eta_{|\rho_1|}, \ldots, \eta_{|\rho_l|}\right) 
\]
for all $p\in\mathcal P(n)$.
			\end{proposition}

\begin{proof}
It remains to show that for $(\rho_1,\ldots, \rho_l)$ as above with $\cup\rho_1< \ldots< \cup\rho_l$,
\[ \sgn(\rho_1,\ldots, \rho_l)=\sgn(\cup\rho_1, \ldots, \cup\rho_l)\cdot\prod_{j=1}^l\sgn(\rho_j).
\]
Set $(Q_1,\ldots,Q_l):=\rho^{(\#\cup\rho_1,\ldots,\#\cup\rho_l)}_{\rm can}$.
Then for each $j=1,\ldots, l$ the ordered partition $p_j:=|\rho_j|$ defines canonically a partition $(Q_j^1,\ldots, Q_j^{d_j})$ of $Q_j$.
Write $\rho_j=(I^1_j,\ldots, I^{d_j}_j)$. Then 
\[ (\rho_1,\ldots, \rho_l)=\left(I_1^1,\ldots, I_1^{d_1}, \ldots, I_l^1, \ldots, I_l^{d_l}\right).
\]
In the same manner, 
\[ \rho_{\rm can}^{(p_1,\ldots, p_l)}=\left(Q_1^1,\ldots, Q_1^{d_1}, \ldots, Q_l^1, \ldots, Q_l^{d_l}\right).
\]
Consider a permutation $\sigma\in S_i$ with $\sigma(Q_j^r)=I_j^r$
 for $j=1,\ldots, l$ and $r=1,\ldots, d_j$.
Then also $\sigma(Q_j)=\cup\rho_j$ for $j=1,\ldots, l$, and 
\begin{equation*}
\begin{split}
\sgn(\cup\rho_1,\ldots,\cup\rho_l)\cdot \prod_{j=1}^l\sgn(\rho_j)&=\frac{\cancel{\prod_{j=1}^l\epsilon(\sigma, Q_j)}}{\epsilon(\sigma, \underline i)}\cdot \prod_{j=1}^l\frac{\prod_{r=1}^{d_j}\epsilon(\sigma,Q_j^r)}{\cancel{\epsilon(\sigma, Q_j)}}\\
&= \frac{\prod_{j=1}^l\prod_{r=1}^{d_j}\epsilon(\sigma,Q_j^r)}{\epsilon(\sigma, \underline i)}=\sgn(\rho_1,\ldots, \rho_l).
\end{split}
\end{equation*}
This completes the proof.
\end{proof}

\begin{definition}
  The category of decomposed symmetric
  $n$-fold vector bundles defined in this section is written  $\mathsf{dSnVB}$.
\end{definition}

\subsection{Symmetric linear splittings and decompositions}

       Let $\E$ be a symmetric $n$-fold vector bundle. 
Consider the decomposed symmetric $n$-fold vector bundle $\E^{\rm dec}$ 
and the vacant decomposed symmetric $n$-fold vector bundle $\overline{\E}$ defined by $\E$ as in 
Definition \ref{decomposed_nfvb}.

\begin{definition}\label{def_n-dec}
Let $\mathbb E\colon \square^n\to \Man$ be a symmetric $n$-fold vector bundle.

A \textbf{symmetric linear splitting} of $\E$ is a monomorphism
$\Sigma\colon\overline{ \E}\to \E$ of symmetric $n$-fold vector bundles,
such that for $i=1,\ldots,n$, $\Sigma(\{i\})\colon \E(\{i\})\to \E(\{i\})$ is the
identity.

A \textbf{symmetric decomposition} of  $\E$ is a
natural isomorphism $\mathcal S\colon \E^{\rm dec}\to \E$
of symmetric $n$-fold vector bundles over the identity maps
$\mathcal S(\{i\})=\id_{\E(\{i\})}\colon \E(\{i\})\to \E(\{i\})$ such that
additionally the induced core morphisms $\mathcal S^I_I(I)$ are the
identities $\id_{\E^I_I}$ for all $I\subseteq\nset$.
\end{definition}

The following theorem implies that symmetric $n$-fold vector bundles always admit a symmetric $n$-fold vector bundle atlas. Its proof builds up on the results in \cite{HeJo20} and Section \ref{fix_of_3.6} and can be found in Appendix \ref{proof_dec_sym_nvb}.
  \begin{theorem}\label{dec_sym_nvb_thm}
    A symmetric $n$-fold  vector bundle always admits a symmetric decomposition.
    \end{theorem}

\subsection{Symmetric $n$-fold vector bundle atlases}
This section finally discusses the local approach to symmetric $n$-fold vector bundles, i.e.~in the language of $n$-fold vector bundle atlases and cocycles.

\begin{definition}\label{sym_atlas_def}
  An $n$-fold vector bundle atlas $\{c_\alpha=(U_\alpha,\Theta_\alpha,(V_I)_{\emptyset\neq I\subseteq \nset}) \mid \alpha\in\Lambda\}$ as in Definition \ref{n_fold_vb_atlas_def} on a smooth surjective submersion $\pi\colon E\to M$ is called \textbf{symmetric} if
  \begin{enumerate}
  \item the model vector spaces satisfy $V_I=V_J=:V^{\# I}$ whenever $\# I=\# J$
  \item for all $\alpha,\beta\in\Lambda$, $\emptyset\neq I,J\subseteq \underline n$ with
    $\# I=\# J$ and $\rho_1\in\mathcal P(I)$,
    $\rho_2\in\mathcal P(J)$ with same size
    $l_{\rho_1}=l_{\rho_2}=(l_1,\ldots,l_k)$:
\[\sgn(\rho_1)\cdot\omega_{\alpha\beta}^{\rho_1}=\sgn(\rho_2)\cdot \omega_{\alpha\beta}^{\rho_2}.\]
\end{enumerate}
Set then
$\omega_{\alpha\beta}^{l_1,\ldots,l_k}:=\omega_{\alpha\beta}^{\rho^{l_1,\ldots,l_k}_{\rm can}}\in C^\infty(U_\alpha\cap U_\beta,\operatorname{Hom}(V^{l_1}\otimes\ldots \otimes V^{l_k},V^{l_1+\ldots+l_k}))$ for each ordered integer partition $(l_1,\ldots, l_k)\in\mathcal P(n)$ and all $\alpha,\beta\in\Lambda$.
  \begin{enumerate}\setcounter{enumi}{2}
  \item The forms $\omega_{\alpha\beta}^{l_1,\ldots,l_k}$ are elements
    of
    $C^\infty(U_\alpha\cap U_\beta,
    \operatorname{Hom}(V^{l_1}\odot\ldots\odot V^{l_k},
    V^{l_1+\ldots+l_k}))$, where  $V^{l_i}\odot V^{l_{i+1}}$
    means \begin{itemize}
\item $\otimes$ if $l_i< l_{i+1}$, 
\item $\wedge$ if $l_i= l_j$ is an odd number,
\item the symmetric product $\cdot$ if $l_i= l_j$ is an even number.
\end{itemize}
\end{enumerate}
\end{definition}
 Hence here for $\alpha,\beta\in\Lambda$, the change of chart
$\Theta_\alpha\circ\Theta_\beta\inv\colon (U_\alpha\cap
U_\beta)\times\prod_{\emptyset\neq I\subseteq \underline n}V^{\# I}\to (U_\alpha\cap
U_\beta)\times\prod_{\emptyset \neq I\subseteq \underline n}V^{\# I}$ reads
 \begin{equation}\label{change_of_charts_sym_atlas}
 \left(p,(v_I)_{\emptyset\neq I\subseteq \underline n}\right)\mapsto \left(p,
     \left(\sum_{\rho=(I_1,\ldots,I_k)\in\mathcal
         P(I)}\sgn(\rho)\cdot \omega_{\alpha\beta}^{\# I_1,\ldots, \# I_k}(p)(v_{I_1},\ldots,
       v_{I_k})\right)_{\emptyset\neq I\subseteq \underline n}\right).
 \end{equation}

 By construction, the space $E$ as in the definition above is locally diffeomorphic to spaces of the form
 \[ U\times \prod_{\emptyset\neq I\subseteq \nset} V^{\# I}.
 \]
 These are the total spaces of decomposed symmetric $n$-fold vector bundles defined as in Section \ref{dec_sym_nvb}
 by the family of vector bundles $(U\times V^{i}\to U)_{1\leq i\leq n}$. Further, by \eqref{change_of_charts_sym_atlas} and Proposition \ref{morphisms_of_dSnVB} the changes of charts $\Theta_\alpha\circ\Theta_{\beta}\inv$
 are morphisms of symmetric decomposed $n$-fold vector bundles 
 \[ (U_\alpha\cap U_\beta)\times \prod_{\emptyset\neq I\subseteq \nset} V^{\# I} \to (U_\alpha\cap U_\beta)\times \prod_{\emptyset\neq I\subseteq \nset} V^{\# I}.
 \]

 Hence there is an ``indirect''
 (i.e.~\emph{signed}, see \cite{Pradines77} in the case $n=2$) right $S_n$-action on any manifold endowed with a
 symmetric $n$-fold vector bundle atlas as above, and  the following proposition holds.
\begin{proposition}\label{S_n_action_symmetric_atlas}
  A smooth manifold $E$ endowed with a symmetric $n$-fold vector bundle atlas as in Definition \ref{sym_atlas_def}
  has a left $S_n$-action, given in a chart $\pi\inv(U_\alpha)$ by
\[\Phi_\sigma^\alpha\colon U_\alpha\times \prod_{\emptyset\neq I\subseteq \underline n}V^{\# I}\to U_\alpha\times \prod_{\emptyset\neq I\subseteq \underline n}V^{\# I},\quad 
  \Phi_\sigma^\alpha\left(p, (v_I)_{\emptyset\neq I\subseteq\underline n}\right)=\left(p,
  \left(\epsilon\left(\sigma\inv,I\right)v_{\sigma\inv(I)}\right)_{\emptyset\neq I\subseteq\underline n}\right).
\]
\end{proposition}
\begin{proof}
For each $\sigma\in S_n$ and all $\alpha,\beta\in\Lambda$ the equality $\Phi_\sigma^\alpha\circ
\Theta_\alpha\circ\Theta_\beta\inv
=\Theta_\alpha\circ\Theta_\beta\inv\circ \Phi_\sigma^\beta$ holds by  \eqref{change_of_charts_sym_atlas} and Proposition \ref{morphisms_of_dSnVB}.
Hence the collection of maps $\{\Phi^\alpha_\sigma\}_{\alpha\in \Lambda}$ defines a global map $\Phi_\sigma\colon E\to E$.

It remains to check that the obtained map $\Phi\colon S_n\times E\to E$
is a left action. 
It is enough to verify that
\[\Phi^\alpha\colon S_n\times U_\alpha\times \prod_{\emptyset\neq I\subseteq \underline
  n}V^{\# I}\to U_\alpha\times \prod_{\emptyset\neq I\subseteq \underline n}V^{\#
  I}\] is a left action for each $\alpha\in\Lambda$. But this is the computation in \eqref{decomposed_Sn_action}.   \end{proof}

\begin{theorem}\label{symmetric_atlas_theorem}
	An $n$-fold vector bundle $\E$ admits a symmetric structure $\Phi$  if and 
	only if $\E(\nset)$ can be endowed with a symmetric atlas, such that the induced $S_n$-action as in Proposition \ref{S_n_action_symmetric_atlas} is $\Phi(\nset)$.
\end{theorem}
\begin{proof}
This proof works as the proof of Corollary \ref{eq_defs_nvb}, but with the additional $S_n$-symmetry taken into account.

\medskip
First let $\mathbb E$ be a symmetric $n$-fold vector bundle. By Theorem \ref{dec_sym_nvb_thm}, there is a decomposition
$\mathcal S\colon \E^{\rm dec}\to \E$ of $\E$, with $\E^{\rm dec}$ the \emph{symmetric} decomposed $n$-fold vector bundle defined by 
the family $\left(q_l\colon A_l:=\E^{\lset}_{\lset}\to M\right)_{1\leq l\leq n}$ of vector bundles over $M$.  Set $E:=\E(\nset)$, as usual $M:=\E(\emptyset)$, and 
$\pi\colon \E(\nset\to\emptyset)\colon E\to M$. For each $l\in\{1,\ldots, n\}$,
set $V^l:=\mathbb R^{\dim A_l}$, the vector space on which $A_l=\E^{\lset}_{\lset}$ is
modelled. Take a covering $\{U_\alpha\}_{\alpha\in\Lambda}$ of $M$ by
open sets trivialising all the vector bundles $A_1, \ldots, A_n$;
\[ \phi^\alpha_l\colon q_{l}^{-1}(U_\alpha)\overset{\sim}{\longrightarrow} U_\alpha\times V^l
\]
for all  $1\leq l\leq n$ and all $\alpha\in\Lambda$.
Then define $n$-fold vector bundle charts $\Theta_\alpha\colon\pi^{-1}(U_\alpha)\to U_\alpha\times \prod_{\emptyset\neq I\subseteq \nset}V^{\# I}$ by 
\[\Theta_\alpha=\left(\pr_M\times \left(\phi^\alpha_{\# I}\right)_{\emptyset \neq I\subseteq \nset}\right)\circ \mathcal S(\nset)^{-1}\an{\pi^{-1}(U_\alpha)}\colon \pi^{-1}(U_\alpha)\to U_\alpha\times  \prod_{\emptyset\neq I\subseteq \nset}V^{\# I}.
\]
Given $\alpha,\beta\in\Lambda$ with $U_\alpha\cap U_\beta\neq\emptyset$, the change of chart
\[\Theta_\alpha\circ\Theta_\beta\inv\colon (U_\alpha\cap U_\beta)\times \prod_{\emptyset\neq I\subseteq \nset}V^{\# I}\to  (U_\alpha\cap U_\beta)\times \prod_{\emptyset\neq I\subseteq \nset}V^{\# I}
\]
is given by 
\begin{equation}\label{simple_change_of_charts_sym}
\left(p, (v_I)_{\emptyset \neq I\subseteq \nset}\right)\mapsto \left(p, \left(\omega_{\alpha\beta}^{\# I}(p)v_I\right)_{\emptyset \neq I\subseteq \nset}\right),
\end{equation}
with
$\omega^{\# I}_{\alpha\beta}\in C^\infty(U_\alpha\cap U_\beta,
\operatorname{Gl}(V^{\# I}))$
the cocycle defined by $\phi^\alpha_{\# I}\circ (\phi_{\# I}^\beta)\inv$. The
two charts are hence smoothly compatible and $\lie A=\{(U_\alpha,\Theta_\alpha, (V^{\# I})_{\emptyset\neq I\subseteq\nset})\mid \alpha\in
\Lambda\}$ is a symmetric $n$-fold vector bundle atlas on $E$. It remains to check that this is a \emph{symmetric} $n$-fold vector bundle atlas. But this is immediate since by \eqref{simple_change_of_charts_sym}
the morphism $\omega_\rho^{\alpha,\beta}\in C^\infty(U_\alpha\cap U_\beta, \Hom(V^{\# I_1}\otimes \ldots \otimes V^{\# I_k}, V^{\# I}))$ is trivial for $\emptyset\neq I\subseteq \nset$ and $\rho=(I_1,\ldots, I_k)$ an ordered partition of $I$ \emph{with $k\geq 2$}.

Finally, the action $\Phi^\alpha$ of $S_n$ on $U_\alpha\times\prod_{\emptyset\neq I\subseteq \nset}V^{\# I}$ is such that the diagram
% https://q.uiver.app/?q=WzAsNixbMCwwLCJcXFBpXnstMX0oVV9cXGFscGhhKSJdLFsyLDAsIlxcUGleey0xfShVX1xcYWxwaGEpIl0sWzAsMiwiXFxQaV97XFxlbXB0eXNldFxcbmVxIElcXHN1YnNldGVxIFxcbnNldH1BX3tcXCMgSX1cXGFue1VfXFxhbHBoYX0iXSxbMiwyLCJcXFBpX3tcXGVtcHR5c2V0XFxuZXEgSVxcc3Vic2V0ZXEgXFxuc2V0fUFfe1xcIyBJfVxcYW57VV9cXGFscGhhfSJdLFswLDQsIlVfXFxhbHBoYVxcdGltZXNcXFBpX3tcXGVtcHR5c2V0XFxuZXEgSVxcc3Vic2V0ZXEgXFxuc2V0fVZee1xcIyBJfSJdLFsyLDQsIlVfXFxhbHBoYVxcdGltZXNcXFBpX3tcXGVtcHR5c2V0XFxuZXEgSVxcc3Vic2V0ZXEgXFxuc2V0fVZee1xcIyBJfSJdLFswLDIsIlxcbWF0aGNhbCBTKFxcbnNldClcXGFue1xcUGlcXGludihVX1xcYWxwaGEpfSIsMl0sWzIsNCwiXFxQaVxcdGltZXMgXFxsZWZ0KFxccGhpXlxcYWxwaGFfe1xcIyBJfVxccmlnaHQpX3tcXGVtcHR5c2V0IFxcbmVxIElcXHN1YnNldGVxIFxcbnNldH0iLDJdLFswLDEsIlxcUGhpX1xcc2lnbWEoXFxuc2V0KVxcYW57XFxQaVxcaW52KFVfXFxhbHBoYSl9Il0sWzIsMywiXFxQc2lfXFxzaWdtYShcXG5zZXQpXFxhbntVX1xcYWxwaGF9Il0sWzEsMywiXFxtYXRoY2FsIFMoXFxuc2V0KVxcYW57XFxQaVxcaW52KFVfXFxhbHBoYSl9Il0sWzQsNSwiXFxQaGleXFxhbHBoYSIsMl0sWzMsNSwiXFxQaVxcdGltZXMgXFxsZWZ0KFxccGhpXlxcYWxwaGFfe1xcIyBJfVxccmlnaHQpX3tcXGVtcHR5c2V0IFxcbmVxIElcXHN1YnNldGVxIFxcbnNldH0iXV0=
\[\begin{tikzcd}
	{\pi^{-1}(U_\alpha)} && {\pi^{-1}(U_\alpha)} \\
	\\
	{\prod_{\emptyset\neq I\subseteq \nset}^MA_{\# I}\an{U_\alpha}} && {\prod_{\emptyset\neq I\subseteq \nset}^MA_{\# I}\an{U_\alpha}} \\
	\\
	{U_\alpha\times\prod_{\emptyset\neq I\subseteq \nset}V^{\# I}} && {U_\alpha\times\prod_{\emptyset\neq I\subseteq \nset}V^{\# I}}
	\arrow["{\mathcal S(\nset)\an{\pi\inv(U_\alpha)}}"', from=1-1, to=3-1]
	\arrow["{\pr_M\times \left(\phi^\alpha_{\# I}\right)_{\emptyset \neq I\subseteq \nset}}"', from=3-1, to=5-1]
	\arrow["{\Phi_\sigma(\nset)\an{\pi\inv(U_\alpha)}}", from=1-1, to=1-3]
	\arrow["{\Psi_\sigma^{\rm dec}(\nset)}", from=3-1, to=3-3]
	\arrow["{\mathcal S(\nset)\an{\pi\inv(U_\alpha)}}", from=1-3, to=3-3]
	\arrow["{\Phi^\alpha_\sigma}"', from=5-1, to=5-3]
	\arrow["{\pr_M\times \left(\phi^\alpha_{\# I}\right)_{\emptyset \neq I\subseteq \nset}}", from=3-3, to=5-3]
\end{tikzcd}\]
commutes, since the decomposition $\mathcal S$ is symmetric. Hence $\Phi^\alpha_\sigma$ is $\Phi_\sigma(\nset)$ in the chart $\Theta_\alpha$.

\medskip

Conversely, given a space $E$ with a symmetric  $n$-fold vector bundle structure
over a smooth manifold $M$ as in Definition \ref{sym_atlas_def},
define $\mathbb E\colon \square^{\N}\to \Man$ as follows.  Take a
maximal symmetric $n$-fold vector bundle atlas
\[\lie A=\left\{\left.\left(U_\alpha,\Theta_\alpha, \left(V^{\# I}\right)_{I\subseteq\nset}\right)\right| \, \alpha\in
\Lambda\right\}\]
and set  $\mathbb E(\nset)=E$, $\mathbb E(\emptyset)=M$, 
and more generally for $\emptyset\neq I\subseteq \nset$, 
\[\mathbb
E(I)=\left.\left(\bigsqcup_{\alpha\in\Lambda}\left(U_\alpha\times\prod_{\emptyset\neq J\subseteq
  I}V^{\# J}\right)\right)\right/\sim \]
with $\sim$ the equivalence relation defined on
$\bigsqcup_{\alpha\in\Lambda}(U_\alpha\times\prod_{\emptyset\neq J\subseteq I}V^{\# J})$
by 
\[ U_\alpha\times\prod_{\emptyset\neq J\subseteq I}V^{\# J}\quad \ni\quad
\left(p,(v_J)_{\emptyset\neq J\subseteq I}\right)\quad \sim \quad
\left(q,(w_J)_{\emptyset\neq J\subseteq I}\right)\quad \in \quad U_\beta\times\prod_{\emptyset\neq J\subseteq I}V^{\# J}
\]
if and only if $p=q$ and
\[v_J=\sum_{\rho=(J_1,\ldots,J_k)\in\mathcal
    P(J)}\sgn(\rho)\omega_{\alpha\beta}^{\# J_1,\ldots,\# J_k}(p)(w_{J_1},\ldots,w_{J_k})
\]
for $\emptyset\neq J\subseteq I$.
Let $\pr_I\colon \bigsqcup_{\alpha\in\Lambda}\left(U_\alpha\times\prod_{\emptyset\neq J\subseteq
  I}V^{\# J}\right)\to \E(I)$ be the quotient map.
Then,  as in the proof of Corollary \ref{eq_defs_nvb}, $\mathbb E(I)$ has a unique smooth
manifold structure such that
$\pr_I\colon \mathbb E(I)\to M$, $\pr_I[p,(v_I)_{\emptyset\neq I\subseteq J}]=p$
is a surjective submersion and such that
 the maps 
\[\Theta^I_\alpha\colon \pr_I\left(U_\alpha\times\prod_{\emptyset\neq J\subseteq
  I}V^{\#J}\right)\to U_\alpha\times\prod_{\emptyset\neq J\subseteq I}V^{\# J}, \qquad
\left[p,(v_I)_{\emptyset\neq I\subseteq J}\right]\mapsto \left(p,(v_I)_{\emptyset\neq I\subseteq J}\right)
\]
are diffeomorphisms. This leads to an $n$-fold vector bundle $\E\colon\square^n\to \Man$ as in the proof of Corollary \ref{eq_defs_nvb}.
The projections $\pr_I$ for $I\subseteq \nset$ define a morphism $\pr\colon \widetilde \E\to \E$ of $n$-fold vector bundles, 
where the $n$-fold vector bundle $\widetilde\E\colon \square^n\to \Man$ is defined\footnote{If $\Lambda$ is not countable, the images of the objects of $\square^n$ under the functor $\widetilde\E$ are not second-countable topological spaces. However, this is not relevant for the use made of $\widetilde\E$ in this proof.} by 
\[ \widetilde{\E}(I)=\bigsqcup_{\alpha\in\Lambda}\left(U_\alpha\times\prod_{\emptyset\neq J\subseteq
  I}V^{\# J}\right).
\]

  It remains to show that $\E$ is symmetric. First, by construction, for $\emptyset\neq I\subseteq \nset$
  \[
  \E^I_I=
\left.\left(\bigsqcup_{\alpha\in\Lambda}\left(U_\alpha\times V^{\# I}\right)\right)\right/\sim \]
with $\sim$ the equivalence relation defined on $\bigsqcup_{\alpha\in\Lambda}\left(U_\alpha\times V^{\# I}\right)$
by 
\[ U_\alpha\times V^{\# I}\quad \ni\quad
\left(p,v\right)\quad \sim \quad
\left(q,w\right)\quad \in \quad U_\beta\times V^{\# I}
\]
if and only if $p=q$ and
$v=\omega_{\alpha\beta}^{\# I}(p)(w)$. Hence the equality $\E_I^I=\E_{\sigma(I)}^{\sigma(I)}$ for $\sigma\in S_n$ is immediate.

$\widetilde\E$ is a symmetric $n$-fold vector bundle by Section \ref{dec_sym_nvb}. Let its $S_n$-action be denoted by $\widetilde\Phi$.
It is easy to check with Lemma \ref{lemma_intermediate_morphisms} that \[ U_\alpha\times\prod_{\emptyset\neq J\subseteq I}V^{\# J}\quad \ni\quad
\left(p,(v_J)_{\emptyset\neq J\subseteq I}\right)\quad \sim \quad
\left(q,(w_J)_{\emptyset\neq J\subseteq I}\right)\quad \in \quad U_\beta\times\!\prod_{\emptyset\neq J\subseteq I}V^{\# J}
\]
if and only if for all $\sigma\in S_n$
\[ U_\alpha\times\!\!\!\prod_{\emptyset\neq K\subseteq \sigma(I)}V^{\# K} \ni
\widetilde{\Phi_\sigma}(I)\left(p,(v_J)_{\emptyset\neq J\subseteq I}\right) \sim
\widetilde{\Phi_\sigma}(I)\left(q,(w_J)_{\emptyset\neq J\subseteq I}\right) \in  U_\beta\times\!\!\!\prod_{\emptyset\neq K\subseteq \sigma(I)}V^{\# K}.
\]
As a consequence, for each $\sigma\in S_n$ and each $\emptyset\neq I\subseteq \nset$, the map $\widetilde\Phi_\sigma(I)\colon \widetilde\E(I)\to \widetilde\E^\sigma(I)$ factors to a map $\Phi_\sigma(I)\colon \E(I)\to \E^\sigma(I)$ such that 
% https://q.uiver.app/?q=WzAsNCxbMCwwLCJcXHdpZGV0aWxkZVxcRShJKSJdLFsxLDAsIlxcd2lkZXRpbGRlXFxFXlxcc2lnbWEoSSkiXSxbMCwxLCJcXEUoSSkiXSxbMSwxLCJcXEVeXFxzaWdtYShJKSJdLFswLDEsIlxcd2lkZXRpbGRlXFxQaGlfXFxzaWdtYShJKSJdLFswLDIsIlxccHJfSSIsMl0sWzEsMywiXFxwcl97XFxzaWdtYShJKX0iXSxbMiwzLCJcXFBoaV9cXHNpZ21hKEkpIiwyXV0=
\[\begin{tikzcd}
	{\widetilde\E(I)} & {\widetilde\E^\sigma(I)} \\
	{\E(I)} & {\E^\sigma(I)}
	\arrow["{\widetilde\Phi_\sigma(I)}", from=1-1, to=1-2]
	\arrow["{\pr_I}"', from=1-1, to=2-1]
	\arrow["{\pr_{\sigma(I)}}", from=1-2, to=2-2]
	\arrow["{\Phi_\sigma(I)}"', from=2-1, to=2-2]
\end{tikzcd}\]
commutes.
Further, since $(\widetilde\Phi_\sigma(I))^I_ I=\epsilon(\sigma, I)\cdot \id_{\widetilde\E^I_ I}\colon \widetilde \E^I_I\to \widetilde \E^{\sigma(I)}_{\sigma(I)}$, also
$\Phi_\sigma(I)^I_I=\epsilon(\sigma, I)\cdot \id_{\E^I_I}$.
These morphisms define hence an $S_n$-symmetry on $\E$.
Since the spaces $U_\alpha\times\prod_{\emptyset\neq J\subseteq \nset}V^{\# J}$ are charts of $\E(\nset)$ and the $S_n$-action on $\E(\nset)$ is given by the action $\widetilde\Phi$ restricted to this chart, it coincides with the action of $S_n$ on $E$ as in Proposition \ref{S_n_action_symmetric_atlas}.
\end{proof}

As usual, it is clear from the proof above and from Proposition \ref{composition_sym_morphisms_nfold} that the information of a symmetric $n$-fold vector bundle atlas is completely encoded in the \emph{symmetric $n$-fold vector bundle cocycle} that it defines.
\begin{definition}
Let $M$ be a smooth manifold. Then a symmetric $n$-fold vector bundle cocycle consists in an open cover $(U_\alpha)_{\alpha\in\Lambda}$ of $M$, together with a family of isomorphisms 
\[ \omega^{\alpha\beta}\colon U_{\alpha\beta}\times \prod_{\emptyset\neq J\subseteq \nset}V^{\# J}\to U_{\alpha\beta}\times \prod_{\emptyset\neq J\subseteq \nset}V^{\# J}, \qquad \qquad \alpha,\beta\in\Lambda
\]
 of decomposed symmetric $n$-fold vector bundles over the identity on $U_{\alpha\beta}$, with $V^1, \ldots, V^n$ fixed vector spaces, such that 
\[ \omega^{\alpha\beta}=\omega^{\alpha\gamma}\circ\omega^{\gamma\beta}\colon U_{\alpha\gamma\beta}\times \prod_{\emptyset\neq J\subseteq \nset}V^{\# J}\to U_{\alpha\gamma\beta}\times \prod_{\emptyset\neq J\subseteq \nset}V^{\# J}
\]
for all $\alpha, \beta,\gamma\in \Lambda$, 
and 
\[ \omega^{\alpha\alpha}=\id\colon U_{\alpha}\times \prod_{\emptyset\neq J\subseteq \nset}V^{\# J}\to U_{\alpha}\times \prod_{\emptyset\neq J\subseteq \nset}V^{\# J}
\]
for all $\alpha\in\Lambda$.
\end{definition}
In other words for all of $\alpha,\beta$ in $\Lambda$, 
\[ \omega^{\alpha\beta}=\left(\omega^{\alpha\beta}_{(i_1,\ldots, i_l)}\in C^\infty(U_{\alpha\beta}, \Hom(V^{i_1}\odot\ldots\odot V^{i_l}, V^{i_1+\ldots+i_l}))\right)_{(i_1,\ldots, i_l)\in\mathcal P(n)},
\]
with the composition given as in Proposition \ref{composition_sym_morphisms_nfold}: for all $\alpha,\beta,\gamma\in \Lambda$ and all canonically ordered integer partitions $(i_1,\ldots, i_l)\in\mathcal P(n)$
\begin{equation*}
\begin{split}
& \omega^{\alpha\beta}(p)_{(i_1,\ldots, i_l)}=(\omega^{\alpha\gamma}(p)\circ\omega^{\gamma\beta}(p))_{(i_1,\ldots, i_l)}\\
 &=\!\!\!\!\!\!\!\!\! \sum_{\substack{\rho_1,\ldots, \rho_s \text{ ordered}\\
\text{such that } \rho_1\cup\ldots\cup \rho_s=\rho^{(i_1,\ldots, i_l)}_{\rm can} \\
\text{ as sets}\\
 \cup\rho_1< \ldots<\cup\rho_s}} \sgn(\rho_1,\ldots,\rho_s)\cdot\omega^{\alpha\gamma}_{\left(\Sigma |\rho_1|,\ldots,\Sigma |\rho_s|\right)}\circ \left(\omega^{\gamma\beta}_{|\rho_1|}, \ldots, \omega^{\gamma\beta}_{|\rho_s|}\right),
\end{split}
\end{equation*}
for $p\in U_{\alpha\beta\gamma}$,  and for $p\in U_\alpha$ 
\[  \omega^{\alpha\alpha}(p)_{(i_1,\ldots, i_l)}=\left\{\begin{array}{cc}
\id_{V^{i_1}} & l=1\\
0& l\geq 2
\end{array}\right..
\]

\section{The equivalence of symmetric $n$-fold vector bundles with
  $[n]$-manifolds}\label{algebraisation}
This section constructs the functor from the category of
symmetric $n$-fold vector bundles to the category of $[n]$-manifolds, and the functor 
from the category of $[n]$-manifolds to the category of symmetric $n$-fold vector bundles.
The key to both functors is the equality is the obvious bijection between $[n]$-manifold cocycles and symmetric $n$-fold vector bundle cocycles.
Using this, both functors are constructed in the same manner.

\medskip

Begin by considering a symmetric $n$-fold vector bundle $\mathbb E\colon \square^n\to \Man$, over a smooth manifold $M:=\E(\emptyset)$. Since $\E$ is symmetric, 
it defines $n$ vector bundles $A_1,\ldots, A_n$ over $M$, such that for $i=1,\ldots, n$
\[ A_i:=\E_I^I \quad \text{ for all } \quad I\subseteq \nset \text{ with } \# I=i.
\]
Let $r_i$ be the rank of $A_i$ for $i=1,\ldots,n$ and set $V^i:=\R^{r_i}$ for simplicity.
By Theorem \ref{symmetric_atlas_theorem} the manifold $\mathbb E(\nset)$ is equipped with a symmetric $n$-fold vector bundle atlas $\{c_\alpha=(U_\alpha, \Theta_\alpha)\mid \alpha\in\Lambda\}$ modeled on the vector spaces $\mathbb R^{r_1}, \ldots, \R^{r_n}$, and hence a symmetric $n$-fold vector bundle cocycle
\[ (U_\alpha)_{\alpha\in\Lambda}, \qquad \left(\omega^{\alpha\beta}:=\Phi_\alpha\circ\Phi_\beta\inv\colon U_{\alpha\beta}\times\prod_{\emptyset\neq J\subseteq \nset}V^{\# J}\to U_{\alpha\beta}\times\prod_{\emptyset\neq J\subseteq \nset}V^{\# J}\right)_{\alpha,\beta\in\Lambda}.
\]
Assume that the atlas is maximal.
Recall that for each pair of indices $\alpha, \beta\in\Lambda$, 
the morphism $\omega^{\alpha\beta}$ amounts to a collection
\[ \omega^{\alpha\beta}=\left( \omega^{\alpha\beta}_{(i_1,\ldots, i_l)}\in\Hom(V^{i_1}\odot\ldots\odot V^{i_l}, V^{i_1+\ldots+i_l})\right)_{(i_1,\ldots, i_l)\in\mathcal P(n)}. 
\]
By Proposition \ref{composition_sym_morphisms_nfold} and \eqref{composition_graded_morphisms}, the $n$-fold vector bundle cocycle defines hence the 
$[n]$-manifold cocycle 
\[ \left( (U_\alpha)_{\alpha\in\Lambda}, \left( \omega^{\alpha\beta}_{(i_1,\ldots, i_l)}\in\Hom(V^{i_1}\odot\ldots\odot V^{i_l}, V^{i_1+\ldots+i_l})\right)_{(i_1,\ldots, i_l)\in\mathcal P(n)}\right)\]
Denote the corresponding $[n]$-manifold by $\mathcal A(\E)$. This defines the map on objects of an \emph{algebraisation functor} 
\[ \mathcal A\colon  \mathsf{SnVB} \to \mathsf{[n]Man}
\]
Consider a morphism $\Phi\colon \E\to \F$ of symmetric $n$-fold vector bundles. Let as before $(U_\alpha, \phi_\alpha \mid \alpha\in\Lambda)$ be a symmetric $n$-fold vector bundle atlas for $\E$ modeled as above on $\R^{r_1}=:V^1, \ldots, \R^{r_n}=:V^n$, and let $(V_{\alpha'}, \psi_{\alpha'}\mid \alpha'\in \Lambda')$ be a symmetric $n$-fold vector bundle atlas for $\F$ modeled on $\R^{s_1}=:W^1, \ldots, \R^{s_n}=:W^n$. Then for each $\alpha\in\Lambda$ and $\alpha'\in\Lambda'$ the map $\Phi(\nset)\colon \E(\nset)\to \F(\nset)$ defines 
\[ \Phi_{\alpha'\alpha}:=\psi_{\alpha'}\circ \Phi(\nset)\circ \phi_{\alpha}\inv\colon (\Phi_0^{-1}(V_{\alpha'})\cap U_\alpha)\times \prod_{\emptyset\neq J\subseteq \nset}V^{\# J}\to V_{\alpha'}\times \prod_{\emptyset\neq J\subseteq \nset}W^{\# J},
\]
which is a morphism of symmetric $n$-fold vector bundles and hence equivalent to a collection
\[ \left( \Phi_{\alpha'\alpha}^{(i_1,\ldots, i_l)}\in C^\infty(\Phi_0^{-1}(V_\alpha')\cap U_\alpha, \Hom( V^{i_1}\odot\ldots\odot V^{i_l}, W^{i_1+\ldots+i_l}))\right)_{(i_1,\ldots, i_l)\in \mathcal P(n)}.
\]
For $\alpha, \beta\in\Lambda$, $\alpha', \beta'\in\Lambda'$ and all $x\in \Phi_0\inv(V_{\alpha'\beta'})\cap U_{\alpha\beta}$
\begin{equation}\label{morphism_cocycle}
\Phi_{\beta'\beta}(x)=\Psi_{\beta'\alpha'}(\Phi_0(x))\circ \Phi_{\alpha'\alpha}(x)\circ \phi_{\alpha\beta}(x).
\end{equation}
it is easy to see that the collection of morphisms $\Phi_{\alpha'\alpha}$, $\alpha\in\Lambda, \alpha'\in\Lambda'$ with \eqref{morphism_cocycle} are equivalent to $\Phi$ in the sense that $\Phi(\nset)$ can be fully recovered from it. Such a collection of morphisms is called a \textbf{morphism of $n$-fold vector bundle cocycles}.
\begin{definition}
Let $C_M:=\left((U_\alpha)_{\alpha\in\Lambda}, (\phi^{\alpha\beta})_{\alpha,\beta\in\Lambda}\right)$ be an $n$-fold vector bundle cocycle on a smooth manifold $M$ and let 
$C_N:=\left((V_{\alpha'})_{\alpha'\in\Lambda'}, (\psi^{\alpha'\beta'})_{\alpha',\beta'\in\Lambda'}\right)$ be an $n$-fold vector bundle cocycle on a smooth manifold $N$.

Then a morphism $\Phi\colon C_M\to C_N$ of $n$-fold vector bundle cocycles over a smooth map $\Phi_0$ is a collection 
\[ \left(\left( \Phi_{\alpha'\alpha}^{(i_1,\ldots, i_l)}\in C^\infty(\Phi_0^{-1}(V_\alpha')\cap U_\alpha, \Hom( V^{i_1}\odot\ldots\odot V^{i_l}, W^{i_1+\ldots+i_l}))\right)_{(i_1,\ldots, i_l)\in \mathcal P(n)}, \alpha\in\Lambda, \alpha'\in\Lambda'\right) 
\]
of symmetric $n$-fold vector bundle morphisms such that \eqref{morphism_cocycle} holds.
\end{definition}
It is also easy to see that such a morphism of symmetric $n$-fold vector bundle cocycles defines a morphism $\E(C_M)\to \E(C_N)$ of symmetric $n$-fold vector bundles, where $\E(C_M)$ is the symmetric $n$-fold vector bundle defined as in the proof of Theorem \ref{symmetric_atlas_theorem}.

The same discussion for $[n]$-manifold gives the notion of morphism of $[n]$-manifold cocycles, which are again exactly the same information as in the last definition. Hence a morphism $\Phi\colon \E\to \F$ of symmetric $n$-fold vector bundles defines a morphism of symmetric $n$-fold vector bundle cocycles, hence a morphism of $[n]$-manifold cocycles, and hence a morphism $\mathcal A(\Phi)\colon \mathcal A(\E)\to\mathcal A(\F)$.
This completes the definition of the functor 
\[  \mathcal A\colon  \mathsf{SnVB} \to \mathsf{[n]Man}.
\]
The \emph{geometrisation functor} 
\[ \mathcal G\colon   \mathsf{[n]Man}\to \mathsf{SnVB}
\]
is defined in exactly the same manner: starting from an $[n]$-manifold $\mathcal M$, consider an associated maximal $[n]$-manifold cocycle, which is a symmetric $n$-fold vector bundle cocycle, and hence defines a symmetric $n$-fold vector bundle $\E(\mathcal M)=:\mathcal G(\mathcal M)$. As above, this extends naturally to morphisms of $[n]$-manifolds.

The two obtained functors define an equivalence between the two categories 
$ \mathsf{SnVB}$ and $\mathsf{[n]Man}$: The composition 
\[  \mathcal G\circ \mathcal A\colon  \mathsf{SnVB} \to \mathsf{SnVB}
\]
sends a symmetric $n$-fold vector bundle $\E$ to the abstract symmetric $n$-fold vector bundle defined by a choice of maximal symmetric $n$-fold vector bundle cocycle 
associated to $\E$. The obtained symmetric $n$-fold vector bundle is canonically isomorphic to $\E$.
In the same manner, 
\[  \mathcal A\circ \mathcal G\colon  \mathsf{[n]Man} \to \mathsf{[n]Man}
\]
sends an $[n]$-manifold $\mathcal M$ to the abstract $[n]$-manifold defined by a choice of maximal $[n]$-manifold cocycle for $\mathcal M$. The two $[n]$-manifolds are canonically isomorphic. This completes the proof of the following theorem.
\begin{theorem}\label{main_main} Let $n\in\mathbb N$. 
The algebraisation and geometrisation functors \[  \mathcal A\colon  \mathsf{SnVB} \to \mathsf{[n]Man}.
\]
and 
\[ \mathcal G\colon   \mathsf{[n]Man}\to \mathsf{SnVB}
\]
establish an equivalence between the category of symmetric $n$-fold vector bundles and the category of $[n]$-manifolds.
\end{theorem}

Note that this equivalence can be extended to an equivalence of $\N\mathsf{Man}$ with the category of symmetric multiple vector bundles, but the details will be carried out in a future work.

\appendix

   \section{From linear splittings to decompositions}\label{dec_splittings}
   
   This section revisits in detail the proof of the following statement in \cite{HeJo20} (see Theorem 3.3 there), because its symmetric analogue requires further insights on it -- in particular the uniqueness part of this statement was not explained in detail in \cite{HeJo20}. In this section a general $n$-fold vector bundle is considered -- no $S_n$-symmetry is assumed.

\begin{proposition}\label{thm_split_dec_n_general}
Let $\E$ be an $n$-fold vector bundle.
Given a linear splitting $\Sigma$ of $\E$ and 
	compatible decompositions $\S^J\colon (E^{\rm dec})^{\rho_J}\to E^{\rho_J}$ of the highest order cores of $\E$, 
    for $J\subseteq \nset$ with $\# J=2$, there 
	exists a unique decomposition $\mathcal S$ of $\E$ such that 
	$\Sigma=\mathcal S\circ\iota$ and such that the core morphisms 
	of $\S$ equal $\S^J$ for all $J$. 
\end{proposition}

The decomposition $\mathcal S\colon \E^{\rm dec}\to \E$ is recursively constructed as follows. 
Write $J_1, \ldots,J_{\binom{n}{2}}$ for the subsets of $\nset$ 
with $\# J_k=2$ and define an increasing chain of $\binom{n}{2}$ decomposed $n$-fold 
vector bundles: For $k=0,\ldots,\binom{n}{2}$  set $\A^k=(B^k_I)_{I\subseteq\nset}$ 
with \[B^k_I=\left\{\begin{array}{cl}
A_{\# I} &\quad \text{ if }\# I=1 \text{ or if there is 
$i\leq k$ such that } J_i\subseteq I;\\
M &\quad \text{ otherwise.}
\end{array}\right.
\]
Set
$\E^k:=\E^{\A^k}$.
There 
are obvious monomorphisms
\[\overline{\E}=\E^0\hookrightarrow \E^1\hookrightarrow\ldots
\hookrightarrow \E^{\binom{n}{2}}=\E^{\rm dec}\]
of $n$-fold vector bundles. In particular the spaces  $E^k:=\E^k(\nset)$ can be seen  
as submanifolds of $\E^{\rm dec}(\nset)$. Note that additionally 
$(\E^{\rm dec})^{\nset}_{J_i}(\nset)\subseteq \E^k(\nset)$ for all $i\leq k$. 

First set $\mathcal S^0:=\Sigma$.
Then take $k\geq 0$ and assume that a monomorphism $\mathcal S^k\colon\E^k\to \E$
 of $n$-fold vector bundles has already been constructed, that
restricts to $\Sigma$ on $\E^0$ and to $\mathcal S^{J_i}$ on
$E^k\cap(E^{\rm dec})^\nset_{J_i}$ for $i=1,\ldots,k$.
Take $\mathbf{x}=(x_I)_{I\subseteq\nset}\in \E^{k+1}(\nset)$ over a base point $m\in M$. Then in 
particular $x_I=0_m^{A_{\# I}}$ if $\#I\geq 2$ and there is no $i\leq k+1$ 
with $J_i\subseteq I$. Set $\mathbf{y}:=(y_I)_{I\subseteq\nset}\in \E^k(\nset)\subseteq \E^{\rm dec}(\nset)$ with 
\begin{equation}\label{y}
y_I=\left\{\begin{array}{cl}
x_I &\text{ if either $\# I=1$ or there is $i\leq k$ such that 
$J_i\subseteq I$},\\
 0^{A_{\# I}}_m& \text{  otherwise.} \end{array}\right.
\end{equation}
Set furthermore 
$\mathbf{z}:=(z_I)_{I\subseteq\nset}\in (\E^{\rm dec})^\nset_{J_{k+1}}$ where 
\begin{equation}\label{z}
z_I=\left\{\begin{array}{cl}y_I &\text{  whenever }
I\subseteq\nset\setminus J_{k+1}, \\
x_I& \text{ whenever $J_{k+1}\subseteq I$ 
and there is no $i\leq k$ with $J_i\subseteq I$,}\\
 0^{A_{\# I}}_m& \text{ otherwise. }
 \end{array}\right.
\end{equation}
 Then writing $J_{k+1}=\{s,t\}$, it is easy to check that 
\[
\mathbf{x}=\mathbf{y}\dvplus{}{\nset\setminus\{s\}}\left(
\nvbzero{\nset}{p_s(\mathbf{y})}\dvplus{}{\nset\setminus\{t\}}
\mathbf{z}\right)=\mathbf{y}\dvplus{}{\nset\setminus\{t\}}\left(
\nvbzero{\nset}{p_t(\mathbf{y})}\dvplus{}{\nset\setminus\{s\}}
\mathbf{z}\right)\,.
\] 
The last equality follows directly from the interchange law in the double vector bundle 
$(E;E_{\nset\setminus\{s\}},E_{\nset\setminus\{t\}};E_{\nset\setminus\{s,t\}})$
since $\mathbf{z}$ is in the core of this double vector bundle. 
The monomorphism $\mathcal S^{k+1}\colon \mathbb E^{k+1}\to \mathbb E$ is then defined by 
\begin{align*}
\mathcal S^{k+1}(\mathbf{x})&:=\mathcal S^k(\mathbf{y})\dvplus{}{\nset\setminus\{s\}}\left(
\nvbzero{\nset}{p_s\bigl(\mathcal S^k(\mathbf{y})\bigr)}\dvplus{}{\nset\setminus\{t\}}
\mathcal S^{J_{k+1}}(\mathbf{z})\right)\\
&= \mathcal S^k(\mathbf{y})\dvplus{}{\nset\setminus\{t\}}\left(
\nvbzero{\nset}{p_t\bigl(\mathcal S^k(\mathbf{y})\bigr)}\dvplus{}{\nset\setminus\{s\}}
\mathcal S^{J_{k+1}}(\mathbf{z})\right)\,.
\end{align*}
Repeating this construction $\binom{n}{2}$-times yields the decomposition $\mathcal S=\mathcal S^{\binom{n}{2}}\colon\E^{\rm dec}\to \E$ of $E$,
see \cite{HeJo20}. 
\medskip

The uniqueness of $\mathcal S$ in the statement needs to be checked next.
It is sufficient to show that $\mathcal S$  does not depend on the choice of numbering $J_1, \ldots, J_{\binom{n}{2}}$ of the subsets of $\nset$ with two elements. Since this fact is crucial for the proof below of the $S_n$-invariance of $\mathcal S$, its proof is done in detail here.
It suffices to show that for each $k=0,\ldots, \binom{n}{2}-2$, the monomorphism $\mathcal S^{k+2}$ does not depend on the order of the two subsets $J_{k+1}$ and $J_{k+2}$. Choose such a $k$ and set $J_{k+1}=\{s,t\}$ and $J_{k+2}=\{p,q\}$. Take $\mathbf{x}=(x_I)_{I\subseteq\nset}\in E^{k+2}$ over a base point $m\in M$. That is, $x_I=0_m^{A_{\# I}}$ if $\#I\geq 2$ and there is no $i\leq k+2$ 
with $J_i\subseteq I$. Set $\mathbf{y}:=(y_I)_{I\subseteq\nset}\in E^{k}\subseteq \E^{\rm dec}(\nset)$ with 
\[y_I=\left\{\begin{array}{cl}
x_I &\text{ if either $\# I=1$ or there is $1\leq i\leq k$ such that 
$J_i\subseteq I$},\\
 0^{A_{\# I}}_m& \text{  otherwise.} \end{array}\right.
\]
Define further $\mathbf{a}\in (\E^{\rm dec})^{J_{k+1}}$ and $\mathbf{b}\in (\E^{\rm dec})^{J_{k+2}}$ by 
\[a_I=\left\{\begin{array}{cl}x_I &\text{  if }
I\subseteq\nset\setminus J_{k+1} \text{ and } (\#I=1 \text{ or there exists } i\in\{1\,\ldots, k\}: J_i\subseteq I), \\
x_I& \text{ if  $J_{k+1}\subseteq I$ 
and there is no $i\in\{1,\ldots,k\}$ with $J_i\subseteq I$,}\\
0^{A_{\# I}}_m& \text{ otherwise, }
 \end{array}\right.
\]
and 
\[b_I=\left\{\begin{array}{cl}x_I &\text{  if }
I\subseteq\nset\setminus J_{k+2} \text{ and } (\#I=1 \text{ or there exists } i\in\{1\,\ldots, k,k+1\}: J_i\subseteq I), \\
x_I& \text{ if  $J_{k+2}\subseteq I$ 
and there is no $i\in\{1,\ldots,k,k+1\}$ with $J_i\subseteq I$,}\\
0^{A_{\# I}}_m& \text{ otherwise. }
 \end{array}\right.
\]
Then 
\[ \mathbf{u}:=\mathbf y\dvplus{}{\nset\setminus\{s\}}\left(\nvbzero{\nset}{p^{\nset}_{\nset\setminus\{s\}}(\mathbf{y})}\dvplus{}{\nset\setminus\{t\}}
\mathbf{a}\right)
\]
lies in $E^{k+1}$ and is described by
\[u_I=\left\{\begin{array}{cl}
x_I &\text{ if either $\# I=1$ or there is $1\leq i\leq k+1$ such that 
$J_i\subseteq I$},\\
 0^{A_{\# I}}_m& \text{  otherwise.} \end{array}\right.\]
The sum
\begin{equation}\label{first_variant_sum}
\mathbf u\dvplus{}{\nset\setminus\{p\}}\left(\nvbzero{\nset}{p^{\nset}_{\nset\setminus\{p\}}(\mathbf{u})}\dvplus{}{\nset\setminus\{q\}}
\mathbf{b}\right)=\left(\mathbf y\dvplus{}{\nset\setminus\{s\}}\left(\nvbzero{\nset}{p^{\nset}_{\nset\setminus\{s\}}(\mathbf{y})}\dvplus{}{\nset\setminus\{t\}}
\mathbf{a}\right)\right)
\dvplus{}{\nset\setminus\{p\}}\left(\nvbzero{\nset}{p^{\nset}_{\nset\setminus\{p\}}(\mathbf{u})}\dvplus{}{\nset\setminus\{q\}}
\mathbf{b}\right)
\end{equation}
then equals $\mathbf x$.

Now inverse the roles of $J_{k+1}$ and $J_{k+2}$.
Define  $\mathbf{c}\in (\E^{\rm dec})^{J_{k+2}}$ and $\mathbf{d}\in (\E^{\rm dec})^{J_{k+1}}$ by 
\[c_I=\left\{\begin{array}{cl}x_I &\text{  if }
I\subseteq\nset\setminus J_{k+2} \text{ and } (\#I=1 \text{ or there exists } i\in\{1\,\ldots, k\}: J_i\subseteq I), \\
x_I& \text{ if  $J_{k+2}\subseteq I$ 
and there is no $i\in\{1,\ldots,k\}$ with $J_i\subseteq I$,}\\
0^{A_{\# I}}_m& \text{ otherwise, }
 \end{array}\right.
\]
and 
\[d_I=\left\{\begin{array}{cl}x_I &\text{  if }
I\subseteq\nset\setminus J_{k+1} \text{ and } (\#I=1 \text{ or there exists } i\in\{1\,\ldots, k,k+2\}: J_i\subseteq I), \\
x_I& \text{ if  $J_{k+1}\subseteq I$ 
and there is no $i\in\{1,\ldots,k,k+2\}$ with $J_i\subseteq I$,}\\
0^{A_{\# I}}_m& \text{ otherwise. }
 \end{array}\right.
\]
Then 
\[ \mathbf{v}:=\mathbf y\dvplus{}{\nset\setminus\{p\}}\left(\nvbzero{\nset}{p^{\nset}_{\nset\setminus\{p\}}(\mathbf{y})}\dvplus{}{\nset\setminus\{q\}}
\mathbf{c}\right)
\]
is described by
\[v_I=\left\{\begin{array}{cl}
x_I &\text{ if either $\# I=1$ or there is $i\in\{1,\ldots,k, k+2\}$ such that 
$J_i\subseteq I$},\\
 0^{A_{\# I}}_m& \text{  otherwise.} \end{array}\right.\]
The sum
\begin{equation}\label{second_variant_sum}\mathbf v\dvplus{}{\nset\setminus\{s\}}\left(\nvbzero{\nset}{p^{\nset}_{\nset\setminus\{s\}}(\mathbf{v})}\dvplus{}{\nset\setminus\{t\}}
\mathbf{d}\right)=\left(\mathbf y\dvplus{}{\nset\setminus\{p\}}\left(\nvbzero{\nset}{p^{\nset}_{\nset\setminus\{p\}}(\mathbf{y})}\dvplus{}{\nset\setminus\{q\}}
\mathbf{c}\right)\right)
\dvplus{}{\nset\setminus\{s\}}\left(\nvbzero{\nset}{p^{\nset}_{\nset\setminus\{s\}}(\mathbf{v})}\dvplus{}{\nset\setminus\{t\}}
\mathbf{d}\right)
\end{equation}
then equals $\mathbf x$.
It is easy to check that the tuples $\mathbf a$, $\mathbf b$, $\mathbf c$ and $\mathbf d$ satisfy
\[ p^{\nset}_{\nset\setminus\{s\}}(\mathbf b)=p^{\nset}_{\nset\setminus\{s\}}(\mathbf c), \quad p^{\nset}_{\nset\setminus\{t\}}(\mathbf b)=p^{\nset}_{\nset\setminus\{t\}}(\mathbf c),
\]
as well as 
\[ p^{\nset}_{\nset\setminus\{p\}}(\mathbf a)=p^{\nset}_{\nset\setminus\{p\}}(\mathbf d), \quad \text{ and } \quad p^{\nset}_{\nset\setminus\{q\}}(\mathbf a)=p^{\nset}_{\nset\setminus\{q\}}(\mathbf d).
\]
It turns then out that 
\[ \mathbf b=\mathbf c\dvplus{}{\nset\setminus\{s\}}\left(\nvbzero{\nset}{p^{\nset}_{\nset\setminus\{s\}}(\mathbf{c})}\dvplus{}{\nset\setminus\{t\}}
\mathbf{e}
\right)=\mathbf c\dvplus{}{\nset\setminus\{t\}}\left(\nvbzero{\nset}{p^{\nset}_{\nset\setminus\{t\}}(\mathbf{c})}\dvplus{}{\nset\setminus\{s\}}
\mathbf{e}
\right)
\]
and also
\[ \mathbf d=\mathbf a\dvplus{}{\nset\setminus\{p\}}\left(\nvbzero{\nset}{p^{\nset}_{\nset\setminus\{p\}}(\mathbf{a})}\dvplus{}{\nset\setminus\{q\}}
\mathbf{e}
\right)=\mathbf a\dvplus{}{\nset\setminus\{q\}}\left(\nvbzero{\nset}{p^{\nset}_{\nset\setminus\{q\}}(\mathbf{a})}\dvplus{}{\nset\setminus\{p\}}
\mathbf{e}
\right)
\]
with $\mathbf{e}\in (\E^{\rm dec})^{\rho_{J_{k+1}}\sqcap\rho_{J_{k+2}}}$ defined by 
\[e_I=\left\{\begin{array}{cl}
x_I &\text{ if } I\subseteq \nset\setminus(J_{k+1}\cup J_{k+2}) \text{ and } (\#I=1 \text{ or there exists } i\in\{1,\ldots, k\}: J_i\subseteq I)\\
 x_I &\text{ if } J_{k+2}\subseteq I\subseteq \nset\setminus J_{k+1} \text{ and there is no }  i\in\{1,\ldots, k\} \text{ with } J_i\subseteq I\\
  x_I &\text{ if } J_{k+1}\subseteq I\subseteq \nset\setminus J_{k+2} \text{ and there is no }  i\in\{1,\ldots, k\} \text{ with } J_i\subseteq I\\
   x_I&\text{ if } J_{k+1}\cup J_{k+2}\subseteq I  \text{ and there is no }  i\in\{1,\ldots, k\} \text{ with } J_i\subseteq I\\
   0^{A_{\# I}}_m& \text{ otherwise. }
  \end{array}\right.\]
  
  \medskip
  
  Consider the second term of \eqref{first_variant_sum}. Compute
  \begin{equation*}
  \begin{split}
  \nvbzero{\nset}{p^{\nset}_{\nset\setminus\{p\}}(\mathbf{u})}&=\nvbzero{\nset}{p^{\nset}_{\nset\setminus\{p\}}\left(\mathbf y\dvplus{}{\nset\setminus\{s\}}\left(\nvbzero{\nset}{p^{\nset}_{\nset\setminus\{s\}}(\mathbf{y})}\dvplus{}{\nset\setminus\{t\}}
\mathbf{a}\right)\right)}\\
&=\nvbzero{\nset}{p^{\nset}_{\nset\setminus\{p\}}(\mathbf{y})\dvplus{}{\nset\setminus\{s,p\}}
\left(\nvbzero{\nset}{p^{\nset}_{\nset\setminus\{s,p\}}(\mathbf{y})}\dvplus{}{\nset\setminus\{t,p\}}p^{\nset}_{\nset\setminus\{p\}}(a)
\right)
}\\
&=\nvbzero{\nset}{p^{\nset}_{\nset\setminus\{p\}}(\mathbf{y})}\dvplus{}{\nset\setminus\{s\}}\left(
\nvbzero{\nset}{p^{\nset}_{\nset\setminus\{s,p\}}(\mathbf{y})}
\dvplus{}{\nset\setminus\{t\}}
\nvbzero{\nset}{p^{\nset}_{\nset\setminus\{p\}}(\mathbf{a})}
\right)
  \end{split}
  \end{equation*}
  and use $\mathbf b=\mathbf c\dvplus{}{\nset\setminus\{s\}}\left(\nvbzero{\nset}{p^{\nset}_{\nset\setminus\{s\}}(\mathbf{c})}\dvplus{}{\nset\setminus\{t\}}
\mathbf{e}
\right)$ as well as the interchange formula in the double vector bundle 
$\left(\E(\nset), \E(\nset\setminus\{s\}), \E(\nset\setminus\{q\}), \E(\nset\setminus\{s,q\})\right)$ in order to get
\begin{equation}\label{first_intermediate_J_k+1J_k+2}
\begin{split}
 \nvbzero{\nset}{p^{\nset}_{\nset\setminus\{p\}}(\mathbf{u})}\dvplus{}{\nset\setminus\{q\}}
\mathbf{b}
&=\left[\nvbzero{\nset}{p^{\nset}_{\nset\setminus\{p\}}(\mathbf{y})}\dvplus{}{\nset\setminus\{s\}}\left(
\nvbzero{\nset}{p^{\nset}_{\nset\setminus\{s,p\}}(\mathbf{y})}
\dvplus{}{\nset\setminus\{t\}}
\nvbzero{\nset}{p^{\nset}_{\nset\setminus\{p\}}(\mathbf{a})}
\right)\right]\dvplus{}{\nset\setminus\{q\}}\left[
\mathbf c\dvplus{}{\nset\setminus\{s\}}\left(\nvbzero{\nset}{p^{\nset}_{\nset\setminus\{s\}}(\mathbf{c})}\dvplus{}{\nset\setminus\{t\}}
\mathbf{e}\right)
\right]\\
&=\left[\nvbzero{\nset}{p^{\nset}_{\nset\setminus\{p\}}(\mathbf{y})}\dvplus{}{\nset\setminus\{q\}} \mathbf c
\right]\dvplus{}{\nset\setminus\{s\}} \left[\left(
\nvbzero{\nset}{p^{\nset}_{\nset\setminus\{s,p\}}(\mathbf{y})}
\dvplus{}{\nset\setminus\{t\}}
\nvbzero{\nset}{p^{\nset}_{\nset\setminus\{p\}}(\mathbf{a})}
\right)\dvplus{}{\nset\setminus\{q\}}\left(\nvbzero{\nset}{p^{\nset}_{\nset\setminus\{s\}}(\mathbf{c})}\dvplus{}{\nset\setminus\{t\}}
\mathbf{e}\right)
\right].
\end{split}
\end{equation}
With the interchange law in the double vector bundle $\left(\E(\nset), \E(\nset\setminus\{t\}), \E(\nset\setminus\{q\}), \E(\nset\setminus\{t,q\})\right)$, the second term on the right-hand face  becomes
\begin{equation*}
\begin{split}
 &\left(
\nvbzero{\nset}{p^{\nset}_{\nset\setminus\{s,p\}}(\mathbf{y})}
\dvplus{}{\nset\setminus\{t\}}
\nvbzero{\nset}{p^{\nset}_{\nset\setminus\{p\}}(\mathbf{d})}
\right)\dvplus{}{\nset\setminus\{q\}}\left(\nvbzero{\nset}{p^{\nset}_{\nset\setminus\{s\}}(\mathbf{c})}\dvplus{}{\nset\setminus\{t\}}
\mathbf{e}\right)\\
&=
 \left(
\nvbzero{\nset}{p^{\nset}_{\nset\setminus\{s,p\}}(\mathbf{y})}
\dvplus{}{\nset\setminus\{q\}}\nvbzero{\nset}{p^{\nset}_{\nset\setminus\{s\}}(\mathbf{c})}\right)
\dvplus{}{\nset\setminus\{t\}}
\left(\nvbzero{\nset}{p^{\nset}_{\nset\setminus\{p\}}(\mathbf{d})}\dvplus{}{\nset\setminus\{q\}}
\mathbf{e}\right).
\end{split}
\end{equation*}
Then in \eqref{first_variant_sum}, this yields against with the interchange law,
\begin{equation*}
\begin{split}
\mathbf x&=\left(\mathbf y\dvplus{}{\nset\setminus\{s\}}\left(\nvbzero{\nset}{p^{\nset}_{\nset\setminus\{s\}}(\mathbf{y})}\dvplus{}{\nset\setminus\{t\}}
\mathbf{a}\right)\right)
\dvplus{}{\nset\setminus\{p\}}\left(\nvbzero{\nset}{p^{\nset}_{\nset\setminus\{p\}}(\mathbf{u})}\dvplus{}{\nset\setminus\{q\}}
\mathbf{b}\right)\\
&=\left[\mathbf y\dvplus{}{\nset\setminus\{s\}}\left(\nvbzero{\nset}{p^{\nset}_{\nset\setminus\{s\}}(\mathbf{y})}\dvplus{}{\nset\setminus\{t\}}
\mathbf{a}\right)\right]\\
&\quad \dvplus{}{\nset\setminus\{p\}}\left[\left(\nvbzero{\nset}{p^{\nset}_{\nset\setminus\{p\}}(\mathbf{y})}\dvplus{}{\nset\setminus\{q\}} \mathbf c
\right)\dvplus{}{\nset\setminus\{s\}}\left( \left(
\nvbzero{\nset}{p^{\nset}_{\nset\setminus\{s,p\}}(\mathbf{y})}
\dvplus{}{\nset\setminus\{q\}}\nvbzero{\nset}{p^{\nset}_{\nset\setminus\{s\}}(\mathbf{c})}\right)
\dvplus{}{\nset\setminus\{t\}}
\left(\nvbzero{\nset}{p^{\nset}_{\nset\setminus\{p\}}(\mathbf{d})}\dvplus{}{\nset\setminus\{q\}}
\mathbf{e}\right)\right)\right]\\
&=\left[\mathbf y\dvplus{}{\nset\setminus\{p\}}\left(\nvbzero{\nset}{p^{\nset}_{\nset\setminus\{p\}}(\mathbf{y})}\dvplus{}{\nset\setminus\{q\}} \mathbf c
\right)
\right]\\
&\quad \dvplus{}{\nset\setminus\{s\}}\left[
\left(\nvbzero{\nset}{p^{\nset}_{\nset\setminus\{s\}}(\mathbf{y})}\dvplus{}{\nset\setminus\{t\}}
\mathbf{a}\right)
\dvplus{}{\nset\setminus\{p\}}\left( \left(
\nvbzero{\nset}{p^{\nset}_{\nset\setminus\{s,p\}}(\mathbf{y})}
\dvplus{}{\nset\setminus\{q\}}\nvbzero{\nset}{p^{\nset}_{\nset\setminus\{s\}}(\mathbf{c})}\right)
\dvplus{}{\nset\setminus\{t\}}
\left(\nvbzero{\nset}{p^{\nset}_{\nset\setminus\{p\}}(\mathbf{a})}\dvplus{}{\nset\setminus\{q\}}
\mathbf{e}\right)\right)\right].
\end{split}
\end{equation*}
Comparing the second term with \eqref{first_intermediate_J_k+1J_k+2}, by exchanging the roles of $\mathbf a$ and  $\mathbf c$, and of $\mathbf v$ with $\mathbf u$ yields  then 
\[ \mathbf x= \left[\mathbf y\dvplus{}{\nset\setminus\{p\}}\left(\nvbzero{\nset}{p^{\nset}_{\nset\setminus\{p\}}(\mathbf{y})}\dvplus{}{\nset\setminus\{q\}} \mathbf c
\right)
\right] \dvplus{}{\nset\setminus\{s\}}\left[
\nvbzero{\nset}{p^{\nset}_{\nset\setminus\{p\}}(\mathbf{v})}\dvplus{}{\nset\setminus\{q\}}
\mathbf{d}
\right]
\]
which is \eqref{second_variant_sum}.
The computation above shows that \eqref{first_variant_sum} and \eqref{second_variant_sum} can be obtained from each other via repeated uses of the interchange formula. Applying $\mathcal S^{k+2}$ to \eqref{first_variant_sum}
is then computing  
\begin{equation*}
\begin{split}
\mathcal S^{k+2}(\mathbf x)=&\mathcal S^{k+1}
\left(\mathbf y\dvplus{}{\nset\setminus\{s\}}\left(\nvbzero{\nset}{p^{\nset}_{\nset\setminus\{s\}}(\mathbf{y})}\dvplus{}{\nset\setminus\{t\}}
\mathbf{a}\right)\right)
\dvplus{}{\nset\setminus\{p\}}\left(\nvbzero{\nset}{p^{\nset}_{\nset\setminus\{p\}}(\mathcal S^{k+1}(\mathbf{u}))}\dvplus{}{\nset\setminus\{q\}}
\mathcal S^{J_{k+2}}(\mathbf{b})\right)\\
=&
\left(\mathcal S^k(\mathbf y)\dvplus{}{\nset\setminus\{s\}}\left(\nvbzero{\nset}{p^{\nset}_{\nset\setminus\{s\}}(\mathcal S^k(\mathbf{y}))}\dvplus{}{\nset\setminus\{t\}}
\mathcal S^{J_{k+1}}(\mathbf{a})\right)\right)
\dvplus{}{\nset\setminus\{p\}}\left(\nvbzero{\nset}{p^{\nset}_{\nset\setminus\{p\}}(\mathcal S^{k+1}(\mathbf{u}))}\dvplus{}{\nset\setminus\{q\}}
\mathcal S^{J_{k+2}}(\mathbf{b})\right)
\end{split}
\end{equation*}
and further transforming this as above using the interchange formulas shows that
this equals 
\begin{equation*}
\begin{split}
&\left(\mathcal S^k(\mathbf y)\dvplus{}{\nset\setminus\{p\}}\left(\nvbzero{\nset}{p^{\nset}_{\nset\setminus\{p\}}(\mathcal S^k(\mathbf{y}))}\dvplus{}{\nset\setminus\{q\}} \mathcal S^{J_{k+2}}(\mathbf c)
\right)
\right) \dvplus{}{\nset\setminus\{s\}}\\
&\quad \left(
\nvbzero{\nset}{p^{\nset}_{\nset\setminus\{p\}}\left(\mathcal S^k(\mathbf y)\dvplus{}{\nset\setminus\{p\}}\left(\nvbzero{\nset}{p^{\nset}_{\nset\setminus\{p\}}(\mathcal S^k(\mathbf{y}))}\dvplus{}{\nset\setminus\{q\}} \mathcal S^{J_{k+2}}(\mathbf c)
\right)\right)}
\dvplus{}{\nset\setminus\{q\}}
\mathcal S^{J_{k+1}}(\mathbf{d})
\right)
\end{split}
\end{equation*}
if and only if 
\[\mathcal S^{J_{k+1}}(\mathbf e)=\mathcal S^{J_{k+2}}(\mathbf e).
\]
But since $\mathbf e$ lies in $(\E^{\rm dec})^{\rho_{J_{k+1}}\sqcap\rho_{J_{k+2}}}$ and $\mathcal S^{J_{k+1}}$ and 
$\mathcal S^{J_{k+2}}$ are compatible as in \eqref{compatible_dec}, this equality is immediate.

%%%%%%%%%%%%%%%%%%%% APPENDIX B   
  
\section{On the existence of symmetric linear splittings and symmetric decompositions}\label{proof_dec_sym_nvb}

This section proves that
any symmetric $n$-fold vector bundle admits a (non-canonical) symmetric linear
decomposition. The existence of a linear decomposition is proved in \cite{HeJo20}, see also Section \ref{fix_of_3.6}
so it remains to show that a \emph{symmetric} decomposition can always be chosen.

Recall that a  linear splitting $\Sigma$ of an $n$-fold vector bundle $\E$ and
decompositions $\mathcal S^{\rho}$ of the $(n-1)$-cores of $\E$ are called \textbf{compatible}
if \eqref{compatible_dec} and \eqref{comp_sigma_S} hold.
If $\E$ is now a symmetric $n$-fold vector bundle with symmetric structure 
$\Phi$,  the decompositions of the $(n-1)$-cores of $\E$ as above are furthermore 
\textbf{symmetrically compatible}, if additionally
\begin{equation}\label{sym_comp_cond}
\left(\Phi_\sigma\right)^{\rho}\circ \S^{\rho}
=\S^{\sigma(\rho)}\circ\left(\Psi^{\rm dec}_\sigma\right)^{\rho}
\end{equation}
for all $\sigma\in S_n$ and $I\subseteq\nset$ with $\# I=2$. For simplicity, the restriction of $\Phi_\sigma$, $\Psi_\sigma^{\rm dec}$ to iterated highest order cores of $\E$ and $\E^{\rm dec}$ are just written $\Phi_\sigma$ and respectively $\Psi_\sigma^{\rm dec}$ for $\sigma\in S_n$. That is,
\eqref{sym_comp_cond} is just written 
\[\Phi_\sigma\circ \S^{\rho}
=\S^{\sigma(\rho)}\circ\Psi^{\rm dec}_\sigma.
\]

\begin{proposition}\label{thm_split_dec_n}
\begin{enumerate}
\item[(a)] Let $\mathcal S$ be a symmetric decomposition of an $n$-fold vector 
bundle $\E\colon\square^n\to \Man$. Then the composition 
$\Sigma=\mathcal S\circ\iota\colon\overline{\E}\to\E$, with 
$\iota$ defined as in (3) of Definition \ref{decomposed_nfvb}, is a symmetric splitting of $\E$. 
Furthermore, the core morphisms $\mathcal S^{\rho_J}\colon(\E^{\rm dec})^{\rho_{J}}\to\E^{\rho_J}$
are decompositions of $\E^{\rho_J}$ for all $J\subseteq\nset$ with $\# J=2$ 
and 
these decompositions and the linear splitting are symmetrically compatible as in \eqref{sym_comp_cond}.

\item[(b)] Conversely, given a symmetric linear splitting $\Sigma$ of $\E$ and compatible, 
	symmetrically compatible decompositions $\mathcal S^{J}$ of the highest order cores 
	$\E^{\rho_J}$
    for $J\subseteq \nset$ with $\# J=2$, there 
	exists a unique symmetric decomposition $\mathcal S$ of $\E$ such that 
	$\Sigma=\mathcal S\circ\iota$ and such that the induced core morphisms $\mathcal S^{\rho_J}$
	of $\mathcal S$ equal $\mathcal S^{J}$ 
	for all $J$. 
\end{enumerate}
\end{proposition}

\begin{proof}\begin{enumerate}
\item[(a)]
  Consider a symmetric decomposition $\mathcal S\colon \E^{\rm dec}\to\E$.  
 Then the composition $\Sigma=\mathcal S\circ\iota$ is a linear splitting of $\mathbb E$ and for all $J\subseteq \nset$ with $\# J=2$ the restrictions $\mathcal S^{\rho_J}$ are decompositions of $\mathbb E^{\rho_J}$ that are compatible with $\Sigma$, see Proposition 3.3 in \cite{HeJo20}.
 
Since $\iota$ is $S_n$-invariant, $\Sigma$ is clearly symmetric:
\begin{equation*}
\begin{split}
\Sigma(I)\circ \overline \Psi_\sigma(I)&=\mathcal S(I)\circ \iota(I)\circ  \overline\Psi_\sigma(I)=\mathcal S(I)\circ  \Psi_\sigma^{\rm dec}(I)\circ \iota(I) \\
&=\Phi_\sigma(I)\circ\mathcal S(I)\circ \iota(I)=\Phi_\sigma(I)\circ\Sigma(I)
\end{split}
\end{equation*}
for $\sigma\in S_n$ and $I\subseteq \nset$. It remains to show that the core morphisms are symmetrically compatible.
But this follows immediately by the definition of these objects and by the $S_n$-invariance of $\mathcal S$ and $\iota$:
\begin{equation*}
\begin{split}
\Phi_\sigma\circ \S^{\rho_J}&=\left.
\left(\Phi_\sigma\circ \mathcal S\right)\right\an{(\E^{\rm dec})^{\rho_J}}=\left.\left(\S\circ\Psi_\sigma^{\rm dec}\right)\right\an{(\E^{\rm dec})^{\rho_J}}
\\
&=\S^{\rho_{\sigma(J)}}\circ\left(\Psi^{\rm dec}_\sigma\right)^{J}
\end{split}
\end{equation*}
for all $J\subseteq \nset$ with $\# J=2$.

\item[(b)] Conversely, assume that an $S_n$-invariant splitting $\Sigma$ of $\E$ 
and compatible, symmetrically compatible decompositions $\mathcal S^J$ of the cores $\E^{\rho_J}$ 
with $J\subseteq \nset$, $\# J=2$ are given as in (b) and \eqref{sym_comp_cond}.
Proposition 3.3 of \cite{HeJo20} (see Proposition \ref{thm_split_dec_n_general} in Section \ref{dec_splittings} above) shows the existence of a unique decomposition $\mathcal S$ of $\E$ that restricts in the sense 
of (b) to $\Sigma$ and the core decompositions $\mathcal S^J$. It remains to check that this decomposition is symmetric.

Going back to the proof of Proposition \ref{thm_split_dec_n_general} in Section \ref{dec_splittings}, fix once and for all the numbering of the sets $J_1, \ldots, J_{\binom{n}{2}}$. For each $\sigma\in S_n$, the list of subsets 
\[ J^\sigma_1:=\sigma(J_1), \ldots, J^\sigma_{\binom{n}{2}}:=\sigma\left(J_{\binom{n}{2}}\right)
\]
is again a numbering of the subsets of $\nset$ with $2$ elements. Repeating the construction in Section \ref{dec_splittings} for this new choice yields 
again for each $k=1,\ldots, \binom{n}{2}$ a family $\mathcal A^k_\sigma=(B^{k,\sigma}_I)_{I\subseteq \nset}$,  
\begin{equation*}\label{new_B_sigma}
\begin{split}
B^{k,\sigma}_I&=\left\{\begin{array}{cl}
A_{\# I} &\quad \text{ if }\# I=1 \text{ or if there is 
$i\leq k$ such that } \sigma(J_i)\subseteq I;\\
M &\quad \text{ otherwise.}
\end{array}\right.
\end{split}
\end{equation*}
A new chain of decomposed $n$-fold vector bundles
\[ \E^0=\overline{\E}=:\E^0_\sigma\hookrightarrow \E^1_\sigma\hookrightarrow\ldots
\hookrightarrow \E^{\binom{n}{2}}_\sigma=\E^{\rm dec}=\E^{\binom{n}{2}},
\]
is defined by the families $\mathcal A^k_\sigma$, $k=0, \ldots, \binom{n}{2}$, as well as a family of monomorphisms
\[ \mathcal S^0_\sigma=\mathcal S^0=\Sigma\colon\overline{\E}\to\E, \quad \mathcal S^1_\sigma\colon \E^1_\sigma\to\E, \quad \ldots, \quad \mathcal S^{\binom{n}{2}}_\sigma=\mathcal S^{\binom{n}{2}}=\mathcal S\colon \E^{\rm dec} \to \E
\]
of $n$-fold vector bundles, where the first and last morphisms are $\Sigma$ and $\mathcal S$, respectively, the first by construction, and the latter since $\mathcal S$ does not depend on the choice of ordering of the subsets of $\nset$ with two elements, by Appendix \ref{dec_splittings}.

Consider  $\mathbf{x}=(x_J)_{J\subseteq I}\in \E^k_\sigma(I)\subseteq\E^{\rm dec}(I)$. Then there is an $m\in M$ such that for all $J\subseteq I$:
\[ x_J\left\{\begin{array}{ll}
=0^{A_{\# J}}_m & \text{ if } \# J\geq 2 \text{ and there exists no } i\in\{1,\ldots, k\} \text{ with } J_i\subseteq J\\
\in A_{\# J}(m) & \text{ otherwise}.  
\end{array}\right.
\]
For $\sigma\in S_n$ the image $\mathbf x'$ of this point $\mathbf x$ under $\Psi^{\rm dec} _\sigma(I)$ lives in $\E^{\rm dec}(\sigma(I))$ and its coordinates satisfy
\begin{equation*}
\begin{split}
x'_J&=\epsilon(\sigma\inv,J)x_{\sigma\inv(J)}\\
&\left\{\begin{array}{ll}
=0^{A_{\# J}}_m & \text{ if } \# J\geq 2 \text{ and there exists no } i\in\{1,\ldots, k\} \text{ with } J_i\subseteq \sigma\inv(J)\\
\in A_{\# J}(m) & \text{ otherwise}
\end{array}\right.
\end{split}
\end{equation*}
for all $J\subseteq \sigma(I)$.
Hence 
\begin{equation*}
\begin{split}
x'_J&=\epsilon(\sigma\inv,J)x_{\sigma\inv(J)}\\
&\left\{\begin{array}{ll}
=0^{A_{\# J}}_m & \text{ if } \# J\geq 2 \text{ and there exists no } i\in\{1,\ldots, k\} \text{ with } \sigma(J_i)\subseteq J\\
\in A_{\# J}(m) & \text{ otherwise}
\end{array}\right.
\end{split}
\end{equation*}
for all $J\subseteq \sigma(I)$, which shows that $\mathbf x'=\Psi^{\rm dec}_\sigma(\mathbf x)\in \E^k_\sigma(\sigma(I))$, and so that 
$\Psi^{\rm dec}_\sigma$ restricts to a natural isomorphism
\[ \Psi^{\rm dec}_\sigma\colon \E^k\to (\E^k_\sigma)^\sigma.
\]

Now since $\E^{\binom{n}{2}}_\sigma=\E^{\binom{n}{2}}=\E^{\rm dec}$ and $\S^{\binom{n}{2}}_\sigma=\mathcal S^{\binom{n}{2}}=\mathcal S$, it suffices to show inductively that for $k=0,\ldots, \binom{n}{2}$ and for $\sigma\in S_n$ the  diagram \[
\begin{tikzcd}
\E^k\ar[r,"\S^k"] \ar[d,"\Psi^{\rm dec}_\sigma"] & \E\ar[d,"\Psi_\sigma"]\\
\left(\E_\sigma^k\right)^\sigma \ar[r,"(\S^k_\sigma)^\sigma "] & \E^\sigma
\end{tikzcd} 
\]
of morphisms of $n$-fold
vector bundles commutes.
It is enough to check that at the top level, i.e.~to check that
\[\begin{tikzcd}
\E^k(\nset)\ar[r,"\S^k(\nset)"] \ar[d,"\Psi^{\rm dec}_\sigma(\nset)"] & \E(\nset)\ar[d,"\Psi_\sigma(\nset)"]\\
\E^k_\sigma(\nset) \ar[r,"\S^k_\sigma(\nset)"] & \E(\nset)
\end{tikzcd} 
\]
commutes.
For $k=0$ this is obviously the case since $\Sigma\colon \overline{\E}\to\E$ is a symmetric linear splitting of $\E$. 
Assume that the diagram commutes for a $k\in\left\{0, \ldots, \binom{n}{2}-1\right\}$, and choose as above $\mathbf{x}=(x_I)_{I\subseteq\nset}\in \E^{k+1}(\nset)$ over a base point $m\in M$. Define $\mathbf{y}:=(y_I)_{I\subseteq\nset}\in \E^k(\nset)\subseteq \E^{\rm dec}(\nset)$ and $\mathbf{z}:=(z_I)_{I\subseteq\nset}\in (\E^{\rm dec})^\nset_{J_{k+1}}$ as in \eqref{y} and \eqref{z}. As above  $\mathbf x':=\Psi^{\rm dec}_\sigma(\mathbf x)$ is then an element 
of $\E^{k+1}_\sigma(\nset)$ and the corresponding objects $\mathbf{y}'\in \E^k_\sigma(\nset)$ and $\mathbf{z}'\in (\E^{\rm dec})^\nset_{\sigma(J_{k+1})}$
are then $\Psi^{\rm dec}_\sigma(\mathbf y)$ and $\Psi^{\rm dec}_\sigma(\mathbf z)$, respectively.
Then
\begin{equation*}
\begin{split}
\Psi_{\sigma}(\mathcal S^{k+1}(\mathbf{x}))&=\Psi_\sigma\left(\mathcal S^k(\mathbf{y})\dvplus{}{\nset\setminus\{s\}}\left(
\nvbzero{\nset}{p_s\bigl(\mathcal S^k(\mathbf{y})\bigr)}\dvplus{}{\nset\setminus\{t\}}
\mathcal S^{J_{k+1}}(\mathbf{z})\right)\right)\\
&=\Psi_\sigma\left(\mathcal S^k(\mathbf{y})\right)
\dvplus{}{\nset\setminus\{\sigma(s)\}}
\left(\nvbzero{\nset}{p_{\sigma(s)}\bigl(\Psi_\sigma\left(\mathcal S^k(\mathbf{y})\right)\bigr)}\dvplus{}{\nset\setminus\{\sigma(t)\}}
\left(\Psi_\sigma\right)^{\nset}_{J_{k+1}}\left(\mathcal S^{J_{k+1}}(\mathbf{z})\right)\right)\\
&= \mathcal S^k_\sigma\left(\Psi^{\rm dec}_\sigma(\mathbf{y})\right)
\dvplus{}{\nset\setminus\{\sigma(s)\}}
\left(\nvbzero{\nset}{p_{\sigma(s)}\bigl(\mathcal S^k_\sigma\left(\Psi^{\rm dec}_\sigma(\mathbf{y})\right)\bigr)}\dvplus{}{\nset\setminus\{\sigma(t)\}}
\mathcal S^{\sigma(J_{k+1})}\left(\left(\Psi_\sigma^{\rm dec}\right)^{\nset}_{J_{k+1}}(\mathbf{z})\right)\right)\\
&=\mathcal S^{k+1}_\sigma\left(\Psi^{\rm dec}_\sigma(\mathbf{x})\right)
\end{split}
\end{equation*}
follows from the induction hypothesis and \eqref{sym_comp_cond}. \qedhere
\end{enumerate}
\end{proof}

The following is a more general version of the second statement of the previous theorem, which is the first step of the proof a symmetric decomposition of a symmetric $n$-fold vector bundle.

\begin{proposition}\label{core_split_to_dec}
Let $\E\colon \square^n\to \Man$ be a symmetric $n$-fold vector bundle, and consider an $l$-partition $\rho$ of $\nset$ for $1\leq l\leq n$. Consider for each partition $\rho'\in \{\sigma(\rho)\mid \sigma \in S_n\}$ a linear splitting $\Sigma^{\rho'}\colon \overline{\E^{\rho'}}\to \E^{\rho'}$. For each coarsement $\underline\rho$ in $(l-1)$ subsets of a partition $\rho'$ in the $S_n$-orbit of $\rho$ take a linear decomposition $\mathcal S^{\underline \rho}\colon (\E^{\rm dec})^{\underline\rho}\to \E^{\underline\rho}$ of the $(l-1)$-core $\E^{\underline\rho}$ of $\E$. Assume that:
\begin{enumerate}
\item For each $\sigma\in S_n$,
the diagram 
\[\begin{tikzcd}
\overline{\E^\rho}\ar[r,"\Sigma^\rho"] \ar[d,"\Psi^{\rm dec}_\sigma  "] & \E^\rho\ar[d,"\Phi_\sigma"]\\
\overline{\E^{\sigma(\rho)}}^\sigma \ar[r,"(\Sigma^{\sigma(\rho)})^\sigma"] & (\E^{\sigma(\rho)}) ^\sigma
\end{tikzcd} 
\]
 commutes.
\item For each $\sigma\in S_n$ and each coarsement $\underline \rho$ of $\rho$
\[  \Phi_\sigma\circ \S^{\underline\rho}
=\S^{\sigma(\underline\rho)}\circ \Psi^{\rm dec}_\sigma\colon (\E^{\rm dec})^{\underline\rho}\to \left(\E^{\sigma(\underline\rho)}\right)^\sigma.
\]
\item For two coarsements $\rho_1$ and $\rho_2$ of $\rho$ in $l-1$ subsets the two decompositions $\S^{\rho_1}$ and $\S^{\rho_2}$ are compatible as in \eqref{compatible_dec}.
\item $\Sigma^\rho$ is compatible with $\S^{\underline\rho}$ as in \eqref{comp_sigma_S} for each coarsement $\underline \rho$ of $\rho$ in $l-1$ subsets.
\end{enumerate}
Then Proposition \ref{thm_split_dec_n_general}  yields for each $\sigma\in S_n$ a decomposition 
\[ \mathcal S^{\sigma(\rho)}\colon (\E^{\rm dec})^{\sigma(\rho)}\to \E^{\sigma(\rho)}
\]
of the $l$-core $\E^{\sigma(\rho)}$ of $\E$, that restricts to $\Sigma^{\sigma(\rho)}$ and the decompositions of the appropriate $(l-1)$-cores of $\E$.
These decompositions satisfy
\[\Phi_\sigma(I)\circ \S^{\rho}(I)
=\S^{\sigma(\rho)}(\sigma(I))\circ\Psi^{\rm dec}_\sigma(I)\colon (\E^{\rm dec})^{\rho}(I)\to \E^{\sigma(\rho)}(\sigma(I))
\]
for all $\sigma\in S_n$ and all $I\in\Obj(\lozenge^\rho)$.
\end{proposition}

\begin{proof}
The proof is very similar to the proof of the second statement in Proposition \ref{thm_split_dec_n}. Write $\rho=\{I_1, \ldots, I_l\}$ and build all sets 
\[ J_1, \ldots, J_{\binom l 2} \in \Obj\left(\lozenge^\rho\right),
\]
each consisting in the union of exactly two elements of $\rho$. Then for each $\sigma\in S_n$ 
\[ J_1^\sigma:=\sigma(J_1), \ldots, J^\sigma_{\binom l 2}:=\sigma\left(J_{\binom l 2}\right)
\]
is such a list of subsets for the partition $\sigma(\rho)=\{\sigma(I_1), \ldots, \sigma(I_l)\}$ of $\nset$. Each such subset $J_i^\sigma$ defines a coarsement $\underline\rho$ of $\sigma(\rho)$, namely the one with $J_i^\sigma$ replacing the two elements of $\sigma(\rho)$ it is a union of. The corresponding splitting 
\[ \mathcal S^{\underline\rho}\colon \E^{{\rm dec}, \underline\rho} \to \E^{\underline\rho}
\]
which is chosen in the hypotheses of the theorem is named $\mathcal S^{J_i^\sigma}$ in the following.

As in Section \ref{dec_splittings} these lists yield for each $\sigma\in S_n$ and each $k=1,\ldots, \binom{l}{2}$ a family $\mathcal A^k_\sigma=(B^{k,\sigma}_I)_{I\in \Obj(\lozenge^{\sigma(\rho)})}$,  
\begin{equation*}\label{new_B_sigma}
\begin{split}
B^{k,\sigma}_I&=\left\{\begin{array}{cl}
A_{\# I} &\quad \text{ if }I\in\sigma(\rho) \text{ or if there is 
$i\leq k$ such that } \sigma(J_i)\subseteq I;\\
M &\quad \text{ otherwise.}
\end{array}\right.
\end{split}
\end{equation*}
A chain of decomposed $l$-fold vector bundles $\lozenge^{\sigma(\rho)}\to \Man$
\[ \overline{\E^{\sigma(\rho)}}=:\E^0_\sigma\hookrightarrow \E^1_\sigma\hookrightarrow\ldots
\hookrightarrow \E^{\binom{l}{2}}_\sigma:=\E^{{\rm dec}, \sigma(\rho)},
\]
is defined by the families $\mathcal A^k_\sigma$, $k=0, \ldots, \binom{l}{2}$, as well as a family of monomorphisms
\[ \mathcal S^0_\sigma=\Sigma^{\sigma(\rho)}\colon\overline{\E^{\sigma(\rho)}}\to\E^{\sigma(\rho)}, \quad \mathcal S^1_\sigma\colon \E^1_\sigma\to\E^{\sigma(\rho)}, \quad \ldots, \quad \mathcal S^{\binom{l}{2}}_\sigma=\mathcal S^{\sigma(\rho)}\colon \E^{{\rm dec}, \sigma(\rho)} \to \E^{\sigma(\rho)}
\]
of $l$-fold vector bundles $\lozenge^{\sigma(\rho)}\to \Man$, where the first and last morphisms are $\Sigma^{\sigma(\rho)}$ and $\mathcal S^{\sigma(\rho)}$, respectively.

Consider  \[\mathbf{x}=(x_J)_{\substack{J\subseteq I\\ J\in\Obj(\lozenge^{\rho})}}\in \E^k_{\id}(I)\]
for some $k\in\{1,\ldots,\binom l 2\}$ and for $\sigma=\id\in S_n$.
Then as before there is an $m\in M$ such that for all $J\subseteq I$ with $J\in\Obj(\lozenge^{\rho})$:
\[ x_J\left\{\begin{array}{ll}
=0^{A_{\# J}}_m & \text{ if } J\not\in\rho \text{ and there exists no } i\in\{1,\ldots, k\} \text{ with } J_i\subseteq J\\
\in A_{\# J}(m) & \text{ otherwise}.  
\end{array}\right.
\]
For $\sigma\in S_n$ the image $\mathbf x'$ of this point $\mathbf x$ under $\Psi^{\rm dec} _\sigma$ lives in $\E^{\rm dec, \sigma(\rho)}(\sigma(I))$ and its coordinates satisfy
\begin{equation*}
\begin{split}
x'_J&=\epsilon(\sigma\inv,J)x_{\sigma\inv(J)}\\
&\left\{\begin{array}{ll}
=0^{A_{\# J}}_m & \text{ if } \sigma\inv(J)\not \in\rho \text{ and there exists no } i\in\{1,\ldots, k\} \text{ with } J_i\subseteq \sigma\inv(J)\\
\in A_{\# J}(m) & \text{ otherwise}
\end{array}\right.
\end{split}
\end{equation*}
for all $J\subseteq \sigma(I)$ with $J\in\Obj(\lozenge^{\sigma(\rho)})$.
Hence 
\begin{equation*}
\begin{split}
x'_J&=\epsilon(\sigma\inv,J)x_{\sigma\inv(J)}\\
&\left\{\begin{array}{ll}
=0^{A_{\# J}}_m & \text{ if } J\not\in\sigma(\rho) \text{ and there exists no } i\in\{1,\ldots, k\} \text{ with } \sigma(J_i)\subseteq J\\
\in A_{\# J}(m) & \text{ otherwise}
\end{array}\right.
\end{split}
\end{equation*}
for all $J\subseteq \sigma(I)$ with $J\in\Obj(\lozenge^{\sigma(\rho)})$, which shows that $\mathbf x'=\Psi^{\rm dec}_\sigma(\mathbf x)\in \E^k_\sigma(\sigma(I))$, and so that 
$\Psi^{\rm dec}_\sigma$ restricts to a natural isomorphism
\[ \Psi^{\rm dec}_\sigma\colon \E^k\to (\E^k_\sigma)^\sigma
\]
of the functors $\E^k, (\E^k_\sigma)^\sigma\colon \lozenge^\rho\to \Man$.

Now since $\E^{\binom{l}{2}}=(\E^{\rm dec})^{\rho}$, $\E^{\binom{l}{2}}_\sigma=(\E^{\rm dec})^{\sigma(\rho)}$ and $\S^{\binom{l}{2}}_{\id}=\mathcal S^\rho$ as well as $\mathcal S^{\binom{l}{2}}_\sigma=\mathcal S^{\sigma(\rho)}$, it suffices to show inductively that for $k=0,\ldots, \binom{l}{2}$ and for $\sigma\in S_n$ the  diagram \[
\begin{tikzcd}
\E^k_{\id}\ar[r,"\S^k_{\id}"] \ar[d,"\Psi^{\rm dec}_\sigma"] & \E^\rho\ar[d,"\Phi^\rho_\sigma"]\\
\left(\E_\sigma^k\right)^\sigma \ar[r,"(\S^k_\sigma)^\sigma "] & (\E^{\sigma(\rho)})^\sigma
\end{tikzcd} 
\]
of morphisms of $l$-fold
vector bundles commutes.
It is enough to check this at the top level, i.e.~to check that
\[\begin{tikzcd}
\E^k_{\id}(\nset)\ar[r,"\S^k_{\id}(\nset)"] \ar[d,"\Psi^{\rm dec}_\sigma(\nset)"] & \E^\rho(\nset)\ar[d,"\Phi^\rho_\sigma(\nset)"]\\
\E^k_\sigma(\nset) \ar[r,"\S^k_\sigma(\nset)"] & \E^{\sigma(\rho)}(\nset)
\end{tikzcd} 
\]
commutes. 
For $k=0$ this is given by the first assumption in the theorem.
Assume that the diagram commutes for a $k\in\left\{0, \ldots, \binom{l}{2}-1\right\}$, and choose \[\mathbf{x}=(x_I)_{ I\in\operatorname{Obj}(\lozenge^\rho)}\in \E^{k+1}_{\id}(\nset)\] over a base point $m\in M$. Define $\mathbf{y}:=(y_I)_{I\in\operatorname{Obj}(\lozenge^\rho)}\in \E^k_{\id}(\nset)\subseteq \E^{\rm dec, \rho}(\nset)$
by
\[y_I=\left\{\begin{array}{cl}
x_I &\text{ if  $I\in \rho$ or there is $i\leq k$ such that 
$J_i\subseteq I$},\\
 0^{A_{\# I}}_m& \text{  otherwise.} \end{array}\right.
\]
 and $\mathbf{z}:=(z_I)_{I\in\operatorname{Obj}(\lozenge^\rho)}\in \E^{{\rm dec}, \rho_{J_{k+1}}}(\nset)$
 by
\[z_I=\left\{\begin{array}{cl}y_I &\text{  whenever }
I\subseteq\nset\setminus J_{k+1}, \\
x_I& \text{ whenever $J_{k+1}\subseteq I$ 
and there is no $i\leq k$ with $J_i\subseteq I$,}\\
 0^{A_{\# I}}_m& \text{ otherwise. }
 \end{array}\right.
\]
Here, $\rho_{J_{k+1}}$ is the $(l-1)$-coarsement of $\rho$ obtained by replacing by $J_{k+1}$ the two sets $J_{k+1}$ is the union of.
  Then writing $J_{k+1}=I_s\cup I_t$ with $1\leq s<t\leq l$, it is easy to check that 
\[
\mathbf{x}=\mathbf{y}\dvplus{}{\nset\setminus I_s}\left(
\nvbzero{\nset}{p^{\nset}_{\nset\setminus I_s}(\mathbf{y})}\dvplus{}{\nset\setminus I_t}
\mathbf{z}\right)=\mathbf{y}\dvplus{}{\nset\setminus I_t}\left(
\nvbzero{\nset}{p^{\nset}_{\nset\setminus I_t}(\mathbf{y})}\dvplus{}{\nset\setminus I_s}
\mathbf{z}\right)\,,
\] 
and 
$\mathcal S^{k+1}_{\id}(\mathbf{x})$ is then defined by
\[\mathcal S^{k+1}_{\id}(\mathbf{x})=\mathcal S^k_{\id}(\mathbf{y})\dvplus{}{\nset\setminus I_s}\left(
\nvbzero{\nset}{p^{\nset}_{\nset\setminus I_s}\bigl(\mathcal S^k_{\id}(\mathbf{y})\bigr)}\dvplus{}{\nset\setminus I_t}
\mathcal S_{\id}^{J^{\id}_{k+1}}(\mathbf{z})\right)
\]
 As above  $\mathbf x':=\Psi^{\rm dec}_\sigma(\mathbf x)$ is then an element 
of $\E^{k+1}_\sigma(\nset)$ and the corresponding objects $\mathbf{y}'\in \E^k_\sigma(\nset)$ and $\mathbf{z}'\in {\E^{{\rm dec}, \rho_{J^\sigma_{k+1}}}}(\nset)$
are then $\Psi^{\rm dec}_\sigma(\mathbf y)$ and $\Psi^{{\rm dec}}_\sigma(\mathbf z)$, respectively.
Then
\begin{equation*}
\begin{split}
\Phi_{\sigma}\left(\mathcal S^{k+1}_{\id}(\mathbf{x})\right)&=\Phi_\sigma\left(\mathcal S^k_{\id}(\mathbf{y})\dvplus{}{\nset\setminus I_s}\left(
\nvbzero{\nset}{p^{\nset}_{\nset\setminus I_s}\bigl(\mathcal S^k_{\id}(\mathbf{y})\bigr)}\dvplus{}{\nset\setminus I_t}
\mathcal S_{\id}^{J^{\id}_{k+1}}(\mathbf{z})\right)\right)\\
&=\Phi_\sigma\left(\mathcal S_{\id}^k(\mathbf{y})\right)
\dvplus{}{\nset\setminus\sigma(I_s)}
\left(\nvbzero{\nset}{p^{\nset}_{\nset\setminus\sigma(I_s)}\bigl(\Phi_\sigma\left(\mathcal S_{\id}^k(\mathbf{y})\right)\bigr)}\dvplus{}{\nset\setminus\sigma(I_t)}
\Phi_\sigma\left(\mathcal S^{J^{\id}_{k+1}}(\mathbf{z})\right)\right)\\
&= \mathcal S^k_\sigma\left(\Psi^{\rm dec}_\sigma(\mathbf{y})\right)
\dvplus{}{\nset\setminus\sigma(I_s)}
\left(\nvbzero{\nset}{p^{\nset}_{\nset\setminus\sigma(I_s)}\bigl(\mathcal S^k_\sigma\left(\Psi^{\rm dec}_\sigma(\mathbf{y})\right)\bigr)}\dvplus{}{\nset\setminus\sigma(I_t)}
\mathcal S^{J^\sigma_{k+1}}\left(\Psi_\sigma^{\rm dec}(\mathbf{z})\right)\right)\\
&=\mathcal S^{k+1}_\sigma\left(\Psi^{\rm dec}_\sigma(\mathbf{x})\right)
\end{split}
\end{equation*}
follows from the induction hypothesis and \eqref{sym_comp_cond}. \qedhere
\end{proof}

\begin{proposition}\label{missing_step_sym}
Let $\E\colon\square^n\to \Man$ be a symmetric $n$-fold vector bundle
and consider an $l$-partition $\rho$ of $\nset$.
Assume that 
\begin{enumerate}
\item For each $\rho'$ in the $S_n$-orbit of $\rho$ and for each $(l-1)$-coarsement $\underline{\rho}$ of $\rho'$ there is
a decomposition $\mathcal S^{\underline{\rho}}\colon \E^{{\rm dec}, {\underline{\rho}}}\to \E^{\underline{\rho}}$.
\item For each pair $\rho_1, \rho_2$
of $(l-1)$-coarsements of a partition $\rho'$ in the $S_n$-orbit of $\rho$, the decompositions
$\mathcal S^{\rho_1}$ and $\mathcal S^{\rho_2}$ are compatible as in \eqref{compatible_dec}.
\item For each $\sigma\in S_n$ and each coarsement $\underline\rho$ of $\rho$
\[  \Phi_\sigma\circ \S^{\underline\rho}
=(\S^{\sigma(\underline\rho)})^\sigma\circ\Psi^{\rm dec}_\sigma\colon (\E^{\rm dec})^{\underline\rho}\to \E^{\sigma(\underline\rho)}.
\]
\end{enumerate}
Then for each $\rho'$ in the $S_n$-orbit of $\rho$ there exists a linear splitting 
\[ \Sigma^{\rho'}\colon \overline{\E^{\rho'}}\to \E^{\rho'}
\]
of $\E^{\rho'}$ that is compatible as in \eqref{comp_sigma_S} with all decompositions $\mathcal S^{\underline{\rho}}$ for all coarsements $\underline{\rho}$ of $\rho'$, and such that 
\begin{equation*}
  \Phi_\sigma\circ \Sigma^{\rho}=(\Sigma^{\sigma(\rho)})^\sigma\circ \Psi_\sigma.
  \end{equation*}
  for all $\sigma\in S_n$.
\end{proposition}

\begin{proof}
Recall that the following claim is proved in \cite[Theorem 3.5]{HeJo20}, see also 
Section \ref{fix_of_3.6}:
Given an $n$-fold vector bundle $\E$, 
with linear splittings
$\Sigma_I$ of $\E^{I,\emptyset}$ for all $I\subsetneq \nset$ with $\# I=n-1$, such that
$\Sigma_{I_1}(J)=\Sigma_{I_2}(J)$ whenever $J\subseteq I_1\cap I_2$, then there exists a  linear
  splitting $\Sigma$ of $\E$ with $\Sigma(J)=\Sigma_{I}(J)$
  whenever $J\subseteq I\subseteq\nset$.
  
  \medskip

For a partition $\rho$ of $\nset$ each $I\in\operatorname{Obj}(\lozenge^\rho)$ defines a face $\E^{\rho, I}$ of $\E^\rho$. It is the restriction of $\E$ to 
the full subcategory of $\lozenge^\rho$ with objects contained in $I$. 

\medskip

Choose an $l$-partition $\rho$ of $\nset$.   Take an $l$-partition $\rho'=\{I_1,\ldots, I_l\}$ of $\nset$ in the $S_n$-orbit of $\rho$ and
choose an $(l-1)$-coarsement $\underline\rho$ of $\rho'$. Then there exist $1\leq s<t\leq l$ such that 
\[\underline \rho=\{I_s\cup I_t, I_1, \ldots, I_{s-1}, I_{s+1}, \ldots, I_{t-1}, I_{t+1}, \ldots, I_l\}=:\rho_{st}.\]
The decomposition $\mathcal S^{\underline\rho}$ of $\E^{\underline\rho}$ defines then a decomposition of the face  $\E_{\rho', \nset\setminus(I_s\cup I_t)}$ of $\E^{\rho'}$ since $\E_{\rho', \nset\setminus(I_s\cup I_t)}=\E_{\underline\rho, \nset\setminus(I_s\cup I_t)}$ is also a face of $\E^{\underline{\rho}}$.
Denote by $\Sigma_{\rho',\nset\setminus (I_s\cup I_t)}$ the induced linear splitting of $\E_{\rho',\nset\setminus (I_s\cup I_t)}$.

Fix $s\in \{1,\ldots, l\}$. Then as above the $(l-2)$-fold vector bundles $\E_{\rho',\nset\setminus (I_s\cup I_t)}$, for $t\in \lset\setminus \{s\}$,  are all sides of the $(l-1)$-fold 
vector bundle $\E_{\rho', \nset\setminus I_s}$. 
Choose $t,r\in \lset\setminus\{s\}$. Then \[\rho_{st}\sqcap\rho_{sr}=\{I_s\cup I_r\cup I_t\}\cup \{I_x\mid x\in\lset\setminus\{r,s,t\}\}.\]
Since $\mathcal S^{\rho_{st}}$ and $\mathcal S^{\rho_{sr}}$ coincide on the common core $\E^{\rho_{st}\sqcap \rho_{sr}}$ of $\E^{\rho_{st}}$ and $\E^{\rho_{sr}}$ by \eqref{compatible_dec}, 
they coincide in particular on its face $\E_{\rho',\nset\setminus (I_s\cup I_r\cup I_t)}$.
Hence $\mathcal S^{\rho_{st}}\an{\nset\setminus(I_s\cup I_r\cup I_t)}=\mathcal S^{\rho_{sr}}\an{\nset\setminus(I_s\cup I_r\cup I_t)}$  and the splittings $\Sigma_{\rho', \nset\setminus (I_s\cup I_t)}$ and $\Sigma_{\rho', \nset\setminus (I_s\cup I_r)}$ satisfy consequently
\[ \Sigma_{\rho', \nset\setminus (I_s\cup I_t)}(I)=\Sigma_{\rho', \nset\setminus (I_s\cup I_r)}(I)
\]
for all $I\in\operatorname{Obj}(\lozenge^{\rho'})$ with $I \subseteq \nset\setminus(I_s\cup I_t\cup I_r)$. 

By the claim proved in \cite[Theorem 3.5]{HeJo20} there is consequently a linear splitting 
$\Sigma_{\rho',\nset\setminus I_s}$
of $\E_{\rho', \nset\setminus I_s}$, such that
\begin{equation}\label{comp_obtained_splittings}
 \Sigma_{\rho', \nset\setminus I_s}\an{\nset\setminus (I_s\cup I_t)}=\Sigma_{\rho', \nset\setminus (I_s\cup I_t)}
\end{equation}
for all $t\in\lset\setminus \{s\}$. 

Apply this to each $l$-partition $\rho'$ in the $S_n$-orbit of $\rho$ and obtain linear splittings of the $(l-1)$-sides of $\E^{\rho'}$, that are compatible with the given linear splittings of its $(l-2)$-sides.
By the condition (2) in the hypotheses of the theorem, 
  \begin{equation}\label{condition_comp_sym_level-2}
  \Phi_\sigma(J)\circ \Sigma_{\rho', \nset\setminus I}(J)=\Sigma_{\sigma(\rho'),\nset\setminus\sigma(I)}(\sigma(J))\circ \Psi_\sigma(J)
  \end{equation}
  follows immediately for all $\sigma\in S_n$ and for all $J\in\operatorname{Obj}(\lozenge^{\rho'})$ with $J\subsetneq \nset\setminus I$.

\medskip

Now choose $\rho'$ in the $S_n$-orbit of $\rho$ and for $I\in\rho'$ define
\[ \tilde\Sigma_{\rho', \nset\setminus I}\colon \overline{\E^{\rho'}}^{\nset\setminus I}\to \E^{\rho', \nset\setminus I},
\]
by
\[ \tilde\Sigma_{\rho', \nset\setminus I}(J)=\Sigma_{\rho', \nset\setminus I}(J)
\]
for $J\in\operatorname{Obj}(\lozenge^{\rho'})$ with $J\subsetneq \nset\setminus I$ and 
\[ \tilde\Sigma_{\rho',\nset\setminus I}(\nset\setminus I)= \frac{1}{n!}\cdot_{K}\sum^{K}_{\sigma\in S_n} \left(\Phi_{\sigma\inv}\circ \Sigma_{\sigma(\rho'),\nset\setminus\sigma(I)}\circ \Psi_{\sigma}\right)(\nset\setminus I)
\]
for any $K\in \rho'\setminus\{I\}$, where the scalar multiplication $\cdot_K$ and the addition $\sum^K$ are the scalar multiplication and the addition in the vector bundle
\[ p^{\nset\setminus I}_{\nset\setminus (I\cup K)}\colon \E^{\rho'}(\nset\setminus I)\to \E^{\rho'}(\nset\setminus (I\cup K)).
\] 
This sum is well-defined since for all such $K$ and for all $\sigma\in S_n$
\begin{equation*}
\begin{split}
&p^{\nset\setminus I}_{\nset\setminus (I\cup K)}\circ \Phi_{\sigma\inv}(\nset\setminus\sigma(I))\circ \Sigma_{\sigma(\rho'),\nset\setminus\sigma(I)}(\nset\setminus \sigma(I))\circ \Psi_{\sigma}(\nset\setminus I)\\
&\,\,=
\Phi_{\sigma\inv}(\nset\setminus\sigma(I\cup K))\circ p^{\nset\setminus\sigma(I)}_{\nset\setminus\sigma(I\cup K)} \circ \Sigma_{\sigma(\rho'), \nset\setminus\sigma(I)}(\nset\setminus\sigma(I))\circ \Psi_{\sigma}(\nset\setminus I)\\
&\overset{\eqref{comp_obtained_splittings}}{=}\Phi_{\sigma\inv}(\nset\setminus\sigma(I\cup K)) \circ \Sigma_{\sigma(\rho'), \nset\setminus\sigma(I)}(\nset\setminus\sigma(I\cup K))\circ p^{\nset\setminus\sigma(I)}_{\nset\setminus\sigma(I\cup K)}\circ \Psi_{\sigma}(\nset\setminus I)\\
&\overset{\eqref{condition_comp_sym_level-2} }{=} \Sigma_{\rho', \nset\setminus I}(\nset\setminus(I\cup K))\circ \Psi_{\sigma\inv}(\nset\setminus\sigma(I\cup K)) \circ  \Psi_{\sigma}(\nset\setminus (I\cup K))\circ p^{\nset\setminus I}_{\nset\setminus (I\cup K)}\\
&\,\,= \Sigma_{\rho', \nset\setminus I}(\nset\setminus(I\cup K))\circ p^{\nset\setminus I}_{\nset\setminus (I\cup K)},
\end{split}
\end{equation*}
which does not depend on $\sigma$.

For $\mathbf x\in \overline{\E^{\rho'}}^{\nset\setminus I}(\nset\setminus I)$ and $K\neq J\in\rho'\setminus\{I\}$
 \begin{equation*}
 \begin{split}
 &n!\cdot_{J}\sum^{K}_{\sigma\in S_n}(\Phi_{\sigma\inv}\circ \tilde\Sigma_{\sigma(\rho), \sigma(I)}\circ \Psi_{\sigma})(\mathbf x)\\
 &=\underset{n!-\text{times}}{\underbrace{\left(\sum^{K}_{\sigma\in S_n}(\Phi_{\sigma\inv}\circ \tilde\Sigma_{\sigma(\rho),\sigma(I)}\circ \Psi_{\sigma})(\mathbf x)\right)\dvplus{}{J}\ldots\dvplus{}{J}\left(\sum^{K}_{\sigma\in S_n}(\Phi_{\sigma\inv}\circ \tilde\Sigma_{\sigma(\rho),\sigma(I)}\circ \Psi_{\sigma})(\mathbf x)\right)}}.
 \\
 \end{split}
 \end{equation*}
 Since all summands can be added with each other over $\nset\setminus K$ and over $\nset \setminus J$,  a reordering of the summands over $\nset\setminus  K$ in each sum and the interchange law\footnote{For $n=2$ this is $\left(\mathbf x_{\id}\dvplus{}{K}\mathbf x_{(12)}\right)\dvplus{}{J}\left(\mathbf x_{(12)}\dvplus{}{K}\mathbf x_{\id}\right)=\left(\mathbf x_{\id}\dvplus{}{J}\mathbf x_{(12)}\right)\dvplus{}{K}\left(\mathbf x_{(12)}\dvplus{}{J}\mathbf x_{\id}\right)$.
}  yield
 \begin{equation*}
 \begin{split}
 &\underset{n!-\text{times}}{\underbrace{\left(\sum^{J}_{\sigma\in S_n}(\Phi_{\sigma\inv}\circ \tilde\Sigma_{\sigma(\rho),\sigma(I)}\circ \Psi_{\sigma})(\mathbf x)\right)\dvplus{}{K}\ldots\dvplus{}{K}\left(\sum^{J}_{\sigma\in S_n}(\Phi_{\sigma\inv}\circ \tilde\Sigma_{\sigma(\rho),\sigma(I)}\circ \Psi_{\sigma})(\mathbf x)\right)}}\\
 &=n!\cdot_{K}\left(\sum^{J}_{\sigma\in S_n}(\Phi_{\sigma\inv}\circ \tilde\Sigma_{\sigma(\rho),\sigma(I)}\circ \Psi_{\sigma})(\mathbf x)\right).
 \end{split}
 \end{equation*}
 Thus by multiplying both sides with $\frac{1}{n!}$ over $\nset\setminus K$ and with 
 $\frac{1}{n!}$ over $\nset\setminus J$:
  \begin{equation*}
 \begin{split}
 \frac{1}{n!}\cdot_{K}\left(\sum^K_{\sigma\in S_n} (\Phi_{\sigma\inv}\circ 
 \tilde\Sigma_{\sigma(\rho),\sigma(I)}\circ \Psi_{\sigma})(\mathbf x)\right)=\frac{1}{n!}
 \cdot_{J}\left(\sum^J_{\sigma\in S_n}(\Phi_{\sigma\inv}\circ \tilde\Sigma_{\sigma(\rho),\sigma(I)}\circ \Psi_{\sigma})(\mathbf x)\right).
 \end{split}
 \end{equation*}

Hence the definition of $\tilde\Sigma_{\rho', \nset\setminus I}$ does not depend on the choice of $K$, and so for each $K\in \rho'\setminus \{I\}$ the map
\[ \tilde\Sigma_{\rho', \nset\setminus I}(\nset\setminus I)\colon \overline{\E^{\rho'}}(\nset\setminus I)\to \E^{\rho'}(\nset\setminus I)
\]
is a vector bundle morphism over 
\[\tilde\Sigma_{\rho', \nset\setminus I}(\nset\setminus(I\cup K))\colon \overline{\E^{\rho'}}(\nset\setminus (I\cup K))\to \E^{\rho'}(\nset\setminus (I\cup K)).
\]
Assume that 
\[ \tilde\Sigma_{\rho', \nset\setminus I}(\nset\setminus I)\left((a_J)_{J\in \rho'\setminus \{I\}}\right)=\tilde\Sigma_{\rho', \nset\setminus I}(\nset\setminus I)\left((b_J)_{J\in \rho'\setminus \{I\}}\right)
\]
for $\left((a_J)_{J\in \rho'\setminus \{I\}}\right)$ and $\left((b_J)_{J\in \rho'\setminus \{I\}}\right)$ in 
$\overline{\E^{\rho'}}(\nset\setminus I)$. Then for each $K\in\rho'\setminus \{I\}$
\begin{equation*}
\begin{split}
\Sigma_{\rho', \nset\setminus (I\cup K)}(\nset\setminus(I\cup K))\left((a_J)_{J\in \rho'\setminus \{I,K\}}\right)&
=\tilde\Sigma_{\rho', \nset\setminus I}(\nset\setminus(I\cup K))\left((a_J)_{J\in \rho'\setminus \{I,K\}}\right)\\
&=p^{\nset\setminus I}_{\nset\setminus (I\cup K)}\left(\tilde\Sigma_{\rho', \nset\setminus I}(\nset\setminus I)\left((a_J)_{J\in \rho'\setminus \{I\}}\right)\right)\\
&=p^{\nset\setminus I}_{\nset\setminus (I\cup K)}\left(\tilde\Sigma_{\rho', \nset\setminus I}(\nset\setminus I)\left((b_J)_{J\in \rho'\setminus \{I\}}\right)\right)\\
&=\tilde\Sigma_{\rho', \nset\setminus I}(\nset\setminus(I\cup K))\left((b_J)_{J\in \rho'\setminus \{I,K\}}\right)\\
&=\Sigma_{\rho', \nset\setminus (I\cup K)}(\nset\setminus(I\cup K))\left((b_J)_{J\in \rho'\setminus \{I,K\}}\right).
\end{split}
\end{equation*}
Since $\Sigma_{\rho', \nset\setminus (I\cup K)}$ is a linear splitting of $\E_{\rho', \nset\setminus(I\cup K)}$, hence a monomorphism of vector bundles, 
\[(a_J)_{J\in \rho'\setminus \{I,K\}}=(b_J)_{J\in \rho'\setminus \{I,K\}}.
\]
But since $K\in\rho'\setminus \{I\}$ was arbitrary, this shows 
\[(a_J)_{J\in \rho'\setminus \{I\}}=(b_J)_{J\in \rho'\setminus \{I\}}.
\]
This shows that $\tilde\Sigma_{\rho', \nset\setminus I}$ is a monomorphism of $(l-1)$-fold vector bundles, hence a linear splitting of the $(l-1)$-fold vector bundle
$\mathbb E^{\rho',I}$.

\bigskip

 By construction the obtained collection of linear splittings of the sides of the $l$-fold vector bundles $\E^{\rho'}$  for all $\rho'$ in the $S_n$-orbit of $\rho$ 
 \begin{equation*}
  \Phi_\sigma(J)\circ \tilde\Sigma_{\rho', \nset\setminus I}(J)=\tilde\Sigma_{\sigma(\rho'),\nset\setminus \sigma(I)}(\sigma(J))\circ \Psi_\sigma(J)
  \end{equation*}
  for all $I\in \rho'$, $J\in \operatorname{Obj}(\lozenge^{\rho'})$ with $J\subseteq \nset\setminus I$ and all $\sigma\in S_n$.
In addition for $I_1\neq I_2\in\rho'$ and $J\in\operatorname{Obj}(\lozenge^{\rho'}$ such that
\[ J\subseteq (\nset\setminus I_1)\cap(\nset\setminus I_2)=\nset\setminus(I_1\cup I_2),
\]
\[ \tilde\Sigma_{\rho', \nset\setminus I_1}(J)=\Sigma_{\rho', \nset\setminus(I_1\cup I_2)}(J)=\tilde\Sigma_{\rho', \nset\setminus I_2}(J).
\]
As a consequence, there exists a linear splitting 
\[ \Sigma^{\rho'}\colon \overline{\E^{\rho'}}\to\E^{\rho'}
\]
of $\E^{\rho'}$ with 
\[ \Sigma^{\rho'}\an{\nset\setminus I}=\tilde\Sigma_{\rho', \nset\setminus I}
\]
for all $I \in \rho'$. 
As above, an averaging of the obtained family of splittings (for all $\rho'$ in the $S_n$-orbit of $\rho$) shows that these linear splittings can be chosen such that 
 \begin{equation}\label{condition_comp_sym_top}
  \Phi_\sigma(J)\circ \Sigma^{\rho'}(J)=\Sigma^{\sigma(\rho')}(\sigma(J))\circ \Psi_\sigma(J)
  \end{equation}
  for all $J\in \operatorname{Obj}(\lozenge^{\rho'})$ and all $\sigma\in S_n$.

The compatibility in \eqref{comp_sigma_S} with the decompositions of the highest order cores is immediate by construction:
For $I_s,I_t\in \rho'$ as above and $J\in\operatorname{Obj}(\lozenge^{\rho_{st}})\subseteq \operatorname{Obj}(\lozenge^{\rho'})$
\[ \Sigma^{\rho'}(J)=\tilde\Sigma_{\rho, \nset\setminus I_s}(J)=\Sigma_{\rho', \nset\setminus(I_s\cup I_t)}(J)=\mathcal S^{\rho_{st}}\circ\iota(J).
\qedhere
\]
\end{proof} 

\subsection{Proof of Theorem \ref{dec_sym_nvb_thm}}

Finally Theorem \ref{dec_sym_nvb_thm} can be proved. This is the subject of the remainder of this section.
	For simplicity, the action $\Psi^{\rm dec}$ of $S_n$ on the decomposed $n$-fold vector bundle $\E^{\rm dec}$ is simply written $\Psi$ in this proof.
	First consider all $2$-cores of $\E$. As explained in Proposition \ref{l_cores_as_partitions}, these are indexed by all possible partitions of $\nset$ in two subsets.
	Consider such a partition $\rho=\{I, \nset\setminus I\}$ for $\emptyset\neq I\subsetneq \nset$
and choose a linear decomposition
 % https://q.uiver.app/?q=WzAsOCxbMCwwLCJFX2xcXHRpbWVzIEVfe24tbH1cXHRpbWVzIEVfe1xcbnNldH0iXSxbMCwyLCJFX2wiXSxbMSwxLCJFX3tuLWx9Il0sWzEsMywiTSJdLFsyLDAsIlxcRV57XFxyaG9fbH0iXSxbMiwyLCJFX2wiXSxbMywxLCJFX3tuLWx9Il0sWzMsMywiTSJdLFswLDJdLFswLDFdLFsxLDNdLFsyLDNdLFswLDQsIlxcbWF0aGNhbCBTXntcXHJob19sfSJdLFs0LDZdLFs0LDVdLFs1LDddLFs2LDddLFsyLDZdLFszLDddLFsxLDVdXQ==
\[\begin{tikzcd}
	{E_I^I\times_M E^{\nset\setminus I}_{\nset\setminus I}\times_M E_{\nset}^{\nset}} && {\E^{\rho}(\nset)} \\
	& {E_{\nset\setminus I}^{\nset\setminus I}} && {E_{\nset\setminus I}^{\nset\setminus I}} \\
	{E_I^I} && {E_I^I} \\
	& M && M
	\arrow[from=1-1, to=2-2]
	\arrow[from=1-1, to=3-1]
	\arrow[from=3-1, to=4-2]
	\arrow[ from=2-2, to=4-2]
	\arrow["{\widetilde{\mathcal S^{\rho}}}", from=1-1, to=1-3]
	\arrow[from=1-3, to=2-4]
	\arrow[ from=1-3, to=3-3]
	\arrow[from=3-3, to=4-4]
	\arrow[from=2-4, to=4-4]
	\arrow["{\id\qquad}", from=2-2, to=2-4]
	\arrow["{\id}", from=4-2, to=4-4]
	\arrow["{\id }", from=3-1, to=3-3]
\end{tikzcd}\]
of the $2$-core $\E^{\rho}$, which is a double vector bundle with sides $E_I^I$ and $E^{\nset\setminus I}_{\nset\setminus I}$ and with core $ E_{\nset}^{\nset}$.
After such a decomposition $\mathcal S^{\rho'} $  of $\E^{\rho'} $ has been chosen for each partition $\rho'$ of $\nset$ in two sets, choose again a fixed such partition $\rho=\{I, \nset\setminus I\}$ for $\emptyset\neq I\subsetneq \nset$ and define
\[ \mathcal S^\rho\colon \E^{{\rm dec}, \rho}\to \E^\rho
\]
by
\begin{equation}\label{level_2_sym} \mathcal S^{\rho}(\nset):=\frac{1}{n! }\cdot_{E_I^I}\sum_{\sigma\in S_{n}}^{E_I^I}\left(\Phi_{\sigma\inv}\circ \widetilde{\S^{\sigma(\rho)}}\circ\Psi_\sigma\right)(\nset),
\end{equation}
as well as 
\[ \mathcal S^\rho(I)=\id_{E^{I}_{I}}\qquad \mathcal S^\rho(\nset\setminus I)=\id_{E^{\nset\setminus I}_{\nset\setminus I}}
\]
and $\mathcal S^\rho(\emptyset)=\id_M$, respectively.
\eqref{level_2_sym} is well-defined since for all $\sigma\in S_{n}$, 
\begin{equation*}
\begin{split}
 p^{\nset}_{I}\circ \Phi_{\sigma\inv}(\nset)\circ \widetilde{\S^{\sigma(\rho)}}(\nset)\circ \Psi_\sigma(\nset)
&=\Phi_{\sigma\inv}(\sigma(I))\circ p^{\nset}_{\sigma(I)}\circ \widetilde{\S^{\sigma(\rho)}}(\nset)\circ \Psi_\sigma(\nset)\\
&=\Phi_{\sigma\inv}(\sigma(I))\circ  \id_{E_{\sigma(I)}^{\sigma(I)}}\circ p^{\nset}_{\sigma(I)}\circ\Psi_\sigma(\nset)\\
&=\Phi_{\sigma\inv}(\sigma(I))\circ\Psi_\sigma(I)\circ p^{\nset}_{I}\\
&=\left(\epsilon(\sigma, I)\cdot\id_{E_{I}^{I}}\right)\circ \left(\epsilon(\sigma, I)\cdot\id_{E_{I}^{I}}\right)\circ p^{\nset}_{I}=p_{I}^{\nset},
\end{split}
\end{equation*}
since $\Phi_{\sigma\inv}(\sigma(I))$ and $\Psi_\sigma(I)$, being the restrictions to $\E^{\sigma(\rho)}(\sigma(I))=E_{\sigma(I)}^{\sigma(I)}$ and $\E^{{\rm dec}, \rho}(I)=E_{I}^{I}$ of $\Phi_{\sigma\inv}(\sigma(I))$ and $\Psi_{\sigma}(I)$, must equal $\epsilon(\sigma\inv, \sigma(I))\cdot\id_{E_{\sigma(I)}^{\sigma(I)}}=\epsilon(\sigma, I)\cdot\id_{E_{I}^{I}}$ and $\epsilon(\sigma, I)\cdot\id_{E_{I}^{I}}$, respectively.
Further, by the interchange law, \eqref{level_2_sym} does not change if $I$ is replaced by $\nset\setminus I$.

By construction, \begin{equation}\label{gen_eq_level_2}
\Phi_\sigma\circ \mathcal S^\rho=\mathcal S^{\sigma(\rho)}\circ \Psi_\sigma\colon \E^{{\rm dec}, \rho}\rightarrow (\E^{\sigma(\rho)})^\sigma
\end{equation}
for all $\sigma\in S_n$ and all partitions $\rho$ of $\nset$ in two subsets.
All the $2$-partitions of $\nset$ have the unique common $1$-coarsement $\underline\rho=\{\nset\}$. Since the restriction to the ultracore of each decomposition of a $2$-core is the identity $\id_{\E^{\nset}_{\nset}}$, all the obtained $2$-core decompositions are compatible as in \eqref{compatible_dec}.

By Proposition \ref{missing_step_sym} there exist linear splittings of the $3$-cores of $\E$, which are symmetrically compatible and also compatible with all previously chosen decompositions of the $2$-cores. By Proposition \ref{core_split_to_dec} there exist then symmetrically compatible decompositions of all the $3$-cores of $\E$, that are compatible as in \eqref{compatible_dec}. A recursive use of Propositions \ref{missing_step_sym}
and \ref{core_split_to_dec} yields then for each $l=3,\ldots,n$ symmetrically compatible decompositions of all $l$-cores, which are compatible as in \eqref{compatible_dec}. The claim is then proved at the last step $l=n$.

\def\cprime{$'$} \def\polhk#1{\setbox0=\hbox{#1}{\ooalign{\hidewidth
  \lower1.5ex\hbox{`}\hidewidth\crcr\unhbox0}}} \def\cprime{$'$}
  \def\cprime{$'$}

\end{document}